\documentclass[11pt]{amsart}

\usepackage{microtype, textcomp,fbb}
\usepackage{
	amsmath,  amssymb,  amsthm,   amscd,
	gensymb,  graphicx, etoolbox, 
	booktabs, stackrel, mathtools  
}
\usepackage{ragged2e}
\usepackage{float}
\usepackage[english]{babel}
\usepackage{csquotes}
\usepackage{enumerate}
\usepackage{geometry}
\usepackage{multicol}
\usepackage{tasks}
\usepackage{lipsum}
\usepackage{fullpage}
\usepackage{wrapfig}
\usepackage{animate}
\usepackage{caption}
\usepackage{subcaption}
\usepackage{setspace}
\usepackage{tikz}
\usetikzlibrary{arrows.meta}
\usepackage{breqn}
\usepackage{csquotes}
\usepackage{mwe}
\definecolor{darkblue}{rgb}{0,0,0.7}
\definecolor{darkred}{rgb}{0.7,0,0}
\usepackage[colorlinks=true, linkcolor=darkred, citecolor=darkblue, urlcolor=blue, breaklinks=true]{hyperref}
\usepackage{cleveref}
\usepackage{chngcntr}
\usepackage{apptools}

\usepackage{enumitem}
\setlist{  
	listparindent=\parindent,
	parsep=0pt,
}

\newtheorem{theorem}{Theorem}[section]

\newtheorem{case}{Case}
\newtheorem{conj}{Conjecture}[section]
\newtheorem{claim}[theorem]{Claim}
\theoremstyle{definition}
\newtheorem{defn}{Definition}[section]

\newtheorem{lemma}[theorem]{Lemma}
\newtheorem{corollary}[theorem]{Corollary}

\newcommand{\pmcgg}{{{\mathcal{M}_p\left(\mathcal{G}_{2\times n}\right)}}}

\newcommand{\pmplt}{{{\mathcal{M}_p\left(\mathcal{E}_{n,k}\right)}}}
\newcommand{\pma}{{{\mathcal{M}_p\left(\mathcal{O}^{\mathcal{A}}_{n,k}\right)}}}
\newcommand{\pms}{{{\mathcal{M}_p\left(\mathcal{O}^{\mathcal{S}}_{n,k}\right)}}}
\newcommand{\pmt}{{{\mathcal{M}_p\left(\Delta_k\right)}}}
\newcommand{\twongg}{{{\mathcal{G}_{2 \times n}}}}
\newcommand{\ind}{\mathrm{Ind}}

\theoremstyle{remark}
\newtheorem*{remark}{Remark}

\usepackage{pgfplots}
\pgfplotsset{compat=1.15}
\usepackage{mathrsfs}
\usetikzlibrary{arrows}
\definecolor{ffffff}{rgb}{1,1,1}
\definecolor{xdxdff}{rgb}{0.49019607843137253,0.49019607843137253,1}
\definecolor{ududff}{rgb}{0.30196078431372547,0.30196078431372547,1}
\definecolor{qqffff}{rgb}{0,1,1}
\definecolor{qqffqq}{rgb}{0,1,0}
\definecolor{ffqqqq}{rgb}{1,0,0}
\definecolor{qqqqff}{rgb}{0,0,1}
\definecolor{qqqqcc}{rgb}{0,0,0.8}
\definecolor{ccqqqq}{rgb}{0.8,0,0}
\definecolor{wwwwww}{rgb}{0.4,0.4,0.4}

\makeatletter
\newcommand{\addresseshere}{%
  \enddoc@text\let\enddoc@text\relax
}
\makeatother

\makeatother

\begin{document}

\title{Perfect Matching Complexes of Polygonal Line Tilings}
\author{Himanshu Chandrakar}
\address{Department of Mathematics, Indian Institute of Technology Bhilai, himanshuc@iitbhilai.ac.in}
\author{Anurag Singh}
\address{Department of Mathematics, Indian Institute of Technology Bhilai, anurags@iitbhilai.ac.in}
\date{}


\begin{abstract}
    \justifying
    The \textit{perfect matching complex} of a simple graph $G$ is a simplicial complex having facets (maximal faces) as the \textit{perfect matchings} of $G$. This article discusses the perfect matching complex of polygonal line tilings and the $\left(2 \times n\right)$-grid graph in particular. We use tools from discrete Morse theory to show that the perfect matching complex of any polygonal line tiling is either contractible or homotopy equivalent to a wedge of spheres. While proving our results, we also characterize all the matchings of $\left(2 \times n\right)$-grid graph that cannot be extended to form a perfect matching.    
\end{abstract}

\maketitle
    
\section{Introduction}

The study of the topology of simplicial complexes derived from graph properties is a well-studied problem in topological combinatorics. Jonsson's work \cite{jonsson2008simplicial} on this is an excellent treatise in this regard. The matching complex is one such simplicial complex that has been an active research topic for about three decades now. A matching complex is a simplicial complex defined on the edge set of a graph where the maximal faces are the \textit{maximal matchings} (see \Cref{defn: Matching}) of the graph. Initial work on matching complexes can be traced back to 1992 from a paper by Bouc \cite{bouc1992homologie} in connection with the Brown and Quillen complexes. Afterward, much of the work in this regard appeared in association with other areas of topology, algebra and combinatorics. For instance, the matching complexes of the complete bipartite graphs (also known as the \textit{chessboard complexes}) are well studied (for example, see \cite{christos},  \cite{bjorner1994chessboard}, \cite{jojic2018h}, \cite{shareshian2007torsion}, \cite{ziegler1994shellability}). Wach's \cite{wachs2003topology} survey is an excellent article for more information about these complexes.

One also has a simplicial complex associated with a graph, defined using the \textit{perfect matchings} (see \Cref{defn: Perfect Matching}) of the graph, known as the \textit{perfect matching complex}. The perfect matching complex of a simple graph $G$, denoted by $\mathcal{M}\left(G\right)$, is a simplicial complex where the maximal faces are the perfect matchings of $G$. In a recent work \cite{bayer2022perfect}, Bayer et al. have discussed the homotopy type of the perfect matching complexes of honeycomb graphs for certain cases. Their motivation for studying the honeycomb graph's perfect matching complex lies in understanding the connection between the matching complex and the perfect matching complex. This is also the driving force for our work in this paper.

This paper first discusses the homotopy type of the perfect matching complex of the $\left(2\times n\right)$-grid graph, denoted by $\twongg$. The reason for choosing this graph lies in the fact that the matching complex of the $\left(2\times n\right)$-grid graph has already been studied (see \cite{Matsushita2018MatchingCO}); thus, learning about the perfect matching complex might help us understand a connection between these two complexes. 

Apart from the topological viewpoint, a graph with perfect matching is also rich in combinatorial properties. Interestingly, we use one such combinatorial property to obtain our results. Precisely speaking, we look into the matchings of the graph that are not contained in any perfect matching; we coin the term \textit{bad matchings} (see \Cref{defn: Bad Matching}) for them. This property slightly resembles the problem of the extendability of matchings, studied by Plummer (\cite{PLUMMER1980201}, \cite{PLUMMER1994301}, \cite{PLUMMER1994277}). However, in our problem, we aim to find all those matchings that cannot be extended to form a perfect matching, giving our problem a different combinatorial flavour.

Moreover, a $\left(2\times n\right)$-grid graph can be visualized as a square line tiling. This observation allows us to generalize our problem of discussing the perfect matching complex of $\twongg$ to the perfect matching complex of polygonal line tilings of any size, {\it i.e.}, any number of polygons attached in a line. This is the other problem we discuss in this paper. A result for the hexagonal line tiling is proved in \cite{bayer2022perfect}, in which the authors proved that the homotopy type of the perfect matching complex of the hexagonal line tiling is contractible using the nerve lemma. Moreover, like the $\left(2\times n\right)$-grid graph, the information regarding the matching complex of polygonal line tilings is also available to us due to the work of Matsushita (\cite{matsushita2022matching}) and Bayer et al. (\cite{GeneralBayer}); thus, we can still talk about the problem which dragged us here.

We majorly use a well-known combinatorial tool, Discrete Morse theory, to obtain our results. Robin Forman developed this theory (\cite{forman1998morse}, \cite{forman2002user}) as a tool to study the topological properties, primarily the homotopy types of the simplicial complexes. This approach involves matching faces within a simplicial complex, which is equivalent to a sequence of collapses, thereby establishing homotopy equivalence with a simplified cell complex. A convenient way of doing this is to define a \textit{discrete Morse function} on the simplicial complex. However, finding an \textit{optimum} discrete Morse function is $\mathtt{NP}-$hard (\cite{joswig}), equivalently, finding a favourable sequence of element pairing (see \Cref{def: Element Pairing}) in our case.

In this paper, we explicitly construct acyclic pairings to address all the problems related to polygonal line tilings. The only exception is the case of $\left(2 \times n\right)$-grid graphs, where we apply the fold lemma for the independence complexes (\Cref{foldlemma}), which effectively constructs an acyclic pairing (see proof of \cite[Lemma 2.4]{engstrom_clawfree}). More precisely, we prove the following results in this article.

\begin{theorem}
The perfect matching complex of 
    \begin{enumerate}
        \item (\Cref{thm: 2n grid graph homotopy}) the $\left(2 \times n\right)$-grid graph is homotopy equivalent to a $k$-sphere when $n$ is even and written as $n=2k+2$, and it is contractible when $n$ is odd;
        \item (\Cref{Theorem 7}) the line tiling of triangles is homotopy equivalent to the wedge sum of spheres or is contractible according to the number of triangles attached;
        \item (\Cref{thm: homotopy of 2n-gon}, \Cref{Theorem 6}, \Cref{Theorem 8}) the general line tiling of polygons with sides strictly greater than \textit{four} is contractible.
    \end{enumerate}
\end{theorem}

The flow of the article is as follows: In the next section, we provide all the necessary definitions and preliminary results. \Cref{Section 3} discusses the homotopy type of the perfect matching complex of $\twongg$ by first defining all the possible bad matchings of $\twongg$ and then using the fold lemma. \Cref{Section 4} discusses the homotopy type of the perfect matching complex of a line tiling of even-sided polygons. \Cref{Section 5} discusses the homotopy type of the perfect matching complex of the line tiling of the odd-sided polygons as it follows from the result obtained for the perfect matching complex of $\twongg$ and the line tiling of the even-sided polygons. We end the paper by exploring some avenues for future research.

\section{Preliminaries}

\subsection{Graph theoretic notions}
    Let $G$ be a simple graph with $V\left(G\right)$ and $E\left(G\right)$ denoting the set of vertices and edges, respectively. Here, $E\left(G\right)$ is a subset of the set of cardinality two subsets of $V\left(G\right)$. We will use $V$ and $E$ if there is no ambiguity about the graph in discussion. 
    
    \begin{defn}\label{defn: Matching}
        A \textit{matching}, say $m$ in $G$, is a subset of the edge set $E$ in which no two edges share a common vertex.
    \end{defn}

    \begin{defn}\label{defn: Perfect Matching}
        A \textit{perfect matching} in $G$ is a matching $m_p$ that covers all the vertices, \textit{i.e.}, for all $v \in V$, there is $e \in m_p$ such that $v\in e$.
    \end{defn}

    For example, consider $C_6$ (cycle graph on six vertices) and label the edges and vertices of $C_6$ as shown in \Cref{fig: Bad matchings of C6}. Here, $\{a,\ d\}$ is a matching of $C_6$, and $\{a,\ f\}$ is not because they share a common vertex, namely $1$. The edge set $\left\{a,\ c,\ e\right\}$ is a perfect matching of $C_6$ since it covers all the vertices of $C_6$.
    
    Note here that we discuss the perfect matching of a graph having \textit{even} number of vertices. We say that a matching $m$ in $G$ can be \textit{extended} to form a perfect matching if a perfect matching $m_p$ exists, such that $m \subseteq m_p$. 

    \begin{defn}\label{defn: Bad Matching}
        A \textit{bad matching} is a minimal matching that cannot be extended to form a perfect matching in $G$.
    \end{defn}
    \vspace{-1.5em}
    \begin{figure}[h!]
    \centering
        \begin{minipage}{0.35\textwidth}
            \captionsetup{justification=centering}
        \centering
        \hspace{1.5em}\begin{tikzpicture}[scale=0.8,line cap=round,line join=round,>=triangle 45,x=0.5cm,y=0.5cm]
            \clip(-0.8,-20) rectangle (6.8,6.2);
            \draw [line width=1pt] (0,3)-- (0,0);
            \draw [line width=1pt,color=ffqqqq] (0,0)-- (2.598076211353316,-1.5);
            \draw [line width=1pt] (2.598076211353316,-1.5)-- (5.196152422706633,0);
            \draw [line width=1pt] (5.196152422706633,0)-- (5.196152422706634,3);
            \draw [line width=1pt,color=ffqqqq] (5.196152422706634,3)-- (2.5980762113533187,4.5);
            \draw [line width=1pt] (2.5980762113533187,4.5)-- (0,3);
            
            \begin{scriptsize}
                \draw [fill=xdxdff] (0,3) circle (2.5pt);
                \draw[color=xdxdff] (-0.5,3.2) node {\large 6};
                \draw [fill=xdxdff] (0,0) circle (2.5pt);
                \draw[color=xdxdff] (-0.5,-0.1) node {\large 5};
                \draw[color=black] (-0.35,1.5) node {\large$e$};
                \draw[color=black] (1,-1.1) node {\large$d$};
                \draw[color=black] (4.1,-1) node {\large$c$};
                \draw[color=black] (5.5,1.5) node {\large$b$};
                \draw[color=black] (4.1,4.1) node {\large$a$};
                \draw[color=black] (1,4.15) node {\large$f$};
                \draw [fill=xdxdff] (2.598076211353316,-1.5) circle (2.5pt);
                \draw[color=xdxdff] (2.6,-2.1) node {\large 4};
                \draw [fill=xdxdff] (5.196152422706633,0) circle (2.5pt);
                \draw[color=xdxdff] (5.65,-0.1) node {\large 3};
                \draw [fill=xdxdff] (5.196152422706634,3) circle (2.5pt);
                \draw[color=xdxdff] (5.65,3.2) node {\large 2};
                \draw [fill=xdxdff] (2.5980762113533187,4.5) circle (2.5pt);
                \draw[color=xdxdff] (2.6,5.2) node {\large 1};
            \end{scriptsize}
            \end{tikzpicture}
            \vspace{-18em}
            \caption{The bad matching $\left\{a,d\right\}$ of $C_6$.}
            \label{fig: Bad matchings of C6}
        \end{minipage}\hfill
        \begin{minipage}{0.6\textwidth}
        \captionsetup{justification=centering}
        \centering
        \begin{tikzpicture}[scale=0.8,line cap=round,line join=round,>=triangle 45,x=0.5cm,y=0.5cm]
            \clip(-0.8,-20) rectangle (16.8,6.2);
            \draw [line width=1pt] (0,3)-- (0,0);
            \draw [line width=1pt] (0,0)-- (2.598076211353316,-1.5);
            \draw [line width=1pt,color=ffqqqq] (2.598076211353316,-1.5)-- (5.196152422706633,0);
            \draw [line width=1pt] (5.196152422706633,0)-- (5.196152422706634,3);
            \draw [line width=1pt,color=ffqqqq] (5.196152422706634,3)-- (2.5980762113533187,4.5);
            \draw [line width=1pt] (2.5980762113533187,4.5)-- (0,3);

            \draw [line width=1.2pt,{Stealth[]}-] (9,1.5) -- (6.196152422706633,1.5) ;
            
            \draw [line width=1pt,color=ffqqqq] (10,3)-- (10,0);
            \draw [line width=1pt] (10,0)-- (12.598076211353316,-1.5);
            \draw [line width=1pt,color=ffqqqq] (12.598076211353316,-1.5)-- (15.196152422706633,0);
            \draw [line width=1pt] (15.196152422706633,0)-- (15.196152422706634,3);
            \draw [line width=1pt,color=ffqqqq] (15.196152422706634,3)-- (12.5980762113533187,4.5);
            \draw [line width=1pt] (12.5980762113533187,4.5)-- (10,3);
            
            \begin{scriptsize}
                \draw [fill=xdxdff] (0,3) circle (2.5pt);
                \draw[color=xdxdff] (-0.5,3.2) node {\large 6};
                \draw [fill=xdxdff] (0,0) circle (2.5pt);
                \draw[color=xdxdff] (-0.5,-0.1) node {\large 5};
                \draw[color=black] (-0.35,1.5) node {\large$e$};
                \draw[color=black] (1,-1.1) node {\large$d$};
                \draw[color=black] (4.1,-1) node {\large$c$};
                \draw[color=black] (5.5,1.5) node {\large$b$};
                \draw[color=black] (4.1,4.1) node {\large$a$};
                \draw[color=black] (1,4.15) node {\large$f$};
                \draw [fill=xdxdff] (2.598076211353316,-1.5) circle (2.5pt);
                \draw[color=xdxdff] (2.6,-2.1) node {\large 4};
                \draw [fill=xdxdff] (5.196152422706633,0) circle (2.5pt);
                \draw[color=xdxdff] (5.65,-0.1) node {\large 3};
                \draw [fill=xdxdff] (5.196152422706634,3) circle (2.5pt);
                \draw[color=xdxdff] (5.65,3.2) node {\large 2};
                \draw [fill=xdxdff] (2.5980762113533187,4.5) circle (2.5pt);
                \draw[color=xdxdff] (2.6,5.2) node {\large 1};

                \draw [fill=xdxdff] (10,3) circle (2.5pt);
                \draw[color=xdxdff] (9.5,3.2) node {\large 6};
                \draw [fill=xdxdff] (10,0) circle (2.5pt);
                \draw[color=xdxdff] (9.5,-0.1) node {\large 5};
                \draw[color=black] (9.65,1.5) node {\large$e$};
                \draw[color=black] (11,-1.1) node {\large$d$};
                \draw[color=black] (14.1,-1) node {\large$c$};
                \draw[color=black] (15.5,1.5) node {\large$b$};
                \draw[color=black] (14.1,4.1) node {\large$a$};
                \draw[color=black] (11,4.15) node {\large$f$};
                \draw [fill=xdxdff] (12.598076211353316,-1.5) circle (2.5pt);
                \draw[color=xdxdff] (12.6,-2.1) node {\large 4};
                \draw [fill=xdxdff] (15.196152422706633,0) circle (2.5pt);
                \draw[color=xdxdff] (15.65,-0.1) node {\large 3};
                \draw [fill=xdxdff] (15.196152422706634,3) circle (2.5pt);
                \draw[color=xdxdff] (15.65,3.2) node {\large 2};
                \draw [fill=xdxdff] (12.5980762113533187,4.5) circle (2.5pt);
                \draw[color=xdxdff] (12.6,5.2) node {\large 1};
            \end{scriptsize}
            \end{tikzpicture}
            \vspace{-18em}
            \caption{Extending the matching $\left\{a,c\right\}$ to the perfect matching $\left\{a,c,e\right\}$ of $C_6$.}
            \label{fig: Matching extending to perfect matching}
            \end{minipage}
    \end{figure}

     The term \lq minimal\rq\ saves us from counting any matching (again as a bad matching) that properly contains a bad matching, as such matchings cannot be extended to form a perfect matching, too. For example, consider the same graph $C_6$. Note that the edge sets $\left\{a,\ c\right\}$ and $\left\{a,\ d\right\}$ are matchings, among which $\left\{a,\ c\right\}$ can be extended to form the perfect matching $\left\{a,\ c,\ e\right\}$ (see \Cref{fig: Matching extending to perfect matching}), whereas $\left\{a,\ d\right\}$ can never be extended to form a perfect matching of $C_6$ since there is no edge left to cover the vertex 3 or 6.  It is easy to see that $\left\{a\right\}$ and $\left\{d\right\}$ can each be extended to perfect matchings and therefore $\left\{a,\ d\right\}$ is minimal (and hence bad) matching in this case (see \Cref{fig: Bad matchings of C6}).
    
    \subsection{Discrete Morse theory}

    As stated earlier, the primary tool for our discussion is the discrete Morse theory formulated by Forman \cite{forman1998morse, forman2002user}. This is the classical reference for discrete Morse theory, but we use the definitions and notations given in \cite{kozlov2008combinatorial}. We also follow the basic terminologies of the abstract simplicial complex from \cite{kozlov2008combinatorial} and refrain from discussing it here. 


    Note that the set of vertices for the perfect matching complex of a graph $G$ is a subset of $E(G)$. However, in our discussion, we will use the term \textit{edge} to denote the element of the set of vertices of the perfect matching complex. This is done to avoid confusion with the elements of $V(G)$ since we extensively mention them in our proof.
    
    Following the terminology used in \cite{bayer2022perfect}, we will use the term \textit{paired} instead of the conventional expression \textit{matched} when two faces are associated with each other to prevent confusion with matchings of a graph. We now explore some essential tools from discrete Morse theory needed in the subsequent sections.

    \begin{defn} (\cite[{Definition~11.1}]{kozlov2008combinatorial})
        \begin{enumerate}
            \item\label{defn: Partial pairing} A \textit{partial pairing} in a poset $\mathcal{P}$ is a subset $\mathcal{M} \subseteq \mathcal{P}\times\mathcal{P}$, such that
            \begin{itemize}
                \item $\left(\alpha,\beta\right) \in \mathcal{M}$ implies $\alpha \prec \beta$;
                \item each $\gamma \in \mathcal{P}$ belongs to at most one element (pair) in $\mathcal{M}$.
            \end{itemize}
            Here, $\alpha \prec \beta$ means there exists no $\delta \in \mathcal{P}$ such that $\alpha < \delta < \beta$. Moreover, note that $\mathcal{M}$ is a partial matching on a poset $\mathcal{P}$ if and only if there exists $\mathcal{T} \subset \mathcal{P}$ and an injective map $\phi: \mathcal{T} \to \mathcal{P} \setminus \mathcal{T}$ such that $t \prec \phi\left(t\right)$ for all $t \in \mathcal{T}$.

            \item\label{Acyclic pairing} A partial pairing on $\mathcal{P}$ is said to be \textit{acyclic} if there does not exist a cycle, $$\alpha_1 \prec \phi\left(\alpha_1\right) \succ \alpha_2 \prec \phi\left(\alpha_2\right) \succ \dots \succ \alpha_k \prec \phi\left(\alpha_k\right) \succ \alpha_1,$$
            where $k\geq2$ and all $\alpha_i \in \mathcal{P}$ are distinct.
        \end{enumerate}

        For an acyclic pairing $\mathcal{M}$ on poset $\mathcal{P}$, we define \textit{critical elements} to be those elements of $\mathcal{P}$ which remain unpaired.
    \end{defn}

    We now state the main theorem of discrete Morse theory.

    \begin{theorem}\label{theorem 1}(\cite[{Theorem~11.13}]{kozlov2008combinatorial})
        Let $\mathcal{K}$ be a polyhedral cell complex, and let $\mathcal{M}$ be an acyclic pairing on the face poset of $\mathcal{K}$. If $c_i$ denotes the number of critical $i$-dimensional cells of $\mathcal{K}$ then the space $\mathcal{K}$ is homotopy equivalent to a cell complex $\mathcal{K}_c$ with $c_i$ cells of dimension $i$ for each $i \geq 0$, plus a single 0-dimensional cell in the case where the empty set is matched in the pairing.
    \end{theorem}

    Based on the above theorem, the following information can be obtained:

    \begin{corollary}\label{Corollary 1}(\cite[Corollary~2.5]{deshpande2020higher})
        For an acyclic pairing $\mathcal{M}$, if all the critical cells in $\mathcal{M}$ are of dimension $d$, then $\mathcal{K}$ is homotopy equivalent to a wedge of $d$-dimensional spheres.
    \end{corollary}

    We would like to point out here that the empty wedge would mean the space is contractible. A convenient way of matching elements on the face poset is to perform \textit{element pairing} on them. We define this as follows.

    \begin{defn}[\cite{deshpande2020higher}, \cite{jonsson2008simplicial}] \label{def: Element Pairing}
        Let $x$ be a vertex and $\Delta$ be a simplicial complex. The \textit{element pairing using $x$ on $\Delta$} is defined as the following set of pairs, \[M_{x} = \left\{\left(\sigma,\sigma \cup \left\{x\right\}\right)\ |\ x \notin \sigma,\ \sigma \cup \left\{x\right\} \in \Delta\right\}\]
    \end{defn}

    Let $C_{x}$ denote the set of cells left unpaired after element pairing with $x$ (critical cells). We then define element pairing using some other element, say $y$, on the elements of $C_{x}$ and denote the critical cells left unpaired after this element pairing as $C_{y}$. Note that, $C_{x} \supset C_{y}$. Similarly, we get a sequence of element pairings, and we stop this process when either no critical cells are left or when any additional element pairing results in the same set of critical cells. The advantage of performing this procedure is that the union of a sequence of element pairings is an acyclic matching, as stated in the following theorem.

    \begin{theorem}\label{theorem 2}(\cite[Proposition~3.1]{deshpande2020higher}, \cite[Lemma~4.1]{jonsson2008simplicial})
        Let $\Delta$ be a simplicial complex and $\{v_{1},v_{2},\dots, $ $v_{t}\}$ be a subset of the vertex set of $\Delta$. Let $\Delta_0 = \Delta$, and for all $i \in \left\{1,\ 2,\dots,\ t\right\}$, define
        \begin{align*}
            M_{v_{i}} &= \left\{\left(\sigma,\ \sigma \cup \left\{v_{i}\right\}\right)\ \middle|\ v_{i} \notin \sigma, \text{and}\ \sigma,\ \sigma \cup \left\{v_{i}\right\} \in \Delta_{{i-1}}\right\},\\
            N_{v_{i}} &= \left\{\sigma \in \Delta_{i-1}\ \middle|\ \sigma \in \eta\ \text{for some}\ \eta \in M_{v_{i}}\right\},\ \text{and}\\
            \Delta_i &= \Delta_{i-1} \setminus N_{v_{i}}.
        \end{align*}
        Then $\bigsqcup_{i=1}^t{M_{v_{i}}}$ is an acyclic pairing on $\Delta$.
    \end{theorem}
    
    Note that $\Delta_i$, for all $i\neq 0$, is the set of critical cells $C_{v_{_i}}$ we defined earlier. In the proof of our results, we will analyze the characteristics of the critical cells left unpaired after every element pairing.

\section{The Perfect Matching Complex of \texorpdfstring{$\left(2\times n\right)$}~ Grid Graph}\label{Section 3}
    
    In this section, we find the homotopy type of the perfect matching complex of $\left(2\times n\right)$-grid graph. Let $\mathcal{M}_p\left(\mathcal{G}_{2\times n}\right)$ denote the perfect matching complex of $2\times n$ grid graph, where $\mathcal{G}_{2\times n}$ denotes the $2\times n$ grid graph.

    For $\twongg$ the vertex set $V\left(\twongg\right)$ and edge set $E\left(\twongg\right)$ of $\twongg$ is defined as follows:
    \begin{align*}
        V\left(\twongg\right) &= \left\{p_{{i,j}}\ \middle|\ i\in\left\{1,2,\dots,n\right\} \text{ and } j\in\left\{1,2\right\}\right\};\\
        E\left(\twongg\right) &= \left\{a_{i}, b_{j},c_{j}\ \middle|\ i\in\left\{1,2,\dots,n\right\} \text{ and } j\in\left\{1,2,\dots,n-1\right\}\right\},
    \end{align*}
    \begin{figure}[H]
    \centering
    \vspace{-2em}
        \begin{tikzpicture}[line cap=round,line join=round,>=triangle 45,x=0.7cm,y=0.7cm]
            \clip(-1,-8) rectangle (16,4);
            \draw [line width=1pt] (0,3)-- (3,3);
            \draw [line width=1pt] (3,3)-- (6,3);
            \draw [line width=1pt] (6,3)-- (9,3);
            \draw [line width=1pt,dash pattern=on 5pt off 5pt] (9,3)-- (12,3);
            \draw [line width=1pt] (12,3)-- (15,3);
            \draw [line width=1pt] (15,3)-- (15,0);
            \draw [line width=1pt] (0,3)-- (0,0);
            \draw [line width=1pt] (0,0)-- (3,0);
            \draw [line width=1pt] (3,0)-- (6,0);
            \draw [line width=1pt] (6,0)-- (9,0);
            \draw [line width=1pt,dash pattern=on 5pt off 5pt] (9,0)-- (12,0);
            \draw [line width=1pt] (12,0)-- (15,0);
            \draw [line width=1pt] (3,3)-- (3,0);
            \draw [line width=1pt] (6,3)-- (6,0);
            \draw [line width=1pt] (9,3)-- (9,0);
            \draw [line width=1pt] (12,3)-- (12,0);
            \begin{scriptsize}
                \draw [fill=ududff] (0,3) circle (1.5pt);
                \draw [fill=ududff] (3,3) circle (1.5pt);
                \draw [fill=ududff] (6,3) circle (1.5pt);
                \draw [fill=ududff] (9,3) circle (1.5pt);
                \draw [fill=ududff] (12,3) circle (1.5pt);
                \draw [fill=ududff] (15,3) circle (1.5pt);
                \draw [fill=ududff] (15,0) circle (1.5pt);
                \draw [fill=ududff] (0,0) circle (1.5pt);
                \draw [fill=ududff] (3,0) circle (1.5pt);
                \draw [fill=ududff] (6,0) circle (1.5pt);
                \draw [fill=ududff] (9,0) circle (1.5pt);
                \draw [fill=ududff] (12,0) circle (1.5pt);
                
                \draw[color=ududff] (-0.2,3.4) node[anchor=center] {\large$p_{{1,2}}$};                
                \draw[color=ududff] (2.8,3.4) node[anchor=center] {\large$p_{{2,2}}$};
                \draw[color=ududff] (5.8,3.4) node[anchor=center] {\large$p_{{3,2}}$};
                \draw[color=ududff] (8.8,3.4) node[anchor=center] {\large$p_{{4,2}}$};
                \draw[color=ududff] (11.8,3.4) node[anchor=center] {\large$p_{{n-1,2}}$};
                \draw[color=ududff] (15,3.4) node[anchor=center] {\large$p_{{n,2}}$};

                \draw[color=ududff] (-0.2,-0.4) node[anchor=center] {\large$p_{{1,1}}$};
                \draw[color=ududff] (2.8,-0.4) node[anchor=center] {\large$p_{{2,1}}$};
                \draw[color=ududff] (5.8,-0.4) node[anchor=center] {\large$p_{{3,1}}$};
                \draw[color=ududff] (8.8,-0.4) node[anchor=center] {\large$p_{{4,1}}$};
                \draw[color=ududff] (11.8,-0.4) node[anchor=center] {\large\large$p_{{n-1,1}}$};
                \draw[color=ududff] (15,-0.4) node[anchor=center] {\large$p_{{n,1}}$};

                \draw[color=black] (0.4,1.5) node[anchor=center] {\large$a_{1}$};
                \draw[color=black] (3.4,1.5) node[anchor=center] {\large$a_{2}$};
                \draw[color=black] (6.4,1.5) node[anchor=center] {\large$a_{3}$};
                \draw[color=black] (9.4,1.5) node[anchor=center] {\large$a_{4}$};
                \draw[color=black] (12.75,1.5) node[anchor=center] {\large$a_{{n-1}}$};
                \draw[color=black] (15.4,1.5) node[anchor=center] {\large$a_{n}$};
                                
                \draw[color=black] (1.4,3.4) node[anchor=center] {\large$b_{1}$};
                \draw[color=black] (4.4,3.4) node[anchor=center] {\large$b_{2}$};                
                \draw[color=black] (7.4,3.4) node[anchor=center] {\large$b_{3}$};         
                \draw[color=black] (13.5,3.4) node[anchor=center] {\large$b_{{n-1}}$};                
  
                \draw[color=black] (1.4,-0.4) node[anchor=center] {\large\large$c_{1}$};                
                \draw[color=black] (4.4,-0.4) node[anchor=center] {\large$c_{2}$};                
                \draw[color=black] (7.4,-0.4) node[anchor=center] {\large$c_{3}$};                
                \draw[color=black] (13.5,-0.4) node[anchor=center] {\large$c_{{n-1}}$};
            \end{scriptsize}
        \end{tikzpicture}
        \vspace{-13em}
        \caption{Labelling in $2\times n$ Grid Graph}
        \vspace{-1em}
        \label{fig: 2n Grid Graph}
    \end{figure}

    where $a_{i}$'s are the vertical edges with the endpoints $p_{{i,1}}$ and $p_{{i,2}}$; $b_{j}$'s and $c_{j}$'s, are the horizontal edges with the endpoints of the edge $b_{j}$, are $p_{{j,2}}$ and $p_{{j+1,2}}$ and the endpoints of the edge $c_{j}$, are $p_{{j,1}}$ and $p_{{j+1,1}}$ (see \Cref{fig: 2n Grid Graph}).

    In order to find the homotopy type of $\mathcal{M}_p\left(\mathcal{G}_{2\times n}\right)$, we are going to need the following two lemmas:

    \begin{lemma}\label{Lemma 1}
    A matching $\pi$ of $\twongg$ is a bad matching if and only if it has either of the following forms,

    \begin{tasks}[label-offset = 1em, label = \text{X}\arabic*.](2)
        \task\label{X1} $\left\{b_{i},\ c_{{i+1}}\right\}$, for all $1\leq i\leq n-2$.
        \vspace{-5em}
        \begin{figure}[H]
            \captionsetup{oneside,margin={-0.2cm,0.5cm}}
            \hspace{-2em}
            \begin{tikzpicture}[line cap=round,line join=round,>=triangle 45,x=1.2cm,y=1.2cm]
                \clip(0,0) rectangle (6,4);
                    \draw [line width=1pt,dash pattern=on 5pt off 5pt] (5,2)-- (4,2);
                    \draw [line width=1pt] (4,2)-- (3,2);
                    \draw [line width=1pt,color=ffqqqq] (3,2)-- (2,2);
                    \draw [line width=1pt] (2,2)-- (2,1);
                    \draw [line width=1pt] (2,1)-- (3,1);
                    \draw [line width=1pt,color=ffqqqq] (3,1)-- (4,1);
                    \draw [line width=1pt,dash pattern=on 5pt off 5pt] (4,1)-- (5,1);
                    \draw [line width=1pt] (4,2)-- (4,1);
                    \draw [line width=1pt] (3,2)-- (3,1);
                    \draw [line width=1pt,dash pattern=on 5pt off 5pt] (2,2)-- (1,2);
                    \draw [line width=1pt,dash pattern=on 5pt off 5pt] (2,1)-- (1,1);
                    \draw (2.5,2.25) node[anchor=center] {$b_{{i}}$};
                    \draw (3.5,1.2) node[anchor=center] {$c_{{i+1}}$};
                    \begin{scriptsize}
                        \draw [fill=wwwwww] (5,2) circle (1.5pt);
                        \draw [fill=wwwwww] (4,2) circle (1.5pt);
                        \draw [fill=wwwwww] (3,2) circle (1.5pt);
                        \draw [fill=wwwwww] (2,2) circle (1.5pt);
                        \draw [fill=wwwwww] (1,2) circle (1.5pt);
                        \draw [fill=wwwwww] (2,1) circle (1.5pt);
                        \draw [fill=wwwwww] (3,1) circle (1.5pt);
                        \draw [fill=wwwwww] (4,1) circle (1.5pt);
                        \draw [fill=wwwwww] (5,1) circle (1.5pt);
                        \draw [fill=wwwwww] (1,1) circle (1.5pt);
                    \end{scriptsize}
                \end{tikzpicture}
            \vspace{-3em}
            \caption{The set given in X1.}
            \label{fig: X1}
        \end{figure}
            
        \task\label{X2} $\left\{c_{i},\ b_{{i+1}}\right\}$, for all $1\leq i\leq n-2$.
        \vspace{-5em}
        \begin{figure}[H]
            \captionsetup{oneside,margin={-0.2cm,0.5cm}}
            \hspace{-2em}\begin{tikzpicture}[line cap=round,line join=round,>=triangle 45,x=1.2cm,y=1.2cm]
                \clip(0,0) rectangle (6,4);
                    \draw [line width=1pt,dash pattern=on 5pt off 5pt] (5,2)-- (4,2);
                    \draw [line width=1pt,color=ffqqqq] (4,2)-- (3,2);
                    \draw [line width=1pt] (3,2)-- (2,2);
                    \draw [line width=1pt] (2,2)-- (2,1);
                    \draw [line width=1pt,color=ffqqqq] (2,1)-- (3,1);
                    \draw [line width=1pt] (3,1)-- (4,1);
                    \draw [line width=1pt,dash pattern=on 5pt off 5pt] (4,1)-- (5,1);
                    \draw [line width=1pt] (4,2)-- (4,1);
                    \draw [line width=1pt] (3,2)-- (3,1);
                    \draw [line width=1pt,dash pattern=on 5pt off 5pt] (2,2)-- (1,2);
                    \draw [line width=1pt,dash pattern=on 5pt off 5pt] (2,1)-- (1,1);
                    \draw (3.5,2.25) node[anchor=center] {$b_{{i+1}}$};
                    \draw (2.5,1.2) node[anchor=center] {$c_{{i}}$};
                    \begin{scriptsize}
                        \draw [fill=wwwwww] (5,2) circle (1.5pt);
                        \draw [fill=wwwwww] (4,2) circle (1.5pt);
                        \draw [fill=wwwwww] (3,2) circle (1.5pt);
                        \draw [fill=wwwwww] (2,2) circle (1.5pt);
                        \draw [fill=wwwwww] (1,2) circle (1.5pt);
                        \draw [fill=wwwwww] (2,1) circle (1.5pt);
                        \draw [fill=wwwwww] (3,1) circle (1.5pt);
                        \draw [fill=wwwwww] (4,1) circle (1.5pt);
                        \draw [fill=wwwwww] (5,1) circle (1.5pt);
                        \draw [fill=wwwwww] (1,1) circle (1.5pt);
                    \end{scriptsize}
                \end{tikzpicture}
            \vspace{-3em}
            \caption{The set given in X2.}
            \label{fig: X2}
            \end{figure}
        \end{tasks}
    \end{lemma}

    \begin{proof}
        We first show that if the matching $\pi$ is $\left\{b_{i},\ c_{{i+1}}\right\}$ or $\left\{c_{i},\ b_{{i+1}}\right\}$, for some $1\leq i \leq n-2$, then $\pi$ is a bad matching. Note that there will be an odd number of vertices on the left and right-hand sides of the edges in both of these matchings, which cannot be covered using any possible choice of edges. Therefore, these are bad matchings.
    
        Conversely, we now show that if $\pi$ is bad matching of $\twongg$ then either $\pi = \left\{b_{i},\ c_{{i+1}}\right\}$ or $\pi = \left\{c_{i},\ b_{{i+1}}\right\}$, for some $1\leq i \leq n-2$. Let us assume to the contrary that, $\pi = \left\{x_{1}, x_{2},\dots, x_{t}\right\}$, where $x_{i} \in E$, for all $1\leq i \leq t$ is a bad matching of $\twongg$ other than the form \hyperref[X1]{X1} or \hyperref[X2]{X2}. If $V_{\pi}$ denote the vertices covered by the edges in $\pi$, then $\twongg \left[V\setminus V_{\pi}\right]$ will give us non-empty connected components, say, $L_1,L_2,\dots,L_s$.

        Let $j$ be the largest index such that a vertex of $a_j$ is present in $L_1$, \textit{i.e.}, no vertex of $a_{j+1}$ is in $L_1$. In other words, at least one of the following holds:  

        \begin{tasks}[label = \Alph*.](3)
            \task\label{A} $p_{{j,1}} \in L_1,\ p_{{j,2}} \notin L_1$;
            \begin{figure}[H]
            \captionsetup{oneside,margin={0cm,0.5cm}}
            \hspace{-1.5em}
                \begin{tikzpicture}[line cap=round,line join=round,>=triangle 45,x=0.35cm,y=0.35cm]                  \clip(-0.2,-1) rectangle (12.2,4);
                    \draw [line width=1pt,dash pattern=on 5pt off 5pt] (0,3)-- (3,3);
                    \draw [line width=1pt] (3,3)-- (6,3);
                    \draw [line width=1pt] (6,3)-- (9,3);
                    \draw [line width=1pt,dash pattern=on 5pt off 5pt] (9,3)-- (12,3);
                    \draw [line width=1pt,dash pattern=on 5pt off 5pt] (0,0)-- (3,0);
                    \draw [line width=1pt] (3,0)-- (6,0);
                    \draw [line width=1pt] (6,0)-- (9,0);
                    \draw [line width=1pt,dash pattern=on 5pt off 5pt] (9,0)-- (12,0);
                    \draw [line width=1pt] (3,3)-- (3,0);
                    \draw [line width=1pt] (6,3)-- (6,0);
                    \draw [line width=1pt] (9,3)-- (9,0);
                    \draw [rotate around={90:(6,1.5)},line width=1.5pt,dash pattern=on 5pt off 5pt] (6,1.5) ellipse (0.8cm and 0.4cm);
                    \begin{scriptsize}
                        \draw [fill=wwwwww] (0,3) circle (2pt);
                        \draw [fill=wwwwww] (3,3) circle (2pt);
                        \draw [fill=ffqqqq] (6,3) circle (2pt);
                        \draw [fill=ffqqqq] (9,3) circle (2pt);
                        \draw [fill=wwwwww] (12,3) circle (2pt);
                        \draw [fill=wwwwww] (0,0) circle (2pt);
                        \draw [fill=wwwwww] (3,0) circle (2pt);
                        \draw [fill=qqqqff] (6,0) circle (2pt);
                        \draw [fill=ffqqqq] (9,0) circle (2pt);
                        \draw [fill=wwwwww] (12,0) circle (2pt);
                    \end{scriptsize}
                \end{tikzpicture}
                \caption{Case \hyperref[A]{A}.}
                \label{fig: Case A}
            \end{figure}
        
            \task\label{B} $p_{{j,2}} \in L_1,\ p_{{j,1}} \notin L_1$;
            \begin{figure}[H]
            \captionsetup{oneside,margin={-0.2cm,0.5cm}}
                \hspace{-1.5em}\begin{tikzpicture}[line cap=round,line join=round,>=triangle 45,x=0.35cm,y=0.35cm]                  \clip(-0.2,-1) rectangle (12.2,4);
                    \draw [line width=1pt,dash pattern=on 5pt off 5pt] (0,3)-- (3,3);
                    \draw [line width=1pt] (3,3)-- (6,3);
                    \draw [line width=1pt] (6,3)-- (9,3);
                    \draw [line width=1pt,dash pattern=on 5pt off 5pt] (9,3)-- (12,3);
                    \draw [line width=1pt,dash pattern=on 5pt off 5pt] (0,0)-- (3,0);
                    \draw [line width=1pt] (3,0)-- (6,0);
                    \draw [line width=1pt] (6,0)-- (9,0);
                    \draw [line width=1pt,dash pattern=on 5pt off 5pt] (9,0)-- (12,0);
                    \draw [line width=1pt] (3,3)-- (3,0);
                    \draw [line width=1pt] (6,3)-- (6,0);
                    \draw [line width=1pt] (9,3)-- (9,0);
                    \draw [rotate around={90:(6,1.5)},line width=1.5pt,dash pattern=on 5pt off 5pt] (6,1.5) ellipse (0.8cm and 0.4cm);
                    \begin{scriptsize}
                        \draw [fill=wwwwww] (0,3) circle (2pt);
                        \draw [fill=wwwwww] (3,3) circle (2pt);
                        \draw [fill=qqqqff] (6,3) circle (2pt);
                        \draw [fill=ffqqqq] (9,3) circle (2pt);
                        \draw [fill=wwwwww] (12,3) circle (2pt);
                        \draw [fill=wwwwww] (0,0) circle (2pt);
                        \draw [fill=wwwwww] (3,0) circle (2pt);
                        \draw [fill=ffqqqq] (6,0) circle (2pt);
                        \draw [fill=ffqqqq] (9,0) circle (2pt);
                        \draw [fill=wwwwww] (12,0) circle (2pt);
                    \end{scriptsize}
                \end{tikzpicture}
                \caption{Case \hyperref[B]{B}.}
                \label{fig: Case B}
            \end{figure}
        
            \task\label{C} $p_{{j,1}},\ p_{{j,2}} \in L_1$.
            \begin{figure}[H]
            \captionsetup{oneside,margin={-0.2cm,0.5cm}}
                \hspace{-1.5em}\begin{tikzpicture}[line cap=round,line join=round,>=triangle 45,x=0.35cm,y=0.35cm]                  \clip(-0.2,-1) rectangle (12.2,4);
                    \draw [line width=1pt,dash pattern=on 5pt off 5pt] (0,3)-- (3,3);
                    \draw [line width=1pt] (3,3)-- (6,3);
                    \draw [line width=1pt] (6,3)-- (9,3);
                    \draw [line width=1pt,dash pattern=on 5pt off 5pt] (9,3)-- (12,3);
                    \draw [line width=1pt,dash pattern=on 5pt off 5pt] (0,0)-- (3,0);
                    \draw [line width=1pt] (3,0)-- (6,0);
                    \draw [line width=1pt] (6,0)-- (9,0);
                    \draw [line width=1pt,dash pattern=on 5pt off 5pt] (9,0)-- (12,0);
                    \draw [line width=1pt] (3,3)-- (3,0);
                    \draw [line width=1pt] (6,3)-- (6,0);
                    \draw [line width=1pt] (9,3)-- (9,0);
                    \draw [rotate around={90:(6,1.5)},line width=1.5pt,dash pattern=on 5pt off 5pt] (6,1.5) ellipse (0.8cm and 0.4cm);
                    \begin{scriptsize}
                        \draw [fill=wwwwww] (0,3) circle (2pt);
                        \draw [fill=wwwwww] (3,3) circle (2pt);
                        \draw [fill=qqqqff] (6,3) circle (2pt);
                        \draw [fill=ffqqqq] (9,3) circle (2pt);
                        \draw [fill=wwwwww] (12,3) circle (2pt);
                        \draw [fill=wwwwww] (0,0) circle (2pt);
                        \draw [fill=wwwwww] (3,0) circle (2pt);
                        \draw [fill=qqqqff] (6,0) circle (2pt);
                        \draw [fill=ffqqqq] (9,0) circle (2pt);
                        \draw [fill=wwwwww] (12,0) circle (2pt);
                    \end{scriptsize}
                \end{tikzpicture}
                \caption{Case \hyperref[C]{C}.}
                \label{fig: Case C}
            \end{figure}
        \end{tasks}
    Here, the red-coloured vertices are covered by some edge, the blue-coloured vertex is uncovered, and the marked edge is $a_j$. We claim that every vertex of $L_1$ can be covered, \textit{i.e.}, we can get a perfect matching of $L_1$. Before proving that, first note that case \hyperref[B]{B} is equivalent to case \hyperref[A]{A}, and hence, a similar argument used for case \hyperref[A]{A} will work for case \hyperref[B]{B} also. For the case \hyperref[C]{C}, look at $a_{j-1}$. If at least one vertex is covered, then it can be considered either as case \hyperref[A]{A} or \hyperref[B]{B}, and we proceed according to the argument we provided for them. If both vertices are uncovered, look at $a_{j-2}$ and continue until all the vertical edges (\textit{i.e.}, $a_k$'s) in $L_1$ are exhausted. Thus, we only need to analyze case \hyperref[A]{A}.

    Considering \hyperref[A]{A}, note that to obtain \hyperref[A]{A}, both the vertices of $a_{j+1}$ must be covered along with the vertex $p_{{j,2}}$. This can be done in the following three ways:
    
    \begin{tasks}[label-offset = 1em, label = \text{A}\arabic*.](2)
        \task\label{A1} $b_{{j}},c_{{j+1}}\in \pi$.
        \begin{figure}[H]
            \hspace{1em}\begin{tikzpicture}[line cap=round,line join=round,>=triangle 45,x=0.35cm,y=0.35cm]                  
            \clip(-0.25,-1) rectangle (15.2,4);
                \draw [line width=1pt,dash pattern=on 5pt off 5pt] (0,3)-- (3,3);
                \draw [line width=1pt] (3,3)-- (6,3);
                \draw [line width=1pt,color=ffqqqq] (6,3)-- (9,3);
                \draw [line width=1pt] (9,3)-- (12,3);
                \draw [line width=1pt,dash pattern=on 5pt off 5pt] (0,0)-- (3,0);
                \draw [line width=1pt] (3,0)-- (6,0);
                \draw [line width=1pt] (6,0)-- (9,0);
                \draw [line width=1pt,color=ffqqqq] (9,0)-- (12,0);
                \draw [line width=1pt] (3,3)-- (3,0);
                \draw [line width=1pt] (6,3)-- (6,0);
                \draw [line width=1pt] (9,3)-- (9,0);
                \draw [line width=1pt,dash pattern=on 5pt off 5pt] (12,0)-- (15,0);
                \draw [line width=1pt,dash pattern=on 5pt off 5pt] (12,3)-- (15,3);
                \draw [line width=1pt] (12,3)-- (12,0);
                \draw [rotate around={90:(6,1.5)},line width=1.5pt,dash pattern=on 5pt off 5pt] (6,1.5) ellipse (0.8cm and 0.4cm);
                \begin{scriptsize}
                    \draw [fill=wwwwww] (0,3) circle (2pt);
                    \draw [fill=wwwwww] (3,3) circle (2pt);
                    \draw [fill=ffqqqq] (6,3) circle (2pt);
                    \draw [fill=ffqqqq] (9,3) circle (2pt);
                    \draw [fill=wwwwww] (12,3) circle (2pt);
                    \draw [fill=wwwwww] (0,0) circle (2pt);
                    \draw [fill=wwwwww] (3,0) circle (2pt);
                    \draw [fill=qqqqff] (6,0) circle (2pt);
                    \draw [fill=ffqqqq] (9,0) circle (2pt);
                    \draw [fill=ffqqqq] (12,0) circle (2pt);
                    \draw [fill=wwwwww] (15,0) circle (2pt);
                    \draw [fill=wwwwww] (15,3) circle (2pt);
                \end{scriptsize}
            \end{tikzpicture}
            \caption{Case \hyperref[A1]{A1}.}
            \label{fig: Case A1}
        \end{figure}
        
        \task\label{A2} $b_{{j-1}},a_{{j+1}}\in \pi$.

        \begin{figure}[H]
            \hspace{1em}\begin{tikzpicture}[line cap=round,line join=round,>=triangle 45,x=0.35cm,y=0.35cm]         
            \clip(-0.25,-1) rectangle (15.2,4);
                \draw [line width=1pt,dash pattern=on 5pt off 5pt] (0,3)-- (3,3);
                \draw [line width=1pt,color=ffqqqq] (3,3)-- (6,3);
                \draw [line width=1pt] (6,3)-- (9,3);
                \draw [line width=1pt] (9,3)-- (12,3);
                \draw [line width=1pt,dash pattern=on 5pt off 5pt] (0,0)-- (3,0);
                \draw [line width=1pt] (3,0)-- (6,0);
                \draw [line width=1pt] (6,0)-- (9,0);
                \draw [line width=1pt] (9,0)-- (12,0);
                \draw [line width=1pt] (3,3)-- (3,0);
                \draw [line width=1pt] (6,3)-- (6,0);
                \draw [line width=1pt,color=ffqqqq] (9,3)-- (9,0);
                \draw [line width=1pt,dash pattern=on 5pt off 5pt] (12,0)-- (15,0);
                \draw [line width=1pt,dash pattern=on 5pt off 5pt] (12,3)-- (15,3);
                \draw [line width=1pt] (12,3)-- (12,0);
                \draw [rotate around={90:(6,1.5)},line width=1.5pt,dash pattern=on 5pt off 5pt] (6,1.5) ellipse (0.8cm and 0.4cm);
                \begin{scriptsize}
                    \draw [fill=wwwwww] (0,3) circle (2pt);
                    \draw [fill=ffqqqq] (3,3) circle (2pt);
                    \draw [fill=ffqqqq] (6,3) circle (2pt);
                    \draw [fill=ffqqqq] (9,3) circle (2pt);
                    \draw [fill=wwwwww] (12,3) circle (2pt);
                    \draw [fill=wwwwww] (0,0) circle (2pt);
                    \draw [fill=wwwwww] (3,0) circle (2pt);
                    \draw [fill=qqqqff] (6,0) circle (2pt);
                    \draw [fill=ffqqqq] (9,0) circle (2pt);
                    \draw [fill=wwwwww] (12,0) circle (2pt);
                    \draw [fill=wwwwww] (15,0) circle (2pt);
                    \draw [fill=wwwwww] (15,3) circle (2pt);
                \end{scriptsize}
            \end{tikzpicture}
            \caption{Case \hyperref[A2]{A2}.}
            \label{fig: Case A2}
        \end{figure}
        
        \task\label{A3} $b_{{j-1}},b_{{j+1}},c_{{j+1}} \in \pi$.
        \begin{figure}[H]
            \hspace{1em}\begin{tikzpicture}[line cap=round,line join=round,>=triangle 45,x=0.35cm,y=0.35cm]         
            \clip(-0.25,-1) rectangle (15.2,4);
                \draw [line width=1pt,dash pattern=on 5pt off 5pt] (0,3)-- (3,3);
                \draw [line width=1pt,color=ffqqqq] (3,3)-- (6,3);
                \draw [line width=1pt] (6,3)-- (9,3);
                \draw [line width=1pt,color=ffqqqq] (9,3)-- (12,3);
                \draw [line width=1pt,dash pattern=on 5pt off 5pt] (0,0)-- (3,0);
                \draw [line width=1pt] (3,0)-- (6,0);
                \draw [line width=1pt] (6,0)-- (9,0);
                \draw [line width=1pt,color=ffqqqq] (9,0)-- (12,0);
                \draw [line width=1pt] (3,3)-- (3,0);
                \draw [line width=1pt] (6,3)-- (6,0);
                \draw [line width=1pt] (9,3)-- (9,0);
                \draw [line width=1pt,dash pattern=on 5pt off 5pt] (12,0)-- (15,0);
                \draw [line width=1pt,dash pattern=on 5pt off 5pt] (12,3)-- (15,3);
                \draw [line width=1pt] (12,3)-- (12,0);
                \draw [rotate around={90:(6,1.5)},line width=1.5pt,dash pattern=on 5pt off 5pt] (6,1.5) ellipse (0.8cm and 0.4cm);
                \begin{scriptsize}
                    \draw [fill=wwwwww] (0,3) circle (2pt);
                    \draw [fill=ffqqqq] (3,3) circle (2pt);
                    \draw [fill=ffqqqq] (6,3) circle (2pt);
                    \draw [fill=ffqqqq] (9,3) circle (2pt);
                    \draw [fill=ffqqqq] (12,3) circle (2pt);
                    \draw [fill=wwwwww] (0,0) circle (2pt);
                    \draw [fill=wwwwww] (3,0) circle (2pt);
                    \draw [fill=qqqqff] (6,0) circle (2pt);
                    \draw [fill=ffqqqq] (9,0) circle (2pt);
                    \draw [fill=ffqqqq] (12,0) circle (2pt);
                    \draw [fill=wwwwww] (15,0) circle (2pt);
                    \draw [fill=wwwwww] (15,3) circle (2pt);
                \end{scriptsize}
            \end{tikzpicture}
            \caption{Case \hyperref[A3]{A3}.}
            \label{fig: Case A3}
        \end{figure}
    \end{tasks}
    
    Clearly, \hyperref[A1]{A1} is not possible since then $\pi$ will contain a subset of the form \hyperref[X1]{X1}. For \hyperref[A2]{A2} and \hyperref[A3]{A3}, we look at $a_{j-2}$. This is because the vertex $p_{j-1,1}$ must remain uncovered; otherwise $\pi$ will contain a subset of the form \hyperref[X2]{X2}, since the only possible edge left to cover $p_{j-1,1}$ will be $c_{j-2}$ and we have $b_{j-1} \in \pi$. Now, if at least one vertex of $a_{j-2}$ is covered, then it can be considered as case \hyperref[A]{A} or \hyperref[B]{B}, and proceed according to the argument provided for them. If both the vertices are uncovered, we look at $a_{j-3}$ and continue the process until all the vertical edges in $L_1$ are exhausted.

     Observe that this exhaustive method will give us pairs of uncovered vertices that can be covered using vertical and horizontal edges, \textit{i.e.}, all the vertices of $L_1$ will get covered and give us perfect matching in $L_1$. Applying the same process to the rest of the components will lead us to conclude that the matching $\pi = \left\{x_{1}, x_{2},\dots, x_{t}\right\}$, can be extended to form a perfect matching, which is a contradiction to our assumption about $\pi$ being a bad matching. Thus, a bad matching of $\twongg$ only has the form \hyperref[X1]{X1} or \hyperref[X2]{X2}.
\end{proof}

    \begin{lemma}\label{Lemma 2}
        Let $\tau\in \mathcal{M}_p\left(\mathcal{G}_{2\times n}\right)$ be a \textit{facet}, \textit{i.e.}, $\tau$ is a perfect matching of $\mathcal{G}_{2\times n}$. For any $i\in \left\{1,2,\dots,n-1\right\}$, $b_{i} \in \tau$ if and only if $c_{i} \in \tau$.
    \end{lemma}

    \begin{proof}
        It is given that $\tau$ is a perfect matching of $\mathcal{G}_{2\times n}$ and let $b_{i} \in \tau$, for some $i\in \left\{1,2,\dots,n-1\right\}$. Let us assume that $c_{i} \notin \tau$. Now, in order to cover $p_{{i,1}}$, we must have either $a_{i}$ or $c_{{i-1}}$ in $\tau$. Clearly, $a_{i} \notin \tau$ as $b_{i} \in \tau$. Moreover, if $c_{{i-1}} \in \tau$, then $\tau$ will contain a subset of the form \hyperref[X2]{X2}, contradicting $\tau$ being a perfect matching. The converse part follows from a similar argument.
    \end{proof}

\subsection{The homotopy type of \texorpdfstring{$(2\times n)$ }~Grid Graph}
    \begin{theorem}\label{thm: 2n grid graph homotopy}
        The perfect matching complex of  $\left(2\times n\right)$ grid graph, $\mathcal{M}_p \left(\mathcal{G}_{2\times n}\right)$ is \textit{contractible}, when $n$ is odd and is homotopy equivalent to $\mathbb{S}^k$, when $n$ is even and written as $n=2k+2$, \textit{i.e.},

        $$\pmcgg \simeq \begin{cases}
            *, & \text{if n is odd};\\
            \mathbb{S}^k, & \text{if n is even and } n=2k+2.
        \end{cases}$$
    \end{theorem}

    In an earlier version of this article, this result was proven by explicitly constructing a sequence of element pairings. The proof presented here, suggested by an anonymous referee and inspired by the earlier proof, is based on the fold lemma for independence complexes (see \Cref{foldlemma}). This approach significantly shortens the proof while effectively constructing an acyclic pairing.

    For $n\geq 2$, consider the graph $X_n$, as follows,
    \begin{align*}
        V\left(X_n\right) &= E\left(\twongg\right)\\
        E\left(X_n\right) &= \left\{\{v,w\} \subset V\left(X_n\right)\ \middle |\ v \text{ and } w \text{ share a common vertex} \right\}\\ &\ \ \ \ \ \bigsqcup \left\{\{v,w\} \subset V\left(X_n\right)\ \middle |\ \begin{array}{c}
         (v,w) = \left(b_i,c_{i+1}\right) \text{ or } (v,w) = \left(c_i,b_{i+1}\right),\\
         \text{ for some } 1\leq i \leq n-2 
        \end{array}
        \right\}
    \end{align*}

    In other words, $X_n$ is the graph formed by the line graph of $\twongg$ with additional edges for bad matchings of $\twongg$ (see \Cref{fig: graph of G(2×5) and X5.} for an example).
    
    \begin{figure}[H]
    \centering
    \begin{tikzpicture}[line cap=round, scale=0.45, line join=round, >=triangle 45, x=1.1cm, y=1.1cm]
        \clip(-1,-1) rectangle (33,4.5);
        \draw [line width=1pt] (0,0)-- (3,0);
        \draw [line width=1pt] (3,0)-- (6,0);
        \draw [line width=1pt] (6,0)-- (9,0);
        \draw [line width=1pt] (9,0)-- (12,0);
        
        \draw [line width=1pt] (0,3)-- (3,3);
        \draw [line width=1pt] (3,3)-- (6,3);
        \draw [line width=1pt] (6,3)-- (9,3);
        \draw [line width=1pt] (9,3)-- (12,3);
        
        \draw [line width=1pt] (0,0)-- (0,3);
        \draw [line width=1pt] (3,0)-- (3,3);
        \draw [line width=1pt] (6,0)-- (6,3);
        \draw [line width=1pt] (9,0)-- (9,3);
        \draw [line width=1pt] (12,0)-- (12,3);


        \draw (1.5,3.5) node[anchor=center] {$b_1$};
        \draw (4.5,3.5) node[anchor=center] {$b_2$};
        \draw (7.5,3.5) node[anchor=center] {$b_3$};
        \draw (10.5,3.5) node[anchor=center] {$b_4$};

        \draw (1.5,-0.5) node[anchor=center] {$c_1$};
        \draw (4.5,-0.5) node[anchor=center] {$c_2$};
        \draw (7.5,-0.5) node[anchor=center] {$c_3$};
        \draw (10.5,-0.5) node[anchor=center] {$c_4$};

        \draw (-0.5,1.5) node[anchor=center] {$a_1$};
        \draw (2.5,1.5) node[anchor=center] {$a_2$};
        \draw (5.5,1.5) node[anchor=center] {$a_3$};
        \draw (8.5,1.5) node[anchor=center] {$a_4$};
        \draw (11.5,1.5) node[anchor=center] {$a_5$};
        
        \begin{scriptsize}
            \draw [fill=xdxdff] (0,0) circle (3.5pt);
            \draw [fill=xdxdff] (3,0) circle (3.5pt);
            \draw [fill=xdxdff] (6,0) circle (3.5pt);
            \draw [fill=xdxdff] (9,0) circle (3.5pt);
            \draw [fill=xdxdff] (12,0) circle (3.5pt);
            \draw [fill=xdxdff] (0,3) circle (3.5pt);
            \draw [fill=xdxdff] (3,3) circle (3.5pt);
            \draw [fill=xdxdff] (6,3) circle (3.5pt);
            \draw [fill=xdxdff] (9,3) circle (3.5pt);
            \draw [fill=xdxdff] (12,3) circle (3.5pt);
        \end{scriptsize}


        \draw [line width=1pt] (18.5,3)-- (16.5,1.5);
        \draw [line width=1pt] (18.5,0)-- (16.5,1.5);
        \draw [line width=1pt] (18.5,0)-- (20.5,1.5);
        \draw [line width=1pt] (20.5,1.5)-- (18.5,3);
        \draw [line width=1pt] (20.5,1.5)-- (22.5,3);
        \draw [line width=1pt] (22.5,3)-- (24.5,1.5);
        \draw [line width=1pt] (24.5,1.5)-- (26.5,3);
        \draw [line width=1pt] (26.5,3)-- (28.5,1.5);
        \draw [line width=1pt] (28.5,1.5)-- (30.5,3);
        \draw [line width=1pt] (30.5,3)-- (32.5,1.5);
        \draw [line width=1pt] (32.5,1.5)-- (30.5,0);
        \draw [line width=1pt] (30.5,0)-- (28.5,1.5);
        \draw [line width=1pt] (28.5,1.5)-- (26.5,0);
        \draw [line width=1pt] (26.5,0)-- (24.5,1.5);
        \draw [line width=1pt] (24.5,1.5)-- (22.5,0);
        \draw [line width=1pt] (22.5,0)-- (20.5,1.5);
        \draw [line width=1pt] (18.5,3)-- (22.5,3);
        \draw [line width=1pt] (22.5,3)-- (26.5,3);
        \draw [line width=1pt] (26.5,3)-- (30.5,3);
        \draw [line width=1pt] (18.5,0)-- (22.5,0);
        \draw [line width=1pt] (22.5,0)-- (26.5,0);
        \draw [line width=1pt] (26.5,0)-- (30.5,0);

        \draw (18.5,3.5) node[anchor=center] {$b_1$};
        \draw (22.5,3.5) node[anchor=center] {$b_2$};
        \draw (26.5,3.5) node[anchor=center] {$b_3$};
        \draw (30.5,3.5) node[anchor=center] {$b_4$};

        \draw (18.5,-0.5) node[anchor=center] {$c_1$};
        \draw (22.5,-0.5) node[anchor=center] {$c_2$};
        \draw (26.5,-0.5) node[anchor=center] {$c_3$};
        \draw (30.5,-0.5) node[anchor=center] {$c_4$};

        \draw (16,1.6) node[anchor=center] {$a_1$};
        \draw (19.8,1.6) node[anchor=center] {$a_2$};
        \draw (23.8,1.6) node[anchor=center] {$a_3$};
        \draw (27.8,1.6) node[anchor=center] {$a_4$};
        \draw (31.8,1.6) node[anchor=center] {$a_5$};

        \draw [line width=1pt] (18.5,3) .. controls (21,2) .. (22.5,0);
        \draw [line width=1pt] (18.5,0) .. controls (21,0.9) .. (22.5,3);
        \draw [line width=1pt] (22.5,3) .. controls (25,2) .. (26.5,0);
        \draw [line width=1pt] (22.5,0) .. controls (25,0.9) .. (26.5,3);
        \draw [line width=1pt] (26.5,3) .. controls (29,2) .. (30.5,0);
        \draw [line width=1pt] (26.5,0) .. controls (29,0.9) .. (30.5,3);

        \begin{scriptsize}
            \draw [fill=xdxdff] (16.5,1.5) circle (3.5pt);
            \draw [fill=xdxdff] (18.5,0) circle (3.5pt);
            \draw [fill=xdxdff] (20.5,1.5) circle (3.5pt);
            \draw [fill=xdxdff] (18.5,3) circle (3.5pt);
            \draw [fill=xdxdff] (22.5,3) circle (3.5pt);
            \draw [fill=xdxdff] (24.5,1.5) circle (3.5pt);
            \draw [fill=xdxdff] (26.5,3) circle (3.5pt);
            \draw [fill=xdxdff] (28.5,1.5) circle (3.5pt);
            \draw [fill=xdxdff] (30.5,3) circle (3.5pt);
            \draw [fill=xdxdff] (32.5,1.5) circle (3.5pt);
            \draw [fill=xdxdff] (30.5,0) circle (3.5pt);
            \draw [fill=xdxdff] (26.5,0) circle (3.5pt);
            \draw [fill=xdxdff] (22.5,0) circle (3.5pt);
        \end{scriptsize}
    \end{tikzpicture}
    \caption{The graphs $\mathcal{G}_{2\times 5}$ (left) and $X_5$ (right).}
    \label{fig: graph of G(2×5) and X5.}
    \end{figure}
 
    In order to understand the proof, we first discuss the independent sets and neighborhood of a vertex in a graph $G$. An \textit{independent set}, $I$ is a subset of the vertex set $V(G)$ such that no two vertices in $I$ are adjacent. The neighborhood of a vertex $v$ in $G$, denoted $N_G(v)$ is the set $\{w \in V(G)\ | \ (v,w)\in E(G)\}$. If the graph is clear from context, we write $N(v)$ in place of $N_G(v)$.
    
    \begin{defn}
        For an undirected graph $G$, the \textit{independence complex} of $G$, denoted $\ind(G)$, is a simplicial complex with $V(G)$ as the set of vertices and $I\subseteq V(G)$ is a face of $\ind(G)$ if and only if $I$ is an independent set of $G$.
    \end{defn}
    
    The following lemma (known as the \textit{fold lemma} for the independence complexes) states that the removal of a vertex from the graph under certain conditions preserves the homotopy type of the independence complex.

    \begin{lemma}(\cite[{Lemma~2.4}]{engstrom_clawfree})\label{foldlemma}
        For two distinct vertices $v$ and $w$ of a graph $G$ with $N(v) \subseteq N(w)$, $\ind(G)$ collapses onto $\ind\left(G \setminus \left\{w\right\}\right)$
    \end{lemma}
    
    We now discuss the proof of \Cref{thm: 2n grid graph homotopy}.
    
    \begin{proof}[Proof of \Cref{thm: 2n grid graph homotopy}]
        Note that, using \Cref{Lemma 1}, we have $$\ind\left(X_n\right) = \pmcgg.$$

        The proof is by induction on $n$. The result follows immediately for the case when $n=1$ and $n=2$, \textit{i.e.}, $\ind(X_1) \simeq *$ and $\ind(X_2) \simeq \mathbb{S}^0$. For $n\geq 3$, observe the following:
        \begin{enumerate}
            \item In $X_n$, $N(a_n) \subset N(b_{n-2})$, therefore, using \Cref{foldlemma} we have, $$\ind\left(X_n\right) \simeq \ind\left(X_n \setminus \left\{b_{n-2}\right\}\right);$$
            \begin{figure}[H]
            \centering
            \begin{tikzpicture}[line cap=round, scale=0.42, line join=round, >=triangle 45, x=1cm, y=1.2cm]
                \clip(-0.5,-1) rectangle (37.5,4.5);
                \draw [line width=1pt] (2,0)-- (4,1.5);
                \draw [line width=1pt] (4,1.5)-- (2,3);
                \draw [line width=1pt] (4,1.5)-- (6,3);
                \draw [line width=1pt] (6,3)-- (8,1.5);
                \draw [line width=1pt] (8,1.5)-- (10,3);
                \draw [line width=1pt] (10,3)-- (12,1.5);
                \draw [line width=1pt] (12,1.5)-- (14,3);
                \draw [line width=1pt] (14,3)-- (16,1.5);
                \draw [line width=1pt] (16,1.5)-- (14,0);
                \draw [line width=1pt] (14,0)-- (12,1.5);
                \draw [line width=1pt] (12,1.5)-- (10,0);
                \draw [line width=1pt] (10,0)-- (8,1.5);
                \draw [line width=1pt] (8,1.5)-- (6,0);
                \draw [line width=1pt] (6,0)-- (4,1.5);
                \draw [line width=1pt] (2,3)-- (6,3);
                \draw [line width=1pt] (6,3)-- (10,3);
                \draw [line width=1pt] (10,3)-- (14,3);
                \draw [line width=1pt] (2,0)-- (6,0);
                \draw [line width=1pt] (6,0)-- (10,0);
                \draw [line width=1pt] (10,0)-- (14,0);
        
                \draw [line width=1pt,dash pattern=on 5pt off 5pt] (0,3)-- (2,3);
                \draw [line width=1pt,dash pattern=on 5pt off 5pt] (0,0)-- (2,0);
        
                \draw [line width=1pt] (2,3) .. controls (4.5,2) .. (6,0);
                \draw [line width=1pt] (2,0) .. controls (4.5,0.9) .. (6,3);
                \draw [line width=1pt] (6,3) .. controls (8.5,2) .. (10,0);
                \draw [line width=1pt] (6,0) .. controls (8.5,0.9) .. (10,3);
                \draw [line width=1pt] (10,3) .. controls (12.5,2) .. (14,0);
                \draw [line width=1pt] (10,0) .. controls (12.5,0.9) .. (14,3);
        
                \draw (2.4,3.5) node[anchor=center] {$b_{n-4}$};
                \draw (6.4,3.5) node[anchor=center] {$b_{n-3}$};
                \draw (10.4,3.5) node[anchor=center] {$b_{n-2}$};
                \draw (14.4,3.5) node[anchor=center] {$b_{n-1}$};
        
                \draw (2.4,-0.5) node[anchor=center] {$c_{n-4}$};
                \draw (6.4,-0.5) node[anchor=center] {$c_{n-3}$};
                \draw (10.4,-0.5) node[anchor=center] {$c_{n-2}$};
                \draw (14.4,-0.5) node[anchor=center] {$c_{n-1}$};
        
                \draw (2.5,1.6) node[anchor=center] {$a_{n-3}$};
                \draw (6.5,1.6) node[anchor=center] {$a_{n-2}$};
                \draw (10.5,1.6) node[anchor=center] {$a_{n-1}$};
                \draw (15.1,1.6) node[anchor=center] {$a_{n}$};
                
                \begin{scriptsize}
                    \draw [fill=xdxdff] (2,0) circle (3.5pt);
                    \draw [fill=ududff] (4,1.5) circle (3.5pt);
                    \draw [fill=ududff] (2,3) circle (3.5pt);
                    \draw [fill=ududff] (6,3) circle (3.5pt);
                    \draw [fill=ududff] (8,1.5) circle (3.5pt);
                    \draw [fill=ududff] (10,3) circle (3.5pt);
                    \draw [fill=ududff] (12,1.5) circle (3.5pt);
                    \draw [fill=ududff] (14,3) circle (3.5pt);
                    \draw [fill=ududff] (16,1.5) circle (3.5pt);
                    \draw [fill=ududff] (14,0) circle (3.5pt);
                    \draw [fill=ududff] (10,0) circle (3.5pt);
                    \draw [fill=xdxdff] (6,0) circle (3.5pt);
                \end{scriptsize}
        
        
                \draw [line width=1pt, -stealth] (17.5,1.5)-- (19.5,1.5);
        
                \draw [line width=1pt] (22.5,0)-- (24.5,1.5);
                \draw [line width=1pt] (24.5,1.5)-- (22.5,3);
                \draw [line width=1pt] (24.5,1.5)-- (26.5,3);
                \draw [line width=1pt] (26.5,3)-- (28.5,1.5);
                \draw [line width=1pt] (32.5,1.5)-- (34.5,3);
                \draw [line width=1pt] (34.5,3)-- (36.5,1.5);
                \draw [line width=1pt] (36.5,1.5)-- (34.5,0);
                \draw [line width=1pt] (34.5,0)-- (32.5,1.5);
                \draw [line width=1pt] (32.5,1.5)-- (30.5,0);
                \draw [line width=1pt] (30.5,0)-- (28.5,1.5);
                \draw [line width=1pt] (28.5,1.5)-- (26.5,0);
                \draw [line width=1pt] (26.5,0)-- (24.5,1.5);
                \draw [line width=1pt] (22.5,3)-- (26.5,3);
                \draw [line width=1pt] (22.5,0)-- (26.5,0);
                \draw [line width=1pt] (26.5,0)-- (30.5,0);
                \draw [line width=1pt] (30.5,0)-- (34.5,0);
        
                \draw [line width=1pt,dash pattern=on 5pt off 5pt] (20.5,3)-- (22.5,3);
                \draw [line width=1pt,dash pattern=on 5pt off 5pt] (20.5,0)-- (22.5,0);
        
                \draw [line width=1pt] (22.5,3) .. controls (25,2) .. (26.5,0);
                \draw [line width=1pt] (22.5,0) .. controls (25,0.9) .. (26.5,3);
                \draw [line width=1pt] (26.5,3) .. controls (29,2) .. (30.5,0);
                \draw [line width=1pt] (30.5,0) .. controls (33,0.9) .. (34.5,3);

                \draw (22.9,3.5) node[anchor=center] {$b_{n-4}$};
                \draw (26.9,3.5) node[anchor=center] {$b_{n-3}$};
                \draw (30.9,3.6) node[anchor=center] {$b_{n-2}$};
                \draw (34.9,3.5) node[anchor=center] {$b_{n-1}$};
        
                \draw (22.9,-0.5) node[anchor=center] {$c_{n-4}$};
                \draw (26.9,-0.5) node[anchor=center] {$c_{n-3}$};
                \draw (30.9,-0.5) node[anchor=center] {$c_{n-2}$};
                \draw (34.9,-0.5) node[anchor=center] {$c_{n-1}$};
        
                \draw (23,1.6) node[anchor=center] {$a_{n-3}$};
                \draw (27,1.6) node[anchor=center] {$a_{n-2}$};
                \draw (31.3,1.7) node[anchor=center] {$a_{n-1}$};
                \draw (35.6,1.6) node[anchor=center] {$a_{n}$};
        
                \begin{scriptsize}
                    \draw [fill=ududff] (22.5,0) circle (3.5pt);
                    \draw [fill=ududff] (24.5,1.5) circle (3.5pt);
                    \draw [fill=ududff] (22.5,3) circle (3.5pt);
                    \draw [fill=ududff] (26.5,3) circle (3.5pt);
                    \draw [fill=ududff] (28.5,1.5) circle (3.5pt);
                    \draw (30.5,3) circle (4pt);
                    \draw [fill=ududff] (32.5,1.5) circle (3.5pt);
                    \draw [fill=ududff] (34.5,3) circle (3.5pt);
                    \draw [fill=ududff] (36.5,1.5) circle (3.5pt);
                    \draw [fill=ududff] (34.5,0) circle (3.5pt);
                    \draw [fill=ududff] (30.5,0) circle (3.5pt);
                    \draw [fill=xdxdff] (26.5,0) circle (3.5pt);
                \end{scriptsize}
            \end{tikzpicture}
            \label{fig: Xn to Xn - b(n-2)}
            \caption{Since $N(a_n) \subset N(b_{n-2})$ in $X_n$, $b_{n-2}$ can be deleted from $X_n$.}
            \end{figure}
            
            \item Similarly, in $X_n \setminus \left\{b_{n-2}\right\}$, $N(a_n) \subset N(c_{n-2})$, therefore, using \Cref{foldlemma} we have, $$\ind\left(X_n\setminus \left\{b_{n-2}\right\}\right) \simeq \ind\left(X_n \setminus \left\{b_{n-2}, c_{n-2}\right\}\right);$$
            \begin{figure}[H]
            \centering
            \begin{tikzpicture}[line cap=round, scale=0.42, line join=round, >=triangle 45, x=1cm, y=1.2cm]
                \clip(-0.5,-1) rectangle (37.5,4.5);
                \draw [line width=1pt] (2,0)-- (4,1.5);
                \draw [line width=1pt] (4,1.5)-- (2,3);
                \draw [line width=1pt] (4,1.5)-- (6,3);
                \draw [line width=1pt] (6,3)-- (8,1.5);
                \draw [line width=1pt] (12,1.5)-- (14,3);
                \draw [line width=1pt] (14,3)-- (16,1.5);
                \draw [line width=1pt] (16,1.5)-- (14,0);
                \draw [line width=1pt] (14,0)-- (12,1.5);
                \draw [line width=1pt] (12,1.5)-- (10,0);
                \draw [line width=1pt] (10,0)-- (8,1.5);
                \draw [line width=1pt] (8,1.5)-- (6,0);
                \draw [line width=1pt] (6,0)-- (4,1.5);
                \draw [line width=1pt] (2,3)-- (6,3);
                \draw [line width=1pt] (2,0)-- (6,0);
                \draw [line width=1pt] (6,0)-- (10,0);
                \draw [line width=1pt] (10,0)-- (14,0);
        
                \draw [line width=1pt,dash pattern=on 5pt off 5pt] (0,3)-- (2,3);
                \draw [line width=1pt,dash pattern=on 5pt off 5pt] (0,0)-- (2,0);
        
                \draw [line width=1pt] (2,3) .. controls (4.5,2) .. (6,0);
                \draw [line width=1pt] (2,0) .. controls (4.5,0.9) .. (6,3);
                \draw [line width=1pt] (6,3) .. controls (8.5,2) .. (10,0);
                \draw [line width=1pt] (10,0) .. controls (12.5,0.9) .. (14,3);
        
                \draw (2.4,3.5) node[anchor=center] {$b_{n-4}$};
                \draw (6.4,3.5) node[anchor=center] {$b_{n-3}$};
                \draw (10.4,3.6) node[anchor=center] {$b_{n-2}$};
                \draw (14.4,3.5) node[anchor=center] {$b_{n-1}$};
        
                \draw (2.4,-0.5) node[anchor=center] {$c_{n-4}$};
                \draw (6.4,-0.5) node[anchor=center] {$c_{n-3}$};
                \draw (10.4,-0.5) node[anchor=center] {$c_{n-2}$};
                \draw (14.4,-0.5) node[anchor=center] {$c_{n-1}$};
        
                \draw (2.5,1.6) node[anchor=center] {$a_{n-3}$};
                \draw (6.5,1.6) node[anchor=center] {$a_{n-2}$};
                \draw (10.8,1.7) node[anchor=center] {$a_{n-1}$};
                \draw (15.1,1.6) node[anchor=center] {$a_{n}$};
                
                \begin{scriptsize}
                    \draw [fill=ududff] (2,0) circle (3.5pt);
                    \draw [fill=ududff] (4,1.5) circle (3.5pt);
                    \draw [fill=ududff] (2,3) circle (3.5pt);
                    \draw [fill=ududff] (6,3) circle (3.5pt);
                    \draw [fill=ududff] (8,1.5) circle (3.5pt);
                    \draw (10,3) circle (4pt);
                    \draw [fill=ududff] (12,1.5) circle (3.5pt);
                    \draw [fill=ududff] (14,3) circle (3.5pt);
                    \draw [fill=ududff] (16,1.5) circle (3.5pt);
                    \draw [fill=ududff] (14,0) circle (3.5pt);
                    \draw [fill=ududff] (10,0) circle (3.5pt);
                    \draw [fill=xdxdff] (6,0) circle (3.5pt);
                \end{scriptsize}
        
        
                \draw [line width=1pt, -stealth] (17.5,1.5)-- (19.5,1.5);
        
                \draw [line width=1pt] (22.5,0)-- (24.5,1.5);
                \draw [line width=1pt] (24.5,1.5)-- (22.5,3);
                \draw [line width=1pt] (24.5,1.5)-- (26.5,3);
                \draw [line width=1pt] (26.5,3)-- (28.5,1.5);
                \draw [line width=1pt] (32.5,1.5)-- (34.5,3);
                \draw [line width=1pt] (34.5,3)-- (36.5,1.5);
                \draw [line width=1pt] (36.5,1.5)-- (34.5,0);
                \draw [line width=1pt] (34.5,0)-- (32.5,1.5);
                \draw [line width=1pt] (28.5,1.5)-- (26.5,0);
                \draw [line width=1pt] (26.5,0)-- (24.5,1.5);
                \draw [line width=1pt] (22.5,3)-- (26.5,3);
                \draw [line width=1pt] (22.5,0)-- (26.5,0);
        
                \draw [line width=1pt,dash pattern=on 5pt off 5pt] (20.5,3)-- (22.5,3);
                \draw [line width=1pt,dash pattern=on 5pt off 5pt] (20.5,0)-- (22.5,0);
        
                \draw [line width=1pt] (22.5,3) .. controls (25,2) .. (26.5,0);
                \draw [line width=1pt] (22.5,0) .. controls (25,0.9) .. (26.5,3);

                \draw (22.9,3.5) node[anchor=center] {$b_{n-4}$};
                \draw (26.9,3.5) node[anchor=center] {$b_{n-3}$};
                \draw (30.9,3.6) node[anchor=center] {$b_{n-2}$};
                \draw (34.9,3.5) node[anchor=center] {$b_{n-1}$};
        
                \draw (22.9,-0.5) node[anchor=center] {$c_{n-4}$};
                \draw (26.9,-0.5) node[anchor=center] {$c_{n-3}$};
                \draw (30.9,-0.6) node[anchor=center] {$c_{n-2}$};
                \draw (34.9,-0.5) node[anchor=center] {$c_{n-1}$};
        
                \draw (23,1.6) node[anchor=center] {$a_{n-3}$};
                \draw (27,1.6) node[anchor=center] {$a_{n-2}$};
                \draw (31.3,1.6) node[anchor=center] {$a_{n-1}$};
                \draw (35.6,1.6) node[anchor=center] {$a_{n}$};
        
                \begin{scriptsize}
                    \draw [fill=ududff] (22.5,0) circle (3.5pt);
                    \draw [fill=ududff] (24.5,1.5) circle (3.5pt);
                    \draw [fill=ududff] (22.5,3) circle (3.5pt);
                    \draw [fill=ududff] (26.5,3) circle (3.5pt);
                    \draw [fill=ududff] (28.5,1.5) circle (3.5pt);
                    \draw (30.5,3) circle (4pt);
                    \draw [fill=ududff] (32.5,1.5) circle (3.5pt);
                    \draw [fill=ududff] (34.5,3) circle (3.5pt);
                    \draw [fill=ududff] (36.5,1.5) circle (3.5pt);
                    \draw [fill=ududff] (34.5,0) circle (3.5pt);
                    \draw (30.5,0) circle (4pt);
                    \draw [fill=xdxdff] (26.5,0) circle (3.5pt);
                \end{scriptsize}
            \end{tikzpicture}
            \label{fig: Xn - b(n-2) to Xn - [b(n-2), c(n-2)]}
            \caption{\centering Since $N(a_n) \subset N(c_{n-2})$ in $X_n \setminus \left\{b_{n-2}\right\}$, $c_{n-2}$ can be deleted from $X_n \setminus \left\{b_{n-2}\right\}$.}
            \end{figure}
            
            \item At last, in $X_n \setminus \left\{b_{n-2}, c_{n-2}\right\}$, $N(a_n) \subset N(a_{n-1})$, therefore, using \Cref{foldlemma} we have, $$\ind\left(X_n\setminus \left\{b_{n-2},c_{n-2}\right\}\right) \simeq \ind\left(X_n \setminus \left\{b_{n-2}, c_{n-2},a_{n-1}\right\}\right);$$
            \begin{figure}[H]
            \centering
            \begin{tikzpicture}[line cap=round, scale=0.42, line join=round, >=triangle 45, x=1cm, y=1.2cm]
                \clip(-0.5,-1) rectangle (37.5,4.5);
                \draw [line width=1pt] (2,0)-- (4,1.5);
                \draw [line width=1pt] (4,1.5)-- (2,3);
                \draw [line width=1pt] (4,1.5)-- (6,3);
                \draw [line width=1pt] (6,3)-- (8,1.5);
                \draw [line width=1pt] (12,1.5)-- (14,3);
                \draw [line width=1pt] (14,3)-- (16,1.5);
                \draw [line width=1pt] (16,1.5)-- (14,0);
                \draw [line width=1pt] (14,0)-- (12,1.5);
                \draw [line width=1pt] (8,1.5)-- (6,0);
                \draw [line width=1pt] (6,0)-- (4,1.5);
                \draw [line width=1pt] (2,3)-- (6,3);
                \draw [line width=1pt] (2,0)-- (6,0);
        
                \draw [line width=1pt,dash pattern=on 5pt off 5pt] (0,3)-- (2,3);
                \draw [line width=1pt,dash pattern=on 5pt off 5pt] (0,0)-- (2,0);
        
                \draw [line width=1pt] (2,3) .. controls (4.5,2) .. (6,0);
                \draw [line width=1pt] (2,0) .. controls (4.5,0.9) .. (6,3);
        
                \draw (2.4,3.5) node[anchor=center] {$b_{n-4}$};
                \draw (6.4,3.5) node[anchor=center] {$b_{n-3}$};
                \draw (10.4,3.6) node[anchor=center] {$b_{n-2}$};
                \draw (14.4,3.5) node[anchor=center] {$b_{n-1}$};
        
                \draw (2.4,-0.5) node[anchor=center] {$c_{n-4}$};
                \draw (6.4,-0.5) node[anchor=center] {$c_{n-3}$};
                \draw (10.4,-0.6) node[anchor=center] {$c_{n-2}$};
                \draw (14.4,-0.5) node[anchor=center] {$c_{n-1}$};
        
                \draw (2.5,1.6) node[anchor=center] {$a_{n-3}$};
                \draw (6.5,1.6) node[anchor=center] {$a_{n-2}$};
                \draw (10.8,1.7) node[anchor=center] {$a_{n-1}$};
                \draw (15.1,1.6) node[anchor=center] {$a_{n}$};
                
                \begin{scriptsize}
                    \draw [fill=ududff] (2,0) circle (3.5pt);
                    \draw [fill=ududff] (4,1.5) circle (3.5pt);
                    \draw [fill=ududff] (2,3) circle (3.5pt);
                    \draw [fill=ududff] (6,3) circle (3.5pt);
                    \draw [fill=ududff] (8,1.5) circle (3.5pt);
                    \draw (10,3) circle (4pt);
                    \draw [fill=ududff] (12,1.5) circle (3.5pt);
                    \draw [fill=ududff] (14,3) circle (3.5pt);
                    \draw [fill=ududff] (16,1.5) circle (3.5pt);
                    \draw [fill=ududff] (14,0) circle (3.5pt);
                    \draw (10,0) circle (4pt);
                    \draw [fill=xdxdff] (6,0) circle (3.5pt);
                \end{scriptsize}
        
        
                \draw [line width=1pt, -stealth] (17.5,1.5)-- (19.5,1.5);
        
                \draw [line width=1pt] (22.5,0)-- (24.5,1.5);
                \draw [line width=1pt] (24.5,1.5)-- (22.5,3);
                \draw [line width=1pt] (24.5,1.5)-- (26.5,3);
                \draw [line width=1pt] (26.5,3)-- (28.5,1.5);
                \draw [line width=1pt] (34.5,3)-- (36.5,1.5);
                \draw [line width=1pt] (36.5,1.5)-- (34.5,0);
                \draw [line width=1pt] (28.5,1.5)-- (26.5,0);
                \draw [line width=1pt] (26.5,0)-- (24.5,1.5);
                \draw [line width=1pt] (22.5,3)-- (26.5,3);
                \draw [line width=1pt] (22.5,0)-- (26.5,0);
        
                \draw [line width=1pt,dash pattern=on 5pt off 5pt] (20.5,3)-- (22.5,3);
                \draw [line width=1pt,dash pattern=on 5pt off 5pt] (20.5,0)-- (22.5,0);
        
                \draw [line width=1pt] (22.5,3) .. controls (25,2) .. (26.5,0);
                \draw [line width=1pt] (22.5,0) .. controls (25,0.9) .. (26.5,3);

                \draw (22.9,3.5) node[anchor=center] {$b_{n-4}$};
                \draw (26.9,3.5) node[anchor=center] {$b_{n-3}$};
                \draw (30.9,3.6) node[anchor=center] {$b_{n-2}$};
                \draw (34.9,3.5) node[anchor=center] {$b_{n-1}$};
        
                \draw (22.9,-0.5) node[anchor=center] {$c_{n-4}$};
                \draw (26.9,-0.5) node[anchor=center] {$c_{n-3}$};
                \draw (30.9,-0.6) node[anchor=center] {$c_{n-2}$};
                \draw (34.9,-0.5) node[anchor=center] {$c_{n-1}$};
        
                \draw (23,1.6) node[anchor=center] {$a_{n-3}$};
                \draw (27,1.6) node[anchor=center] {$a_{n-2}$};
                \draw (31.3,1.6) node[anchor=center] {$a_{n-1}$};
                \draw (35.6,1.6) node[anchor=center] {$a_{n}$};
        
                \begin{scriptsize}
                    \draw [fill=ududff] (22.5,0) circle (3.5pt);
                    \draw [fill=ududff] (24.5,1.5) circle (3.5pt);
                    \draw [fill=ududff] (22.5,3) circle (3.5pt);
                    \draw [fill=ududff] (26.5,3) circle (3.5pt);
                    \draw [fill=ududff] (28.5,1.5) circle (3.5pt);
                    \draw (30.5,3) circle (4pt);
                    \draw (32.5,1.5) circle (4pt);
                    \draw [fill=ududff] (34.5,3) circle (3.5pt);
                    \draw [fill=ududff] (36.5,1.5) circle (3.5pt);
                    \draw [fill=ududff] (34.5,0) circle (3.5pt);
                    \draw (30.5,0) circle (4pt);
                    \draw [fill=xdxdff] (26.5,0) circle (3.5pt);
                \end{scriptsize}
            \end{tikzpicture}
            \label{fig: Xn - [b(n-2), c(n-2)] to Xn - [b(n-2), c(n-2), a(n-1)]}
            \caption{\centering Since $N(a_n) \subset N(a_{n-1})$ in $X_n \setminus \left\{ b_{n-2}, c_{n-2}\right\}$, $a_{n-1}$ can be deleted from $X_n \setminus \left\{ b_{n-2}, c_{n-2}\right\}$.}
            \end{figure}
        \end{enumerate}
        Additionally, it's evident that, $$\ind\left(X_n \setminus \left\{b_{n-2}, c_{n-2},a_{n-1}\right\}\right) \simeq \ind(X_{n-2} \sqcup P_3).$$
        Here, $P_3$ is a \textit{path graph} on 3 vertices. It is easy to observe that $\ind(P_3) \simeq S^0$. Therefore, using induction, we get the following.
        \begin{align*}
            \pmcgg &\simeq \ind\left(X_{n-2} \sqcup P_3\right)\\
            &\simeq \Sigma \ind(X_{n-2})\\
            &\simeq \begin{cases}
                \Sigma *, &\text{ if } n\text{ is odd};\\
                \Sigma \mathbb{S}^{k-1}, & \text{ if n is even and } n=2k+2.
            \end{cases}
        \end{align*}
        Hence,
        \begin{align*}
            \pmcgg &\simeq \begin{cases}
                *, &\text{ if } n\text{ is odd};\\
                \mathbb{S}^{k}, & \text{ if n is even and } n=2k+2.
            \end{cases}
        \end{align*}
        This completes the proof.
    \end{proof}

\section{The general line tiling of even-sided polygon}\label{Section 4}

    We got the homotopy type of the perfect matching complexes of the $\left(2\times n\right)$-grid graphs in the previous section. As stated earlier, the $\left(2\times n\right)$-grid graph can be visualized as square line tiling. In the next two sections, we determine the homotopy type of the perfect matching complexes of general polygonal line tiling. This section is dedicated to the case of $2n-$gons (polygons with an even number of sides) for $n \geq 3$ only because of the nice symmetry we get in this case. In fact, we show that the perfect matching complex of these graphs is contractible; for whatever size of the $2n-$gon, for $n \geq 3$, we choose, and for any number (at least two) of $2n-$gons, we attach. However, before discussing the proof of what we just stated, let us first understand what we mean by the general line tiling of $2n-$gons. Line tiling of $2n-$gons means polygons with $2n$ edges are attached adjacent to each other along the parallel edges, forming a line. By general line tiling, we mean that we are not restricting ourselves to taking a certain number of the polygons in this line tiling.

    \subsection{Labelling of edges, vertices and polygons in the line tiling of \texorpdfstring{$2n-$}~gons}
    Let $\mathcal{E}_{n,k}$ denote the general $k-2n-$gon line tiling, where $k\geq2$ and $n\geq3$, \textit{i.e.}, $k$ number of $2n-$gons are there in the line tiling. Note here that the edges along which the polygons are attached are parallel; thus, the number of edges above and below them are equal. We label the $2n-$gons, for $n\geq3$ as $P_j$, for $j=1,2,\dots,k$, from left to right, respectively. For each $P_j$, the vertex set and edge set are
    \begin{align*}
        V\left(P_{j}\right) &= \left\{u_{{j,t}}, v_{{j,t}}\ \middle|\ t\in\left\{1,2,\dots,n\right\}\right\};\\
        E\left(P_{j}\right) &= \left\{a_{j}, b_{{j,t}},c_{{j,t}}\ \middle|\ t\in\left\{1,2,\dots,n\right\}\right\}. 
    \end{align*}

    \begin{figure}[H]
    \vspace{-1em}
        \centering
        \hspace*{3em}
            \begin{tikzpicture}[line cap=round,line join=round,>=triangle 45,x=0.6cm,y=0.4cm]
            \clip(-1.5,-6) rectangle (16,10);
            
            \draw [line width=1pt] (0,3)-- (0,0);
            \draw [line width=1pt] (0,0)-- (1.3016512173526733,-2.702906603707256);
            \draw [line width=1pt] (1.3016512173526733,-2.702906603707256)-- (3.6471456647567626,-4.573376009283456);
            \draw [line width=1pt] (3.6471456647567626,-4.573376009283456)-- (6.571929401302232,-5.240938811152399);
            \draw [line width=1pt,dash pattern=on 4pt off 4pt] (6.571929401302232,-5.240938811152399)-- (9.496713137847703,-4.573376009283457);
            \draw [line width=1pt] (9.496713137847703,-4.573376009283457)-- (11.842207585251792,-2.702906603707257);
            \draw [line width=1pt] (11.842207585251792,-2.702906603707257)-- (13.143858802604466,0);
            \draw [line width=1pt] (13.143858802604466,0)-- (13.143858802604466,3);
            \draw [line width=1pt] (13.143858802604466,3)-- (11.842207585251794,5.702906603707255);
            \draw [line width=1pt] (11.842207585251794,5.702906603707255)-- (9.496713137847705,7.573376009283456);
            \draw [line width=1pt,dash pattern=on 4pt off 4pt] (9.496713137847705,7.573376009283456)-- (6.571929401302235,8.2409388111524);
            \draw [line width=1pt] (6.571929401302235,8.2409388111524)-- (3.6471456647567644,7.573376009283457);
            \draw [line width=1pt] (3.6471456647567644,7.573376009283457)-- (1.301651217352676,5.702906603707258);
            \draw [line width=1pt] (1.301651217352676,5.702906603707258)-- (0,3);
            \draw (5.5,2.5) node[anchor=north west] {\Huge$P_{j}$};
            \draw (-1,2.2) node[anchor=north west] {\large$a_{j}$};
            \draw (13.2,2.2) node[anchor=north west] {\large$a_{j+1}$};
            \draw (12.4,5.4) node[anchor=north west] {\large$b_{j,n}$};
            \draw (10.3,8) node[anchor=north west] {\large$b_{j,n-1}$};
            \draw (4.3,9.3) node[anchor=north west] {\large$b_{j,3}$};
            \draw (1.4,8) node[anchor=north west] {\large$b_{j,2}$};
            \draw (-0.5,5.8) node[anchor=north west] {\large$b_{j,1}$};
            \draw (-0.6,-0.8) node[anchor=north west] {\large$c_{j,1}$};
            \draw (1.3,-3.4) node[anchor=north west] {\large$c_{j,2}$};
            \draw (4.4,-4.8) node[anchor=north west] {\large$c_{j,3}$};
            \draw (10.5,-3.4) node[anchor=north west] {\large$c_{j,n-2}$};
            \draw (12.5,-1) node[anchor=north west] {\large$c_{j,n-1}$};
            \draw [color=xdxdff](-1.4,3.8) node[anchor=north west] {\large$u_{j,1}$};
            \draw [color=xdxdff](2.5,8.8) node[anchor=north west] {\large$u_{j,3}$};
            \draw [color=xdxdff](0,6.9) node[anchor=north west] {\large$u_{j,2}$};
            \draw [color=xdxdff](9.1,8.8) node[anchor=north west] {\large$u_{j,n-1}$};
            \draw [color=xdxdff](11.8,6.5) node[anchor=north west] {\large$u_{j,n}$};
            \draw [color=xdxdff](13.2,3.5) node[anchor=north west] {\large$u_{j+1,1}$};
            \draw [color=xdxdff](-1.2,0.8) node[anchor=north west] {\large$l_{j,1}$};
            \draw [color=xdxdff](0,-2) node[anchor=north west] {\large$l_{j,2}$};
            \draw [color=xdxdff](2.8,-4.2) node[anchor=north west] {\large$l_{j,3}$};
            \draw [color=xdxdff](9.4,-4.5) node[anchor=north west] {\large$l_{j,n-2}$};
            \draw [color=xdxdff](11.9,-2.3) node[anchor=north west] {\large$l_{j,n-1}$};
            \draw [color=xdxdff](13.2,0.5) node[anchor=north west] {\large$l_{j+1,1}$};
            \begin{scriptsize}
            \draw [fill=xdxdff] (0,3) circle (1.5pt);
            \draw [fill=xdxdff] (0,0) circle (1.5pt);
            \draw [fill=xdxdff] (1.3016512173526733,-2.702906603707256) circle (1.5pt);
            \draw [fill=xdxdff] (3.6471456647567626,-4.573376009283456) circle (1.5pt);
            \draw [fill=xdxdff] (6.571929401302232,-5.240938811152399) circle (1.5pt);
            \draw [fill=xdxdff] (9.496713137847703,-4.573376009283457) circle (1.5pt);
            \draw [fill=xdxdff] (11.842207585251792,-2.702906603707257) circle (1.5pt);
            \draw [fill=xdxdff] (13.143858802604466,0) circle (1.5pt);
            \draw [fill=xdxdff] (13.143858802604466,3) circle (1.5pt);
            \draw [fill=xdxdff] (11.842207585251794,5.702906603707255) circle (1.5pt);
            \draw [fill=xdxdff] (9.496713137847705,7.573376009283456) circle (1.5pt);
            \draw [fill=xdxdff] (6.571929401302235,8.2409388111524) circle (1.5pt);
            \draw [fill=xdxdff] (3.6471456647567644,7.573376009283457) circle (1.5pt);
            \draw [fill=xdxdff] (1.301651217352676,5.702906603707258) circle (1.5pt);
            \end{scriptsize}
            \end{tikzpicture}
        \vspace{-0.5em}
        \caption{Labelling of edges and vertices of a $2n-$gon in line tiling}
        \label{fig: Labelling}
    \end{figure}
    
    Here, $a_{{j}}$'s are the edges along which polygons are attached, and their endpoints are $u_{{j,1}}$ and $l_{{j,1}}$. For the edges and vertices lying above the edges $a_{{j}}$ and $a_{{j+1}}$, the vertices $u_{{j,s}}$ and $u_{{j,s+1}}$ are the endpoints of the edge $b_{{j,s}}$, for $1 \leq s \leq n-1$ and the vertices $u_{{j,n-1}}$ and $u_{{j+1,1}}$ are the endpoints of the edge $b_{{j,n-1}}$. Similarly, for the edges and vertices lying below, the vertices $l_{{j,s}}$ and $l_{{j,s+1}}$ are the endpoints of the edge $c_{{j,s}}$, for $1 \leq s \leq n-1$, and the vertices $l_{{j,n-1}}$ and $l_{{j+1,1}}$ are the endpoints of the edge $c_{{j,n-1}}$.

    \subsection{Homotopy type of the perfect matching complex of general line tiling of even-sided polygon}

    We now discuss the main result of this section.

    \begin{theorem}\label{thm: homotopy of 2n-gon}
        The perfect matching complex of line tiling of $k-2n-$gon, \textit{i.e.}, $\pmplt$, is contractible, \textit{i.e.},
        $$\pmplt \simeq *.$$
    \end{theorem}

    \begin{proof}
        To prove this result, we perform the following three element pairings to prove this result:

        \begin{itemize}
            \item \textbf{When $n$ is even}, we start element pairing using $a_{1}$, followed by $b_{{1,1}}$ and $c_{2,n-2}$, respectively.
            \item \textbf{When $n$ is odd}, we start element pairing using $a_{1}$, followed by $b_{{1,1}}$ and $c_{{2,n-1}}$, respectively.
        \end{itemize}
        
        In both cases, we are left with no critical cells at the end; hence, this will prove $\pmplt$ is contractible. A visualization of this result can be understood using the \Cref{fig: Hexagon line tiling}, \Cref{fig: Octagon line tiling} and \Cref{Decagon line tiling}, in which we perform the element pairing using the edges marked with \textcolor{red}{red}, \textcolor{green}{green} and \textcolor{cyan}{blue}, respectively.

        \begin{figure}[H]
        \vspace{-4em}
        \centering
        \hspace*{12.5em}
        \begin{tikzpicture}[line cap=round,line join=round,>=triangle 45,x=0.4cm,y=0.4cm]
            \clip(-0.1,-7.22) rectangle (31.54,7.82);
                \draw [line width=2pt,color=ffqqqq] (0,2)-- (0,0);
                \draw [line width=1pt] (0,0)-- (1.7320508075688774,-1);
                \draw [line width=1pt] (1.7320508075688774,-1)-- (3.4641016151377553,0);
                \draw [line width=1pt] (3.4641016151377553,0)-- (3.4641016151377557,2);
                \draw [line width=1pt] (3.4641016151377557,2)-- (1.732050807568879,3);
                \draw [line width=2pt,color=qqffqq] (1.732050807568879,3)-- (0,2);
                \draw [line width=1pt] (3.4641016151377557,2)-- (3.4641016151377553,0);
                \draw [line width=1pt] (3.4641016151377553,0)-- (5.196152422706631,-1);
                \draw [line width=2pt,color=qqffff] (5.196152422706631,-1)-- (6.928203230275508,0);
                \draw [line width=1pt] (6.928203230275508,0)-- (6.928203230275509,2);
                \draw [line width=1pt] (6.928203230275509,2)-- (5.196152422706634,3);
                \draw [line width=1pt] (5.196152422706634,3)-- (3.4641016151377557,2);
                \draw [line width=1pt] (6.928203230275509,2)-- (6.928203230275508,0);
                \draw [line width=1pt] (6.928203230275508,0)-- (8.660254037844386,-1);
                \draw [line width=1pt] (8.660254037844386,-1)-- (10.392304845413262,0);
                \draw [line width=1pt] (10.392304845413262,0)-- (10.392304845413262,2);
                \draw [line width=1pt] (10.392304845413262,2)-- (8.660254037844387,3);
                \draw [line width=1pt] (8.660254037844387,3)-- (6.928203230275509,2);
                \draw [line width=1pt] (13,2)-- (13,0);
                \draw [line width=1pt] (13,0)-- (14.732050807568879,-1);
                \draw [line width=1pt] (14.732050807568879,-1)-- (16.464101615137757,0);
                \draw [line width=1pt] (16.464101615137757,0)-- (16.464101615137757,2);
                \draw [line width=1pt] (16.464101615137757,2)-- (14.732050807568879,3);
                \draw [line width=1pt] (14.732050807568879,3)-- (13,2);
                \draw (10.7,1.3) node[anchor=north west] {\Large$\dots$};
                \begin{scriptsize}
                \draw [fill=xdxdff] (0,2) circle (1.5pt);
                \draw [fill=xdxdff] (0,0) circle (1.5pt);
                \draw [fill=xdxdff] (1.7320508075688774,-1) circle (1.5pt);
                \draw [fill=xdxdff] (3.4641016151377553,0) circle (1.5pt);
                \draw [fill=xdxdff] (3.4641016151377557,2) circle (1.5pt);
                \draw [fill=xdxdff] (1.732050807568879,3) circle (1.5pt);
                \draw [fill=xdxdff] (5.196152422706631,-1) circle (1.5pt);
                \draw [fill=xdxdff] (6.928203230275508,0) circle (1.5pt);
                \draw [fill=xdxdff] (6.928203230275509,2) circle (1.5pt);
                \draw [fill=xdxdff] (5.196152422706634,3) circle (1.5pt);
                \draw [fill=xdxdff] (8.660254037844386,-1) circle (1.5pt);
                \draw [fill=xdxdff] (10.392304845413262,0) circle (1.5pt);
                \draw [fill=xdxdff] (10.392304845413262,2) circle (1.5pt);
                \draw [fill=xdxdff] (8.660254037844387,3) circle (1.5pt);
                \draw [fill=ududff] (13,2) circle (1.5pt);
                \draw [fill=ududff] (13,0) circle (1.5pt);
                \draw [fill=xdxdff] (14.732050807568879,-1) circle (1.5pt);
                \draw [fill=xdxdff] (16.464101615137757,0) circle (1.5pt);
                \draw [fill=xdxdff] (16.464101615137757,2) circle (1.5pt);
                \draw [fill=xdxdff] (14.732050807568879,3) circle (1.5pt);
                \end{scriptsize}
            \end{tikzpicture}
        \vspace{-7.5em}
        \caption{Hexagon line tiling}
        \label{fig: Hexagon line tiling}
        \end{figure}
    
        \begin{figure}[H]
        \vspace{-4em}
        \centering
        \hspace{1em}
        \begin{tikzpicture}[line cap=round,line join=round,>=triangle 45,x=0.3cm,y=0.3cm]
            \clip(-0.2,-6.76) rectangle (24.04,8.28);
            
                \draw [line width=2pt,color=ffqqqq] (0,2)-- (0,0);
                \draw [line width=1pt] (0,0)-- (1.414213562373095,-1.414213562373095);
                \draw [line width=1pt] (1.414213562373095,-1.414213562373095)-- (3.4142135623730945,-1.414213562373095);
                \draw [line width=1pt] (3.4142135623730945,-1.414213562373095)-- (4.82842712474619,0);
                \draw [line width=1pt] (4.82842712474619,0)-- (4.82842712474619,2);
                \draw [line width=1pt] (4.82842712474619,2)-- (3.4142135623730954,3.414213562373095);
                \draw [line width=1pt] (3.4142135623730954,3.414213562373095)-- (1.4142135623730956,3.4142135623730954);
                \draw [line width=2pt,color=qqffqq] (1.4142135623730956,3.4142135623730954)-- (0,2);
                \draw [line width=1pt] (4.82842712474619,2)-- (4.82842712474619,0);
                \draw [line width=1pt] (4.82842712474619,0)-- (6.242640687119284,-1.4142135623730945);
                \draw [line width=2pt,color=qqffff] (6.242640687119284,-1.4142135623730945)-- (8.242640687119284,-1.414213562373095);
                \draw [line width=1pt] (8.242640687119284,-1.414213562373095)-- (9.656854249492378,0);
                \draw [line width=1pt] (9.656854249492378,0)-- (9.656854249492378,2);
                \draw [line width=1pt] (9.656854249492378,2)-- (8.242640687119284,3.4142135623730936);
                \draw [line width=1pt] (8.242640687119284,3.4142135623730936)-- (6.242640687119285,3.414213562373094);
                \draw [line width=1pt] (6.242640687119285,3.414213562373094)-- (4.82842712474619,2);
                \draw [line width=1pt] (9.656854249492378,2)-- (9.656854249492378,0);
                \draw [line width=1pt] (9.656854249492378,0)-- (11.071067811865472,-1.414213562373095);
                \draw [line width=1pt] (11.071067811865472,-1.414213562373095)-- (13.071067811865472,-1.414213562373095);
                \draw [line width=1pt] (13.071067811865472,-1.414213562373095)-- (14.485281374238566,0);
                \draw [line width=1pt] (14.485281374238566,0)-- (14.485281374238566,2);
                \draw [line width=1pt] (14.485281374238566,2)-- (13.071067811865472,3.414213562373093);
                \draw [line width=1pt] (13.071067811865472,3.414213562373093)-- (11.071067811865474,3.414213562373093);
                \draw [line width=1pt] (11.071067811865474,3.414213562373093)-- (9.656854249492378,2);
                \draw [line width=1pt] (17.36,2)-- (17.36,0);
                \draw [line width=1pt] (17.36,0)-- (18.774213562373095,-1.4142135623730958);
                \draw [line width=1pt] (18.774213562373095,-1.4142135623730958)-- (20.774213562373095,-1.4142135623730958);
                \draw [line width=1pt] (20.774213562373095,-1.4142135623730958)-- (22.18842712474619,0);
                \draw [line width=1pt] (22.18842712474619,0)-- (22.18842712474619,2);
                \draw [line width=1pt] (22.18842712474619,2)-- (20.774213562373095,3.414213562373096);
                \draw [line width=1pt] (20.774213562373095,3.414213562373096)-- (18.774213562373095,3.414213562373096);
                \draw [line width=1pt] (18.774213562373095,3.414213562373096)-- (17.36,2);
                \draw (14.5,1.5) node[anchor=north west] {\Large$\dots$};
                \begin{scriptsize}
                \draw [fill=xdxdff] (0,2) circle (1.5pt);
                \draw [fill=xdxdff] (0,0) circle (1.5pt);
                \draw [fill=xdxdff] (1.414213562373095,-1.414213562373095) circle (1.5pt);
                \draw [fill=xdxdff] (3.4142135623730945,-1.414213562373095) circle (1.5pt);
                \draw [fill=xdxdff] (4.82842712474619,0) circle (1.5pt);
                \draw [fill=xdxdff] (4.82842712474619,2) circle (1.5pt);
                \draw [fill=xdxdff] (3.4142135623730954,3.414213562373095) circle (1.5pt);
                \draw [fill=xdxdff] (1.4142135623730956,3.4142135623730954) circle (1.5pt);
                \draw [fill=xdxdff] (6.242640687119284,-1.4142135623730945) circle (1.5pt);
                \draw [fill=xdxdff] (8.242640687119284,-1.414213562373095) circle (1.5pt);
                \draw [fill=xdxdff] (9.656854249492378,0) circle (1.5pt);
                \draw [fill=xdxdff] (9.656854249492378,2) circle (1.5pt);
                \draw [fill=xdxdff] (8.242640687119284,3.4142135623730936) circle (1.5pt);
                \draw [fill=xdxdff] (6.242640687119285,3.414213562373094) circle (1.5pt);
                \draw [fill=xdxdff] (11.071067811865472,-1.414213562373095) circle (1.5pt);
                \draw [fill=xdxdff] (13.071067811865472,-1.414213562373095) circle (1.5pt);
                \draw [fill=xdxdff] (14.485281374238566,0) circle (1.5pt);
                \draw [fill=xdxdff] (14.485281374238566,2) circle (1.5pt);
                \draw [fill=xdxdff] (13.071067811865472,3.414213562373093) circle (1.5pt);
                \draw [fill=xdxdff] (11.071067811865474,3.414213562373093) circle (1.5pt);
                \draw [fill=xdxdff] (17.36,2) circle (1.5pt);
                \draw [fill=xdxdff] (17.36,0) circle (1.5pt);
                \draw [fill=xdxdff] (18.774213562373095,-1.4142135623730958) circle (1.5pt);
                \draw [fill=xdxdff] (20.774213562373095,-1.4142135623730958) circle (1.5pt);
                \draw [fill=xdxdff] (22.18842712474619,0) circle (1.5pt);
                \draw [fill=xdxdff] (22.18842712474619,2) circle (1.5pt);
                \draw [fill=xdxdff] (20.774213562373095,3.414213562373096) circle (1.5pt);
                \draw [fill=xdxdff] (18.774213562373095,3.414213562373096) circle (1.5pt);
                \end{scriptsize}
            \end{tikzpicture}
            \vspace{-4em}
        \caption{Octagon line tiling}
        \label{fig: Octagon line tiling}
        \end{figure}
    
        \begin{figure}[H]
        \vspace{-3em}
        \centering
        \hspace{0.5em}
        \begin{tikzpicture}[line cap=round,line join=round,>=triangle 45,x=0.27cm,y=0.27cm]
            \clip(-0.374,-8.594) rectangle (28.622,7.95);
                \draw [line width=2pt,color=ffqqqq] (0,2)-- (0,0);
                \draw [line width=1pt] (0,0)-- (1.1755705045849458,-1.618033988749894);
                \draw [line width=1pt] (1.1755705045849458,-1.618033988749894)-- (3.0776835371752522,-2.2360679774997894);
                \draw [line width=1pt] (3.0776835371752522,-2.2360679774997894)-- (4.979796569765559,-1.618033988749895);
                \draw [line width=1pt] (4.979796569765559,-1.618033988749895)-- (6.155367074350505,0);
                \draw [line width=1pt] (6.155367074350505,0)-- (6.155367074350506,2);
                \draw [line width=1pt] (6.155367074350506,2)-- (4.9797965697655595,3.618033988749894);
                \draw [line width=1pt] (4.9797965697655595,3.618033988749894)-- (3.0776835371752536,4.236067977499789);
                \draw [line width=1pt] (3.0776835371752536,4.236067977499789)-- (1.1755705045849467,3.618033988749895);
                \draw [line width=2pt,color=qqffqq] (1.1755705045849467,3.618033988749895)-- (0,2);
                \draw [line width=1pt] (6.155367074350506,2)-- (6.155367074350505,0);
                \draw [line width=1pt] (6.155367074350505,0)-- (7.33093757893545,-1.6180339887498951);
                \draw [line width=1pt] (7.33093757893545,-1.6180339887498951)-- (9.233050611525757,-2.2360679774997907);
                \draw [line width=1pt] (9.233050611525757,-2.2360679774997907)-- (11.135163644116064,-1.6180339887498973);
                \draw [line width=2pt,color=qqffff] (11.135163644116064,-1.6180339887498973)-- (12.31073414870101,0);
                \draw [line width=1pt] (12.31073414870101,0)-- (12.310734148701012,2);
                \draw [line width=1pt] (12.310734148701012,2)-- (11.135163644116066,3.6180339887498913);
                \draw [line width=1pt] (11.135163644116066,3.6180339887498913)-- (9.23305061152576,4.236067977499787);
                \draw [line width=1pt] (9.23305061152576,4.236067977499787)-- (7.330937578935454,3.618033988749894);
                \draw [line width=1pt] (7.330937578935454,3.618033988749894)-- (6.155367074350506,2);
                \draw [line width=1pt] (12.310734148701012,2)-- (12.31073414870101,0);
                \draw [line width=1pt] (12.31073414870101,0)-- (13.486304653285956,-1.6180339887498993);
                \draw [line width=1pt] (13.486304653285956,-1.6180339887498993)-- (15.388417685876261,-2.2360679774997965);
                \draw [line width=1pt] (15.388417685876261,-2.2360679774997965)-- (17.29053071846657,-1.6180339887499033);
                \draw [line width=1pt] (17.29053071846657,-1.6180339887499033)-- (18.466101223051517,0);
                \draw [line width=1pt] (18.466101223051517,0)-- (18.46610122305152,2);
                \draw [line width=1pt] (18.46610122305152,2)-- (17.290530718466574,3.618033988749886);
                \draw [line width=1pt] (17.290530718466574,3.618033988749886)-- (15.388417685876268,4.236067977499783);
                \draw [line width=1pt] (15.388417685876268,4.236067977499783)-- (13.486304653285961,3.61803398874989);
                \draw [line width=1pt] (13.486304653285961,3.61803398874989)-- (12.310734148701012,2);
                \draw [line width=1pt] (21.38,2)-- (21.38,0);
                \draw [line width=1pt] (21.38,0)-- (22.555570504584946,-1.6180339887498953);
                \draw [line width=1pt] (22.555570504584946,-1.6180339887498953)-- (24.457683537175253,-2.2360679774997907);
                \draw [line width=1pt] (24.457683537175253,-2.2360679774997907)-- (26.35979656976556,-1.6180339887498953);
                \draw [line width=1pt] (26.35979656976556,-1.6180339887498953)-- (27.535367074350507,0);
                \draw [line width=1pt] (27.535367074350507,0)-- (27.535367074350507,2);
                \draw [line width=1pt] (27.535367074350507,2)-- (26.35979656976556,3.6180339887498953);
                \draw [line width=1pt] (26.35979656976556,3.6180339887498953)-- (24.457683537175253,4.236067977499791);
                \draw [line width=1pt] (24.457683537175253,4.236067977499791)-- (22.555570504584946,3.618033988749896);
                \draw [line width=1pt] (22.555570504584946,3.618033988749896)-- (21.38,2);
                \draw (18.3,1.5) node[anchor=north west] {\Large$\dots$};
                \begin{scriptsize}
                \draw [fill=xdxdff] (0,2) circle (1.5pt);
                \draw [fill=xdxdff] (0,0) circle (1.5pt);
                \draw [fill=xdxdff] (1.1755705045849458,-1.618033988749894) circle (1.5pt);
                \draw [fill=xdxdff] (3.0776835371752522,-2.2360679774997894) circle (1.5pt);
                \draw [fill=xdxdff] (4.979796569765559,-1.618033988749895) circle (1.5pt);
                \draw [fill=xdxdff] (6.155367074350505,0) circle (1.5pt);
                \draw [fill=xdxdff] (6.155367074350506,2) circle (1.5pt);
                \draw [fill=xdxdff] (4.9797965697655595,3.618033988749894) circle (1.5pt);
                \draw [fill=xdxdff] (3.0776835371752536,4.236067977499789) circle (1.5pt);
                \draw [fill=xdxdff] (1.1755705045849467,3.618033988749895) circle (1.5pt);
                \draw [fill=xdxdff] (7.33093757893545,-1.6180339887498951) circle (1.5pt);
                \draw [fill=xdxdff] (9.233050611525757,-2.2360679774997907) circle (1.5pt);
                \draw [fill=xdxdff] (11.135163644116064,-1.6180339887498973) circle (1.5pt);
                \draw [fill=xdxdff] (12.31073414870101,0) circle (1.5pt);
                \draw [fill=xdxdff] (12.310734148701012,2) circle (1.5pt);
                \draw [fill=xdxdff] (11.135163644116066,3.6180339887498913) circle (1.5pt);
                \draw [fill=xdxdff] (9.23305061152576,4.236067977499787) circle (1.5pt);
                \draw [fill=xdxdff] (7.330937578935454,3.618033988749894) circle (1.5pt);
                \draw [fill=xdxdff] (13.486304653285956,-1.6180339887498993) circle (1.5pt);
                \draw [fill=xdxdff] (15.388417685876261,-2.2360679774997965) circle (1.5pt);
                \draw [fill=xdxdff] (17.29053071846657,-1.6180339887499033) circle (1.5pt);
                \draw [fill=xdxdff] (18.466101223051517,0) circle (1.5pt);
                \draw [fill=xdxdff] (18.46610122305152,2) circle (1.5pt);
                \draw [fill=xdxdff] (17.290530718466574,3.618033988749886) circle (1.5pt);
                \draw [fill=xdxdff] (15.388417685876268,4.236067977499783) circle (1.5pt);
                \draw [fill=xdxdff] (13.486304653285961,3.61803398874989) circle (1.5pt);
                \draw [fill=xdxdff] (21.38,2) circle (1.5pt);
                \draw [fill=xdxdff] (21.38,0) circle (1.5pt);
                \draw [fill=xdxdff] (22.555570504584946,-1.6180339887498953) circle (1.5pt);
                \draw [fill=xdxdff] (24.457683537175253,-2.2360679774997907) circle (1.5pt);
                \draw [fill=xdxdff] (26.35979656976556,-1.6180339887498953) circle (1.5pt);
                \draw [fill=xdxdff] (27.535367074350507,0) circle (1.5pt);
                \draw [fill=xdxdff] (27.535367074350507,2) circle (1.5pt);
                \draw [fill=xdxdff] (26.35979656976556,3.6180339887498953) circle (1.5pt);
                \draw [fill=xdxdff] (24.457683537175253,4.236067977499791) circle (1.5pt);
                \draw [fill=xdxdff] (22.555570504584946,3.618033988749896) circle (1.5pt);
                \end{scriptsize}
            \end{tikzpicture}
            \vspace{-4.5em}
            \caption{Decagon line tiling}
            \label{Decagon line tiling}
    
        \end{figure}

        Observe that, for both cases (\textit{i.e.}, when $n$ is even and odd), the first two-element pairings are the same, \textit{i.e.}, first with $a_{1}$ and then with $b_{{1,1}}$. We first interpret the critical cells obtained after element pairing with these two elements and then do the last element pairing with the suitable edges.
        \vspace{0.5cm}
        
        \noindent\textbf{Element pairing using $a_{1}$:} Recall that $C_{a_{1}}$ is the set of critical cells obtained after element pairing with $a_{1}$. If $\sigma \in C_{a_{1}}$ then $a_{1} \notin \sigma$, implying $\sigma \cup \left\{a_{1}\right\} \notin \pmplt$ which will mean either of the following (not both):
        \begin{enumerate}
        \itemindent5em
            \item[\textbf{Condition 1.}]\label{cond 1} $\sigma \cup \left\{a_{1}\right\}$ is not a matching implying $\sigma \cap \left\{b_{{1,1}},c_{{1,1}}\right\} \neq \emptyset$, or;
            \item[\textbf{Condition 2.}]\label{cond 2} $\sigma \cup \left\{a_{1}\right\}$ is a bad matching, implying $\sigma \cap \left\{b_{{1,1}},c_{{1,1}}\right\} = \emptyset$.
        \end{enumerate}

        Thus,
        \[C_{a_{1}} = \left\{ \sigma \in \pmplt\ |\ a_{1} \notin \sigma,\ \sigma\cup \left\{a_{1}\right\} \notin \pmplt \right\}.\]

        \noindent{\textbf{Element pairing using $b_{{1,1}}$:}} We now define element pairing using $b_{{1,1}}$ on the elements of $C_{a_{1}}$. If $\sigma \in C_{b_{{1,1}}}$, then at least one of the following two conditions must hold:
        \begin{enumerate}[label = \Alph*.]
            \item $b_{{1,1}}\notin \sigma$ and $\sigma\cup \left\{b_{{1,1}}\right\} \notin C_{a_{1}}$, which mean at least one of the following:
            \begin{enumerate}[label = \roman*.]
                \item $\sigma\cup \left\{b_{{1,1}}\right\} \notin \pmplt$.\\
                Like before, this could imply that either $\sigma\cup \left\{b_{{1,1}}\right\}$ is not a matching, or it is a bad matching. Let us observe these two cases separately as follows:

                \begin{itemize}
                \itemsep1em
                    \item If $\sigma \cup \left\{b_{{1,1}}\right\}$ is not a matching, then $\sigma \cap \left\{a_{1},\ b_{{1,2}}\right\} \neq \phi$. We know that $a_{1} \notin \sigma$ which implies $b_{{1,2}} \in \sigma$.
                    
                    If $\sigma$ satisfies the \hyperref[cond 1]{Condition 1}, then we must have $c_{{1,1}} \in \sigma$. However, this leads to a contradiction, as no edge will be left to cover the vertex $u_{{1,1}}$ when $\sigma$ is extended to form a perfect matching.
                    
                    Similarly, if $\sigma$ satisfies \hyperref[cond 2]{Condition 2}, we have $c_{{1,1}} \notin \sigma$. This is not true because $b_{{1,1}} \notin \sigma$, implies that in order to cover the vertex $u_{{1,1}}$, we need to include $a_{1}$ along with $\sigma$ when extending it to form a perfect matching. In which case, $\sigma \cup \left\{ a_{1}\right\} \in \pmplt$, contradicting $\sigma \in C_{a_{1}}$.

                    \item If $\sigma \cup \left\{b_{{1,1}}\right\}$ is a bad matching, then $\sigma \cap \left\{a_{1},\ b_{{1,2}}\right\} = \phi$.

                    If $\sigma$ satisfies the \hyperref[cond 1]{Condition 1}, then $c_{{1,1}} \in \sigma$. Note that the only possible edge left to cover the vertex $u_{{1,1}}$ is $b_{{1,1}}$ when $\sigma$ is extended to form a perfect matching. In which case $\sigma \cup \left\{b_{{1,1}} \right\} \in \pmplt$, which is a contradiction.

                    Similarly, if $\sigma$ satisfies \hyperref[cond 2]{Condition 2}, then to cover the vertex $u_{{1,1}}$, we need to either take $a_{1}$ or $b_{{1,1}}$ along with $\sigma$, when it is extended to form a perfect matching. If we take $a_{1}$, then $\sigma \cup \left\{a_{1}\right\} in \pmplt$ will contradict $\sigma \in C_{a_{1}}$ and, if we take $b_{{1,1}}$, then $\sigma \cup \left\{b_{{1,1}}\right\} \in \pmplt$, contradicting our assumption.
                \end{itemize}
                Hence, this case is not possible.
                \item $\sigma\cup \left\{b_{{1,1}}\right\} \in \pmplt$, $\sigma \cup \left\{b_{{1,1}}\right\} \notin C_{a_{1}}$.\\
                This would imply that, $\sigma \cup \left\{a_{1},\ b_{{1,1}}\right\} \in \pmplt$. However, this is a contradiction to $\pmplt$ being simplicial complex as $\sigma \cup \left\{a_{1}\right\} \notin \pmplt$. Hence, this case is not possible.                
            \end{enumerate}
            \item $b_{{1,1}} \in \sigma$, $\sigma \setminus \left\{b_{{1,1}}\right\} \notin C_{a_{1}}$.

            The latter would imply two cases as earlier. The first case is $\sigma \setminus \left\{b_{{1,1}}\right\} \notin \pmplt$, but this will contradict $\pmplt$ being a simplicial complex as $\sigma \in \pmplt$. Therefore, this case is not possible. The other case is when $\sigma \setminus \left\{b_{{1,1}}\right\} \in \pmplt$ and $\sigma \setminus \left\{b_{{1,1}}\right\} \notin C_{a_{1}}$, \textit{i.e.}, $\sigma \cup \left\{a_{1}\right\} \setminus \left\{b_{{1,1}}\right\} \in \pmplt$. This would imply that $c_{{1,1}} \notin \sigma$, since $\sigma \cup \left\{a_{1}\right\} \notin \pmplt$. Hence, this case is possible.
        \end{enumerate}
        Thus,
        \begin{align*}
            C_{b_{{1,1}}} = \left\{ \sigma \in \pmplt\ \middle|\ \begin{array}{c} 
                a_{1} \notin \sigma,\ b_{{1,1}} \in \sigma,\ \sigma\cup \left\{a_{1}\right\} \notin \pmplt,\\
                \sigma\cup\left\{a_{1}\right\}\setminus\left\{b_{{1,1}}\right\} \in \pmplt \end{array}\right\}.
        \end{align*}
    
    Observe here that the property $\sigma \cup \left\{a_{1}\right\} \notin \pmplt$ will give the same result we got earlier when $n$ is even or odd. However, $\sigma \cup \left\{a_{1}\right\} \setminus \left\{b_{{1,1}}\right\} \in \pmplt$ will give slightly different results in case when it is even and when it is odd. Note here that $b_{{1,1}}\in\sigma$ implies that $b_{{1,2}} \notin \sigma$ in either case.

    The proof will be similar till this point, irrespective of what $n$ is. From now on, we interpret the properties of $C_{b_{{1,1}}}$ by bifurcating our proof for the case when $n$ is even and odd.

    \begin{itemize}
        \item \textbf{When $n$ is even.}
         
            For the case when $n$ is even, using the property $b_{{1,1}} \in \sigma$, we conclude the following
            \begin{enumerate}[label = \roman*.]
            \itemsep0em
                \item $b_{1,i} \notin \sigma$, for $i = 4,\ 6,\dots,\ n-2$ ,\item $c_{1,j} \notin \sigma$, for $j = 2,\ 4,\ 6, \dots,\ n-2$,
                \item $a_{2} \notin \sigma$.
            \end{enumerate}
            
            This is because, for every edge here and $b_{{1,1}}$, there will be an odd number of vertices in between them (e.g., $b_{{1,1}}$ and $b_{{1,6}}$ have three vertices, namely, $u_{{1,3}}$, $u_{{1,4}}$ and $u_{{1,5}}$ in between). Thus, when $\sigma$ is extended to form a perfect matching, a vertex will be left uncovered among the vertices in between. Similarly, using the property $\sigma \cup \left\{a_{1}\right\}\setminus\left\{b_{{1,1}}\right\} \in \pmplt$, we conclude that, $b_{{1,k}},\ c_{{1,k}} \notin \sigma$, for $k = 3,\ 5,\ \dots,\ n-1$, where instead of $b_{{1,1}}$, we consider the fact that $a_{1} \in\sigma \cup \left\{a_{1}\right\}\setminus\left\{b_{{1,1}}\right\}$. Also, note that $c_{{1,1}}\notin \sigma$. Thus, from the first polygon, no edge is in $\sigma \in C_{b_{{1,1}}}$ except for $b_{{1,1}}$.

            We also need to find the edges from the second polygon that will be in $\sigma$ because it will help to interpret the result after the third-element pairing.

            Recall that, $\sigma \in C_{b_{{1,1}}}$. From the second polygon, again using the previous argument and the property $b_{{1,1}} \in \sigma$, we conclude that the edges $b_{2, i},\ c_{2, i} \notin \sigma$, for $i=1,\ 3,\dots,\ n-1$. However, note here that while considering the edges $b_{2, i}$, consider the upper vertices of the first polygon, and for the edges $c_{2, i}$, consider the lower vertices for our argument.

            Now, the edges $b_{{2,j}},\ c_{{2,j}}$ can be in $\sigma$, for some $j = 2,\ 4,\dots,\ n-2$ and some $\sigma \in C_{b_{{1,1}}}$. We show this for $b_{{2,2}}$ since the other edges will satisfy similar conditions. To have $b_{{2,2}} \in \sigma$, we need to find such $\sigma \in C_{b_{{1,1}}}$, which can be extended to form a perfect matching. This can be easily done, since if $b_{{2,2}} \in \sigma$ for some $\sigma \in C_{b_{{1,1}}}$, then $b_{{1,1}} \in \sigma$ and we can extend $\sigma$ to a perfect matching $\tau$ such that $b_{{1,m}},\ c_{{1,m}} \in \tau$, for $m = 1,\ 3,\dots,\ n-1$ from the first polygon. Note that, this $\sigma$ will also satisfy the condition $\sigma \cup \left\{a_{1}\right\} \setminus \left\{b_{{1,1}}\right\} \in \pmplt$. Furthermore, it is not hard to see that the same perfect matching $\tau$ will work for all those $\sigma \in C_{b_{{1,1}}}$, such that either $b_{{2,t}} \in \sigma$ or $c_{{2,s}} \in \sigma$, where $t = 4,\ 6,\dots,\ n-2$ and $s = 2,\ 4,\dots,\ n-2$.
            
            Note here that, it is not necessary that every $\sigma\in C_{b_{{1,1}}}$ will contain at least one of the $b_{{2,j}}$ or $c_{{2,j}}$, \textit{i.e.}, there might be a $\sigma \in C_{b_{{1,1}}}$, not having any edge $b_{{2,j}}\ \text{or}\ c_{{2,j}}$, for $j = 2,\ 4,\dots,\ n-2$. Thus, from the second polygon, the edges $b_{{2,i}},\ c_{{2,i}} \notin \sigma$, for $i=1,\ 3,\dots,\ n-1$, and the edges $b_{{2,j}},\ c_{{2,j}}$, for $j = 2,\ 4,\dots,\ n-2$, can be there in some $\sigma \in C_{b_{{1,1}}}$.

            \subsubsection*{Element pairing using $c_{{2,n-2}}$:}
    
            We now have all the necessary information for defining the last element pairing using $c_{{2,n-2}}$ on the elements of $C_{b_{{1,1}}}$. If an element $\sigma \in C_{b_{{1,1}}}$ is left unpaired after element pairing with $c_{{2,n-2}}$, then at least one of the following two conditions must hold:
        
            \begin{enumerate}[label = \Alph*.]
                \item $c_{{2,n-2}}\notin \sigma$ and $\sigma\cup \left\{c_{{2,n-2}}\right\} \notin C_{b_{{1,1}}}$, which means at least one of the following:
                \begin{enumerate}[label = \roman*.]
                    \item $\sigma\cup \left\{c_{{2,n-2}}\right\} \notin \pmplt$.\\
                    Note that, when $\sigma$ is extended to form a perfect matching, say $\tau$, then $\tau$ must contain $c_{{2,n-2}}$ to cover the vertex $l_{{2,n-2}}$ (or $l_{{2,n-1}}$) since $\sigma \cap \left\{c_{{2,n-3}}, c_{{2,n-1}}\right\} = \emptyset$. Thus, $\sigma \cup \left\{c_{{2,n-2}}\right\} \in \pmplt$ will always hold. Hence, this case is not possible.
                    
                    \item $\sigma\cup \left\{c_{{2,n-2}}\right\} \in \pmplt$ and $\sigma \cup \left\{c_{{2,n-2}}\right\} \notin C_{a_{1}}$.\\
                    This would imply that, $\sigma \cup \left\{ a_{1},\ c_{{2,n-2}}\right\} \in \pmplt$. However, this is a contradiction to $\pmplt$ being simplicial complex as $\sigma \cup \left\{a_{1}\right\} \notin \pmplt$. Hence, this case is not possible.
    
                    \item $\sigma\cup \left\{c_{{2,n-2}}\right\} \in C_{a_{1}}$ and $\sigma \cup \left\{c_{{2,n-2}}\right\} \notin C_{b_{{1,1}}}$.\\
                    This would imply that $\sigma \cup \left\{ a_{1},\ c_{{2,n-2}} \right\} \setminus \left\{b_{{1,1}}\right\} \notin \pmplt$. However, this cell can be extended to form a perfect matching, say $\tau$, by taking $a_{{2}} \in \tau$ and filling the rest of the edges accordingly. Hence, this case is not possible.
                \end{enumerate}
                
                \item $c_{{2,n-2}}\in\sigma$ and $\sigma\setminus\left\{c_{{2,n-2}}\right\} \notin C_{b_{{1,1}}}$.
    
                Note that $\sigma \setminus \left\{c_{{2,n-2}}\right\} \notin C_{b_{{1,1}}}$ means at least one of the following:
    
                \begin{enumerate}[label = \roman*.]
                    \item $\sigma\setminus\left\{c_{{2,n-2}}\right\} \notin \pmplt$.
                            
                    This case contradicts $\pmplt$ being a simplicial complex, as $\sigma \in \pmplt$. Hence, this case is not possible.
                            
                    \item $\sigma \setminus \left\{c_{{2,n-2}}\right\} \in \pmplt \text{ and } \sigma\setminus\left\{c_{{2,n-2}}\right\} \notin C_{a_{1}}$.
        
                    This would $\sigma \cup \left\{a_{1}\right\} \setminus \left\{c_{{2,n-2}}\right\} \in \pmplt$, which is a contradiction since $b_{{1,1}} \in \sigma$. Hence, this case is not possible.
        
                    \item $\sigma\setminus\left\{c_{{2,n-2}}\right\} \in C_{a_{1}} \text{ and } \sigma\setminus\left\{c_{{2,n-2}}\right\} \notin C_{b_{{1,1}}}$.
        
                    This would imply that, $\sigma \cup \left\{a_{1}\right\} \setminus \left\{b_{{1,1}},\ c_{{2,n-2}}\right\} \notin \pmplt$. This contradicts $\pmplt$ being a simplicial complex as $\sigma \cup \left\{a_{1}\right\} \setminus \left\{b_{{1,1}}\right\} \in \pmplt$ (since $\sigma \in C_{b_{{1,1}}}$). Hence, this case is also not possible.
                \end{enumerate}
            \end{enumerate}
    
            This computation shows no critical cells are left after element pairing with $C_{c_{{2,n-2}}}$. Thus, 
                \begin{align*}
                        C_{c_{{2,n-2}}} = \emptyset.
                \end{align*}

        \item \textbf{When $n$ is odd.}
        
            In this case, using the property $b_{{1,1}} \in \sigma$, we conclude the following:
            \begin{enumerate}[label = \roman*.]
            \itemsep0em
                \item $b_{{1,i}} \notin \sigma$, for $i = 4,\ 6,\dots,\ n-1$
                \item $c_{{1,j}} \notin \sigma$, for $j = 2,\ 4,\ 6, \dots,\ n-1$
            \end{enumerate}
            This follows from the similar argument we used in the earlier case. Similarly, using the property $\sigma \cup \left\{a_{1}\right\}\setminus\left\{b_{{1,1}}\right\} \in \pmplt$, we conclude the following,
    
            \begin{enumerate}[label = \roman*.]
                \item $b_{{1,k}},\ c_{{1,k}} \notin \sigma$, for $k = 3,\ 5,\ \dots,\ n-2$,
                \item $a_{2} \notin \sigma$.
            \end{enumerate}
            
            Thus, from the first polygon, no edge is in $\sigma \in C_{b_{{1,1}}}$ except for $b_{{1,1}}$.
            Moreover, from the second polygon, the edges $b_{{2,i}},\ c_{{2,i}} \notin \sigma$, for $i = 1,\ 3,\dots,\ n-2$, using the property $\sigma \cup \left\{a_{1}\right\} \setminus \left\{b_{{1,1}}\right\} \in \pmplt$. The edges $b_{{2,j}},\ c_{{2,j}} \in \sigma$, for $j= 2,\ 4,\dots,\ n-1$, using the argument we applied earlier. Thus, from the second polygon, the edges $b_{{2,i}},\ c_{{2,i}} \notin \sigma$, for $i=1,\ 3,\dots,\ n-2$, and the edges $b_{{2,j}},\ c_{{2,j}}$, for $j = 2,\ 4,\dots,\ n-1$, can be there in some $\sigma \in C_{b_{{1,1}}}$.
            
            \vspace{0.5cm}    
            \textbf{Element pairing using $c_{{2,n-1}}$:} We define the last element pairing using $c_{{2,n-1}}$ on the elements of $C_{b_{{1,1}}}$. We refrain from going into details of this as the argument is similar to what we have provided in the previous part. Thus,
            \begin{align*}
                C_{c_{{2,n-1}}} = \emptyset.
            \end{align*}

    \end{itemize}

    In both cases, we are left with no critical cells after the last element pairing. Hence, \Cref{Corollary 1} implies that $\pmplt$ is contractible.
    \end{proof}

\section{The curious case of odd-sided polygons}\label{Section 5}

    
    The reason why we are dealing with this case separately in a section is because of how we can arrange an odd-sided polygon in a line. We use different methods for different arrangements to get to our results. The interesting part is that we get the same result for both scenarios, \textit{i.e.}, the simplicial complex associated here is contractible, no matter what arrangement we choose. Furthermore, we separately deal with the case when we take the general line tiling of triangles.

    Let us first discuss what we mean when we say odd-sided polygons can be arranged differently. Recall how we attached even-sided polygons to obtain the line tiling in the previous section. The even-sided polygons were attached along the parallel edges, so the number of edges above and below them is the same. Odd-sided polygons have no parallel edges; thus, we cannot use the same technique here. We can attach odd-sided polygons in the following ways to make them appear like they are arranged in a line and make further interpretations from it:

    \begin{enumerate}
        \item First, attach them so that the number of edges above the attached edges is exactly one more than that of the number of edges below them in every polygon present in the line tiling. Let us call this a \textit{simple arrangement} of odd-sided polygons in line tiling (see \Cref{fig: Simple arrangement}).

        \begin{figure}[H]
            \vspace{-4em}
            \centering
            \hspace*{0.5em}\begin{tikzpicture}[line cap=round,line join=round,>=triangle 45,x=0.3cm,y=0.3cm]
                \clip(-8.853438905441942,-14.03135743151773) rectangle (33.63375748029108,10.210199231662921);
                    \draw [line width=1pt] (0,0)-- (0,2.5);
                    \draw [line width=1pt] (0,2.5)-- (2.0022750483802536,3.753367938940545);
                    \draw [line width=1pt] (2.0022750483802536,3.753367938940545)-- (4,2.5);
                    \draw [line width=1pt] (4,2.5)-- (4.003406850078927,0);
                    \draw [line width=1pt] (4.003406850078927,0)-- (0,0);
                    \draw [line width=1pt] (4,2.5)-- (6,3.75);
                    \draw [line width=1pt] (6,3.75)-- (8,2.5);
                    \draw [line width=1pt] (8,2.5)-- (8,0);
                    \draw [line width=1pt] (8,0)-- (4.003406850078927,0);
                    \draw [line width=1pt] (8,2.5)-- (10,3.75);
                    \draw [line width=1pt] (10,3.75)-- (12.018213373925816,2.604668471970693);
                    \draw [line width=1pt] (12.018213373925816,2.604668471970693)-- (12,0);
                    \draw [line width=1pt] (12,0)-- (8,0);
                    \draw [line width=1pt] (12.018213373925816,2.604668471970693)-- (14.118286398710405,3.6958530300519614);
                    \draw [line width=1pt] (14.118286398710405,3.6958530300519614)-- (16,2.5);
                    \draw [line width=1pt] (16,2.5)-- (16,0);
                    \draw [line width=1pt] (16,0)-- (12,0);
                    \draw [line width=1pt] (16,2.5)-- (17.997760761151742,3.724601932924742);
                    \draw [line width=1pt] (17.997760761151742,3.724601932924742)-- (20,2.5);
                    \draw [line width=1pt] (20,2.5)-- (20,0);
                    \draw [line width=1pt] (20,0)-- (16,0);
                    \draw [line width=1pt] (20,2.5)-- (22,3.75);
                    \draw [line width=1pt] (22,3.75)-- (24,2.5);
                    \draw [line width=1pt] (24,2.5)-- (24,0);
                    \draw [line width=1pt] (24,0)-- (20,0);
                    \draw [line width=1pt] (0,-6)-- (0,-4);
                    \draw [line width=1pt] (0,-4)-- (1.7320508075688772,-3);
                    \draw [line width=1pt] (1.7320508075688772,-3)-- (4.330127018922193,-3);
                    \draw [line width=1pt] (4.330127018922193,-3)-- (6.06217782649107,-4);
                    \draw [line width=1pt] (6.06217782649107,-4)-- (6.06217782649107,-6);
                    \draw [line width=1pt] (0,-6)-- (3.031088913245535,-7.25);
                    \draw [line width=1pt] (3.031088913245535,-7.25)-- (6.06217782649107,-6);
                    \draw [line width=1pt] (6.06217782649107,-4)-- (7.794228634059947,-3);
                    \draw [line width=1pt] (7.794228634059947,-3)-- (10.392304845413264,-3);
                    \draw [line width=1pt] (10.392304845413264,-3)-- (12.12435565298214,-4);
                    \draw [line width=1pt] (12.12435565298214,-4)-- (12.12435565298214,-6);
                    \draw [line width=1pt] (6.06217782649107,-6)-- (9.093266739736606,-7.25);
                    \draw [line width=1pt] (9.093266739736606,-7.25)-- (12.12435565298214,-6);
                    \draw [line width=1pt] (12.12435565298214,-4)-- (13.856406460551018,-3);
                    \draw [line width=1pt] (13.856406460551018,-3)-- (16.454482671904334,-3);
                    \draw [line width=1pt] (16.454482671904334,-3)-- (18.186533479473212,-4);
                    \draw [line width=1pt] (18.186533479473212,-4)-- (18.186533479473212,-6);
                    \draw [line width=1pt] (18.186533479473212,-6)-- (15.155444566227676,-7.25);
                    \draw [line width=1pt] (15.155444566227676,-7.25)-- (12.12435565298214,-6);
                    \draw [line width=1pt] (18.186533479473212,-4)-- (19.918584287042087,-3);
                    \draw [line width=1pt] (19.918584287042087,-3)-- (22.516660498395403,-3);
                    \draw [line width=1pt] (22.516660498395403,-3)-- (24.24871130596428,-4);
                    \draw [line width=1pt] (24.24871130596428,-4)-- (24.24871130596428,-6);
                    \draw [line width=1pt] (21.217622392718745,-7.25)-- (24.24871130596428,-6);
                    \draw [line width=1pt] (21.217622392718745,-7.25)-- (18.186533479473212,-6);
                    \draw [line width=1pt] (0,-4)-- (-1.7320508075688772,-3);
                    \draw [line width=1pt] (-1.7320508075688772,-3)-- (-4.330127018922193,-3);
                    \draw [line width=1pt] (-4.330127018922193,-3)-- (-6.06217782649107,-4);
                    \draw [line width=1pt] (-6.06217782649107,-4)-- (-6.06217782649107,-6);
                    \draw [line width=1pt] (0,-6)-- (-3.031088913245535,-7.25);
                    \draw [line width=1pt] (-3.031088913245535,-7.25)-- (-6.06217782649107,-6);
                    \draw [line width=1pt] (24.24871130596428,-4)-- (25.980762113533157,-3);
                    \draw [line width=1pt] (25.980762113533157,-3)-- (28.578838324886473,-3);
                    \draw [line width=1pt] (28.578838324886473,-3)-- (30.31088913245535,-4);
                    \draw [line width=1pt] (30.31088913245535,-4)-- (30.31088913245535,-6);
                    \draw [line width=1pt] (30.31088913245535,-6)-- (27.279800219209815,-7.25);
                    \draw [line width=1pt] (27.279800219209815,-7.25)-- (24.24871130596428,-6);
                    \draw [line width=1pt] (24,2.5)-- (26,3.75);
                    \draw [line width=1pt] (26,3.75)-- (28,2.5);
                    \draw [line width=1pt] (28,2.5)-- (28,0);
                    \draw [line width=1pt] (28,0)-- (24,0);
                    \draw [line width=1pt] (0,2.5)-- (-2,3.75);
                    \draw [line width=1pt] (-2,3.75)-- (-4,2.5);
                    \draw [line width=1pt] (-4,2.5)-- (-4,0);
                    \draw [line width=1pt] (-4,0)-- (0,0);
                \begin{scriptsize}
                    \draw [fill=xdxdff] (0,0) circle (1.5pt);
                    \draw [fill=xdxdff] (0,2.5) circle (1.5pt);
                    \draw [fill=ududff] (2.0022750483802536,3.753367938940545) circle (1.5pt);
                    \draw [fill=ududff] (4,2.5) circle (1.5pt);
                    \draw [fill=xdxdff] (4.003406850078927,0) circle (1.5pt);
                    \draw [fill=ududff] (6,3.75) circle (1.5pt);
                    \draw [fill=ududff] (8,2.5) circle (1.5pt);
                    \draw [fill=xdxdff] (8,0) circle (1.5pt);
                    \draw [fill=ududff] (10,3.75) circle (1.5pt);
                    \draw [fill=ududff] (12.018213373925816,2.604668471970693) circle (1.5pt);
                    \draw [fill=xdxdff] (12,0) circle (1.5pt);
                    \draw [fill=ududff] (14.118286398710405,3.6958530300519614) circle (1.5pt);
                    \draw [fill=ududff] (16,2.5) circle (1.5pt);
                    \draw [fill=ududff] (16,0) circle (1.5pt);
                    \draw [fill=ududff] (17.997760761151742,3.724601932924742) circle (1.5pt);
                    \draw [fill=ududff] (20,2.5) circle (1.5pt);
                    \draw [fill=xdxdff] (20,0) circle (1.5pt);
                    \draw [fill=ududff] (22,3.75) circle (1.5pt);
                    \draw [fill=ududff] (24,2.5) circle (1.5pt);
                    \draw [fill=ududff] (24,0) circle (1.5pt);
                    \draw [fill=xdxdff] (0,-6) circle (1.5pt);
                    \draw [fill=xdxdff] (0,-4) circle (1.5pt);
                    \draw [fill=ududff] (1.7320508075688772,-3) circle (1.5pt);
                    \draw [fill=ududff] (4.330127018922193,-3) circle (1.5pt);
                    \draw [fill=ududff] (6.06217782649107,-4) circle (1.5pt);
                    \draw [fill=ududff] (6.06217782649107,-6) circle (1.5pt);
                    \draw [fill=ududff] (3.031088913245535,-7.25) circle (1.5pt);
                    \draw [fill=ududff] (7.794228634059947,-3) circle (1.5pt);
                    \draw [fill=ududff] (10.392304845413264,-3) circle (1.5pt);
                    \draw [fill=ududff] (12.12435565298214,-4) circle (1.5pt);
                    \draw [fill=ududff] (12.12435565298214,-6) circle (1.5pt);
                    \draw [fill=ududff] (9.093266739736606,-7.25) circle (1.5pt);
                    \draw [fill=ududff] (13.856406460551018,-3) circle (1.5pt);
                    \draw [fill=ududff] (16.454482671904334,-3) circle (1.5pt);
                    \draw [fill=ududff] (18.186533479473212,-4) circle (1.5pt);
                    \draw [fill=ududff] (18.186533479473212,-6) circle (1.5pt);
                    \draw [fill=ududff] (15.155444566227676,-7.25) circle (1.5pt);
                    \draw [fill=ududff] (19.918584287042087,-3) circle (1.5pt);
                    \draw [fill=ududff] (22.516660498395403,-3) circle (1.5pt);
                    \draw [fill=ududff] (24.24871130596428,-4) circle (1.5pt);
                    \draw [fill=ududff] (24.24871130596428,-6) circle (1.5pt);
                    \draw [fill=ududff] (21.217622392718745,-7.25) circle (1.5pt);
                    \draw [fill=ududff] (-1.7320508075688772,-3) circle (1.5pt);
                    \draw [fill=ududff] (-4.330127018922193,-3) circle (1.5pt);
                    \draw [fill=ududff] (-6.06217782649107,-4) circle (1.5pt);
                    \draw [fill=ududff] (-6.06217782649107,-6) circle (1.5pt);
                    \draw [fill=ududff] (-3.031088913245535,-7.25) circle (1.5pt);
                    \draw [fill=ududff] (25.980762113533157,-3) circle (1.5pt);
                    \draw [fill=ududff] (28.578838324886473,-3) circle (1.5pt);
                    \draw [fill=ududff] (30.31088913245535,-4) circle (1.5pt);
                    \draw [fill=ududff] (30.31088913245535,-6) circle (1.5pt);
                    \draw [fill=ududff] (27.279800219209815,-7.25) circle (1.5pt);
                    \draw [fill=ududff] (26,3.75) circle (1.5pt);
                    \draw [fill=ududff] (28,2.5) circle (1.5pt);
                    \draw [fill=xdxdff] (28,0) circle (1.5pt);
                    \draw [fill=ududff] (-2,3.75) circle (1.5pt);
                    \draw [fill=ududff] (-4,2.5) circle (1.5pt);
                    \draw [fill=xdxdff] (-4,0) circle (1.5pt);
                \end{scriptsize}
            \end{tikzpicture}
            \vspace{-5.5em}
            \caption{Simple arrangement of odd-sided polygons in line tiling}
            \label{fig: Simple arrangement}
        \end{figure}
    
        \item Secondly, attach the odd-sided polygon so that in the first polygon (from left), the number of edges above the attached edges is precisely one more than those below them. In the second polygon, the number of edges below the attached edges is precisely one more than those above them. Similarly, in the third polygon, we use the first polygon's layout; in the fourth polygon, we use the second polygon's layout and continue similarly. This will give an effect as if the polygons are attached alternatively. We call this an \textit{alternate arrangement} of odd-sided polygons in the line tiling (see \Cref{fig: Alternate arrangement}).

        \begin{figure}[H]
            \vspace{-2em}
            \centering
                \hspace*{0.5em}\begin{tikzpicture}[line cap=round,line join=round,>=triangle 45,x=0.3cm,y=0.3cm]
                \clip(-7.488488939126034,-15.613996986304144) rectangle (32.35348150025198,7.118295707484252);
                    \draw [line width=1pt] (0,0)-- (0,2.5);
                    \draw [line width=1pt] (4,2.5)-- (4,0);
                    \draw [line width=1pt] (4,2.5)-- (6,3.75);
                    \draw [line width=1pt] (6,3.75)-- (8,2.5);
                    \draw [line width=1pt] (8,2.5)-- (8,0);
                    \draw [line width=1pt] (8,0)-- (4.003406850078927,0);
                    \draw [line width=1pt] (12,2.5)-- (12,0);
                    \draw [line width=1pt] (12,2.5)-- (14.118286398710405,3.6958530300519614);
                    \draw [line width=1pt] (14.118286398710405,3.6958530300519614)-- (16,2.5);
                    \draw [line width=1pt] (16,2.5)-- (16,0);
                    \draw [line width=1pt] (16,0)-- (12,0);
                    \draw [line width=1pt] (20,2.5)-- (20,0);
                    \draw [line width=1pt] (20,2.5)-- (22,3.75);
                    \draw [line width=1pt] (22,3.75)-- (24,2.5);
                    \draw [line width=1pt] (24,2.5)-- (24,0);
                    \draw [line width=1pt] (24,0)-- (20,0);
                    \draw [line width=1pt] (0,-6)-- (0,-4);
                    \draw [line width=1pt] (6.06217782649107,-4)-- (6.06217782649107,-6);
                    \draw [line width=1pt] (6.06217782649107,-4)-- (7.794228634059947,-3);
                    \draw [line width=1pt] (7.794228634059947,-3)-- (10.392304845413264,-3);
                    \draw [line width=1pt] (10.392304845413264,-3)-- (12.12435565298214,-4);
                    \draw [line width=1pt] (12.12435565298214,-4)-- (12.12435565298214,-6);
                    \draw [line width=1pt] (6.06217782649107,-6)-- (9.093266739736606,-7.25);
                    \draw [line width=1pt] (9.093266739736606,-7.25)-- (12.12435565298214,-6);
                    \draw [line width=1pt] (18.186533479473212,-4)-- (18.186533479473212,-6);
                    \draw [line width=1pt] (18.186533479473212,-4)-- (19.918584287042087,-3);
                    \draw [line width=1pt] (19.918584287042087,-3)-- (22.516660498395403,-3);
                    \draw [line width=1pt] (22.516660498395403,-3)-- (24.24871130596428,-4);
                    \draw [line width=1pt] (24.24871130596428,-4)-- (24.24871130596428,-6);
                    \draw [line width=1pt] (21.217622392718745,-7.25)-- (24.24871130596428,-6);
                    \draw [line width=1pt] (21.217622392718745,-7.25)-- (18.186533479473212,-6);
                    \draw [line width=1pt] (0,-4)-- (-1.7320508075688772,-3);
                    \draw [line width=1pt] (-1.7320508075688772,-3)-- (-4.330127018922193,-3);
                    \draw [line width=1pt] (-4.330127018922193,-3)-- (-6.06217782649107,-4);
                    \draw [line width=1pt] (-6.06217782649107,-4)-- (-6.06217782649107,-6);
                    \draw [line width=1pt] (0,-6)-- (-3.031088913245535,-7.25);
                    \draw [line width=1pt] (-3.031088913245535,-7.25)-- (-6.06217782649107,-6);
                    \draw [line width=1pt] (30.31088913245535,-4)-- (30.31088913245535,-6);
                    \draw [line width=1pt] (28,2.5)-- (28,0);
                    \draw [line width=1pt] (0,2.5)-- (-2,3.75);
                    \draw [line width=1pt] (-2,3.75)-- (-4,2.5);
                    \draw [line width=1pt] (-4,2.5)-- (-4,0);
                    \draw [line width=1pt] (-4,0)-- (0,0);
                    \draw [line width=1pt] (4.003406850078927,0)-- (2,-1.25);
                    \draw [line width=1pt] (2,-1.25)-- (0,0);
                    \draw [line width=1pt] (8,0)-- (10,-1.25);
                    \draw [line width=1pt] (10,-1.25)-- (12,0);
                    \draw [line width=1pt] (16,0)-- (18,-1.25);
                    \draw [line width=1pt] (18,-1.25)-- (20,0);
                    \draw [line width=1pt] (24,0)-- (26,-1.25);
                    \draw [line width=1pt] (26,-1.25)-- (28,0);
                    \draw [line width=1pt] (0,2.5)-- (4,2.5);
                    \draw [line width=1pt] (8,2.5)-- (12,2.5);
                    \draw [line width=1pt] (16,2.5)-- (20,2.5);
                    \draw [line width=1pt] (24,2.5)-- (28,2.5);
                    \draw [line width=1pt] (0,-4)-- (3.031088913245535,-2.75);
                    \draw [line width=1pt] (3.031088913245535,-2.75)-- (6.06217782649107,-4);
                    \draw [line width=1pt] (0,-6)-- (1.7320508075688772,-7);
                    \draw [line width=1pt] (1.7320508075688772,-7)-- (4.330127018922193,-7);
                    \draw [line width=1pt] (4.330127018922193,-7)-- (6.06217782649107,-6);
                    \draw [line width=1pt] (12.12435565298214,-4)-- (15.155444566227676,-2.75);
                    \draw [line width=1pt] (15.155444566227676,-2.75)-- (18.186533479473212,-4);
                    \draw [line width=1pt] (12.12435565298214,-6)-- (13.856406460551018,-7);
                    \draw [line width=1pt] (13.856406460551018,-7)-- (16.454482671904334,-7);
                    \draw [line width=1pt] (16.454482671904334,-7)-- (18.186533479473212,-6);
                    \draw [line width=1pt] (24.24871130596428,-4)-- (27.279800219209815,-2.75);
                    \draw [line width=1pt] (27.279800219209815,-2.75)-- (30.31088913245535,-4);
                    \draw [line width=1pt] (24.24871130596428,-6)-- (25.980762113533157,-7);
                    \draw [line width=1pt] (25.980762113533157,-7)-- (28.578838324886473,-7);
                    \draw [line width=1pt] (28.578838324886473,-7)-- (30.31088913245535,-6);
                \begin{scriptsize}
                    \draw [fill=xdxdff] (0,0) circle (1.5pt);
                    \draw [fill=xdxdff] (0,2.5) circle (1.5pt);
                    \draw [fill=ududff] (4,2.5) circle (1.5pt);
                    \draw [fill=xdxdff] (4,0) circle (1.5pt);
                    \draw [fill=ududff] (6,3.75) circle (1.5pt);
                    \draw [fill=ududff] (8,2.5) circle (1.5pt);
                    \draw [fill=xdxdff] (8,0) circle (1.5pt);
                    \draw [fill=ududff] (12,2.5) circle (1.5pt);
                    \draw [fill=xdxdff] (12,0) circle (1.5pt);
                    \draw [fill=ududff] (14,3.75) circle (1.5pt);
                    \draw [fill=ududff] (16,2.5) circle (1.5pt);
                    \draw [fill=ududff] (16,0) circle (1.5pt);
                    \draw [fill=ududff] (20,2.5) circle (1.5pt);
                    \draw [fill=xdxdff] (20,0) circle (1.5pt);
                    \draw [fill=ududff] (22,3.75) circle (1.5pt);
                    \draw [fill=ududff] (24,2.5) circle (1.5pt);
                    \draw [fill=ududff] (24,0) circle (1.5pt);
                    \draw [fill=xdxdff] (0,-6) circle (1.5pt);
                    \draw [fill=xdxdff] (0,-4) circle (1.5pt);
                    \draw [fill=ududff] (6.06217782649107,-4) circle (1.5pt);
                    \draw [fill=ududff] (6.06217782649107,-6) circle (1.5pt);
                    \draw [fill=ududff] (7.794228634059947,-3) circle (1.5pt);
                    \draw [fill=ududff] (10.392304845413264,-3) circle (1.5pt);
                    \draw [fill=ududff] (12.12435565298214,-4) circle (1.5pt);
                    \draw [fill=ududff] (12.12435565298214,-6) circle (1.5pt);
                    \draw [fill=ududff] (9.093266739736606,-7.25) circle (1.5pt);
                    \draw [fill=ududff] (18.186533479473212,-4) circle (1.5pt);
                    \draw [fill=ududff] (18.186533479473212,-6) circle (1.5pt);
                    \draw [fill=ududff] (19.918584287042087,-3) circle (1.5pt);
                    \draw [fill=ududff] (22.516660498395403,-3) circle (1.5pt);
                    \draw [fill=ududff] (24.24871130596428,-4) circle (1.5pt);
                    \draw [fill=ududff] (24.24871130596428,-6) circle (1.5pt);
                    \draw [fill=ududff] (21.217622392718745,-7.25) circle (1.5pt);
                    \draw [fill=ududff] (-1.7320508075688772,-3) circle (1.5pt);
                    \draw [fill=ududff] (-4.330127018922193,-3) circle (1.5pt);
                    \draw [fill=ududff] (-6.06217782649107,-4) circle (1.5pt);
                    \draw [fill=ududff] (-6.06217782649107,-6) circle (1.5pt);
                    \draw [fill=ududff] (-3.031088913245535,-7.25) circle (1.5pt);
                    \draw [fill=ududff] (30.31088913245535,-4) circle (1.5pt);
                    \draw [fill=ududff] (30.31088913245535,-6) circle (1.5pt);
                    \draw [fill=ududff] (28,2.5) circle (1.5pt);
                    \draw [fill=xdxdff] (28,0) circle (1.5pt);
                    \draw [fill=ududff] (-2,3.75) circle (1.5pt);
                    \draw [fill=ududff] (-4,2.5) circle (1.5pt);
                    \draw [fill=xdxdff] (-4,0) circle (1.5pt);
                    \draw [fill=ududff] (10,-1.25) circle (1.5pt);
                    \draw [fill=ududff] (18,-1.25) circle (1.5pt);
                    \draw [fill=ududff] (26,-1.25) circle (1.5pt);
                    \draw [fill=xdxdff] (2,-1.25) circle (1.5pt);
                    \draw [fill=ududff] (3.031088913245535,-2.75) circle (1.5pt);
                    \draw [fill=ududff] (1.7320508075688772,-7) circle (1.5pt);
                    \draw [fill=ududff] (4.330127018922193,-7) circle (1.5pt);
                    \draw [fill=ududff] (15.155444566227676,-2.75) circle (1.5pt);
                    \draw [fill=ududff] (13.856406460551018,-7) circle (1.5pt);
                    \draw [fill=ududff] (16.454482671904334,-7) circle (1.5pt);
                    \draw [fill=ududff] (27.279800219209815,-2.75) circle (1.5pt);
                    \draw [fill=ududff] (25.980762113533157,-7) circle (1.5pt);
                    \draw [fill=ududff] (28.578838324886473,-7) circle (1.5pt);
                \end{scriptsize}
            \end{tikzpicture}
            \vspace{-6.5em}
            \caption{Alternate arrangement of odd-sided polygons in line tiling.}
            \label{fig: Alternate arrangement}
        \end{figure}
    \end{enumerate}

    We will discuss the reason for taking two different cases to tackle this particular problem later. Note that the only way in which triangles can be arranged in a line tiling is by alternatively arranging them (see \Cref{fig: Alternate arrangement of triangles}). This line tiling will be dealt with separately when discussing alternatively arranged polygons.

    \begin{figure}[h!]
    \vspace{-5em}
        \centering
            \begin{tikzpicture}[line cap=round,line join=round,>=triangle 45,x=0.5cm,y=0.5cm]
            \clip(-0.9089262832126813,-2.0069579064894896) rectangle (19.347144372775173,9.284611506182573);
                    \draw [line width=1pt] (0,2)-- (3.4641016151377544,2);
                    \draw [line width=1pt] (0,2)-- (1.7320508075688772,5);
                    \draw [line width=1pt] (1.7320508075688772,5)-- (3.4641016151377544,2);
                    \draw [line width=1pt] (1.7320508075688772,5)-- (5.196152422706632,5);
                    \draw [line width=1pt] (5.196152422706632,5)-- (3.4641016151377544,2);
                    \draw [line width=1pt] (3.4641016151377544,2)-- (6.928203230275509,2);
                    \draw [line width=1pt] (6.928203230275509,2)-- (5.196152422706632,5);
                    \draw [line width=1pt] (5.196152422706632,5)-- (8.660254037844386,5);
                    \draw [line width=1pt] (8.660254037844386,5)-- (6.928203230275509,2);
                    \draw [line width=1pt] (6.928203230275509,2)-- (10.392304845413264,2);
                    \draw [line width=1pt] (10.392304845413264,2)-- (8.660254037844386,5);
                    \draw [line width=1pt] (8.660254037844386,5)-- (12.12435565298214,5);
                    \draw [line width=1pt] (12.12435565298214,5)-- (10.392304845413264,2);
                    \draw [line width=1pt] (10.392304845413264,2)-- (13.856406460551018,2);
                    \draw [line width=1pt] (13.856406460551018,2)-- (12.12435565298214,5);
                    \draw [line width=1pt] (12.12435565298214,5)-- (15.588457268119894,5);
                    \draw [line width=1pt] (15.588457268119894,5)-- (13.856406460551018,2);
                    \draw [line width=1pt] (13.856406460551018,2)-- (17.32050807568877,2);
                    \draw [line width=1pt] (17.32050807568877,2)-- (15.588457268119894,5);
                    \draw [line width=1pt] (15.588457268119894,5)-- (19.05255888325765,5);
                    \draw [line width=1pt] (19.05255888325765,5)-- (17.32050807568877,2);
                \begin{scriptsize}
                    \draw [fill=ududff] (0,2) circle (1.5pt);
                    \draw [fill=ududff] (3.4641016151377544,2) circle (1.5pt);
                    \draw [fill=ududff] (1.7320508075688772,5) circle (1.5pt);
                    \draw [fill=ududff] (5.196152422706632,5) circle (1.5pt);
                    \draw [fill=ududff] (6.928203230275509,2) circle (1.5pt);
                    \draw [fill=ududff] (8.660254037844386,5) circle (1.5pt);
                    \draw [fill=ududff] (10.392304845413264,2) circle (1.5pt);
                    \draw [fill=ududff] (12.12435565298214,5) circle (1.5pt);
                    \draw [fill=ududff] (13.856406460551018,2) circle (1.5pt);
                    \draw [fill=ududff] (15.588457268119894,5) circle (1.5pt);
                    \draw [fill=ududff] (17.32050807568877,2) circle (1.5pt);
                    \draw [fill=ududff] (19.05255888325765,5) circle (1.5pt);
                \end{scriptsize}
            \end{tikzpicture}
        \vspace{-5em}
        \caption{Only alternate arrangement of triangles is possible in line tiling}
        \vspace{-1em}
        \label{fig: Alternate arrangement of triangles}
    \end{figure}

    An interesting point to note about these arrangements is that we need to take an even number of polygons for our discussion. If we take an odd number of odd-sided polygons in our line tiling, we end up having an odd number of vertices in the graph, and then the complex $\mathcal{M}_p\left(G\right)$ would be empty.
    
    \subsection*{Labelling of edges, vertices and polygons in the line tiling of odd-sided polygons}

    Let $\mathcal{O}^{\mathcal{S}}_{n,k}$ and $\mathcal{O}^{\mathcal{A}}_{n,k}$ denote the \textit{simple} and \textit{alternate} arrangement of the general line tiling of $2k-\left(2n+1\right)-$gon, respectively, where, $k\geq 2$ and $n\geq 1$. In both cases, $2k$  number of $\left(2n+1\right)-$gon (odd-sided polygons) are in the line tiling.

    We define the labelling for the consecutive polygons in the line tiling as $O_{2j-1}$ and $O_{2j}$, respectively, where $1 \leq j \leq k$ and use the same labelling defined for the edges and vertices as described for the even-sided case with slight modifications; see \Cref{fig: Simple Arrangement Labelling} and \Cref{fig: Alternate Arrangement Labelling}.

    \begin{figure}[h]
        \centering
            \begin{tikzpicture}[line cap=round,line join=round,>=triangle 45,x=0.5cm,y=0.35cm]
            \clip(-2,-6) rectangle (29,10);
                
                \draw [line width=1pt] (0,3)-- (0,0);
                \draw [line width=1pt] (0,0)-- (1.3016512173526733,-2.702906603707256);
                \draw [line width=1pt] (1.3016512173526733,-2.702906603707256)-- (3.6471456647567626,-4.573376009283456);
                \draw [line width=1pt] (3.6471456647567626,-4.573376009283456)-- (6.571929401302232,-5.240938811152399);
                \draw [line width=1pt,dash pattern=on 4pt off 4pt] (6.571929401302232,-5.240938811152399)-- (9.496713137847703,-4.573376009283457);
                \draw [line width=1pt] (9.496713137847703,-4.573376009283457)-- (11.842207585251792,-2.702906603707257);
                \draw [line width=1pt] (11.842207585251792,-2.702906603707257)-- (13.143858802604466,0);
                \draw [line width=1pt] (13.143858802604466,0)-- (13.143858802604466,3);
                \draw [line width=1pt] (13.143858802604466,3)-- (11.842207585251794,5.702906603707255);
                \draw [line width=1pt] (11.842207585251794,5.702906603707255)-- (9.496713137847705,7.573376009283456);
                \draw [line width=1pt,dash pattern=on 4pt off 4pt] (9.496713137847705,7.573376009283456)-- (6.571929401302235,8.2409388111524);
                \draw [line width=1pt] (6.571929401302235,8.2409388111524)-- (3.6471456647567644,7.573376009283457);
                \draw [line width=1pt] (3.6471456647567644,7.573376009283457)-- (1.301651217352676,5.702906603707258);
                \draw [line width=1pt] (1.301651217352676,5.702906603707258)-- (0,3);
                \draw [line width=1pt] (13.143858802604466,3)-- (13.143858802604466,0);
                \draw [line width=1pt] (13.143858802604466,0)-- (14.445510019957139,-2.7029066037072567);
                \draw [line width=1pt] (14.445510019957139,-2.7029066037072567)-- (16.791004467361226,-4.5733760092834554);
                \draw [line width=1pt] (16.791004467361226,-4.5733760092834554)-- (19.715788203906698,-5.2409388111523985);
                \draw [line width=1pt,dash pattern=on 4pt off 4pt] (19.715788203906698,-5.2409388111523985)-- (22.640571940452165,-4.573376009283457);
                \draw [line width=1pt] (22.640571940452165,-4.573376009283457)-- (24.986066387856255,-2.7029066037072575);
                \draw [line width=1pt] (24.986066387856255,-2.7029066037072575)-- (26.28771760520893,0);
                \draw [line width=1pt] (26.28771760520893,0)-- (26.28771760520893,3);
                \draw [line width=1pt] (26.28771760520893,3)-- (24.986066387856255,5.7029066037072536);
                \draw [line width=1pt] (24.986066387856255,5.7029066037072536)-- (22.64057194045217,7.573376009283454);
                \draw [line width=1pt,dash pattern=on 4pt off 4pt] (22.64057194045217,7.573376009283454)-- (19.715788203906698,8.240938811152397);
                \draw [line width=1pt] (19.715788203906698,8.240938811152397)-- (16.79100446736123,7.573376009283454);
                \draw [line width=1pt] (16.79100446736123,7.573376009283454)-- (14.44551001995714,5.702906603707255);
                \draw [line width=1pt] (14.44551001995714,5.702906603707255)-- (13.143858802604466,3);
                \draw (6.5,1.5) node[anchor=center] {\Huge$O_{2j-1}$};
                \draw (-0.6,1.5) node[anchor=center] {\large$a_{j}$};
                \draw (27.3,1.5) node[anchor=center] {\large$a_{j+2}$};
                \draw (19.5,1.5) node[anchor=center] {\Huge$O_{2j}$};
                \draw (12.3,1.5) node[anchor=center] {\large$a_{j+1}$};
                \draw (11.8,4.3) node[anchor=center] {\large$b_{j,n}$};
                \draw (11.5,7.2) node[anchor=center] {\large$b_{j,n-1}$};
                \draw (4.9,8.7) node[anchor=center] {\large$b_{j,3}$};
                \draw (2,7.4) node[anchor=center] {\large$b_{j,2}$};
                \draw (0,4.9) node[anchor=center] {\large$b_{j,1}$};
                \draw (14.9,4.2) node[anchor=center] {\large$b_{j+1,1}$};
                \draw (14.6,7.2) node[anchor=center] {\large$b_{j+1,2}$};
                \draw (18,8.7) node[anchor=center] {\large$b_{j+1,3}$};
                \draw (26.8,4.6) node[anchor=center] {\large$b_{j+1,n}$};
                \draw (25.3,7.2) node[anchor=center] {\large$b_{j+1,n-1}$};
                \draw (-0,-1.6) node[anchor=center] {\large$c_{j,1}$};
                \draw (1.9,-4) node[anchor=center] {\large$c_{j,2}$};
                \draw (5,-5.5) node[anchor=center] {\large$c_{j,3}$};
                \draw (11.6,-4.1) node[anchor=center] {\large$c_{j,n-2}$};
                \draw (11.3,-1) node[anchor=center] {\large$c_{j,n-1}$};
                \draw (14.8,-1.2) node[anchor=center] {\large$c_{j+1,1}$};
                \draw (14.6,-4) node[anchor=center] {\large$c_{j+1,2}$};
                \draw (18.1,-5.4) node[anchor=center] {\large$c_{j+1,3}$};
                \draw (25.4,-4) node[anchor=center] {\large$c_{j+1,n-2}$};
                \draw (27.2,-1.6) node[anchor=center] {\large$c_{j+1,n-1}$};
                \draw [color=xdxdff](-1,3) node[anchor=center] {\large$u_{j,1}$};
                \draw [color=xdxdff](3,8.4) node[anchor=center] {\large$u_{j,3}$};
                \draw [color=xdxdff](0.4,6.2) node[anchor=center] {\large$u_{j,2}$};
                \draw [color=xdxdff](10.3,8.4) node[anchor=center] {\large$u_{j,n-1}$};
                \draw [color=xdxdff](10.9,5.5) node[anchor=center] {\large$u_{j,n}$};
                \draw [color=xdxdff](14.5,2.9) node[anchor=center] {\large$u_{j+1,1}$};
                \draw [color=xdxdff](15.6,8.2) node[anchor=center] {\large$u_{j+1,3}$};
                \draw [color=xdxdff](15.7,5.3) node[anchor=center] {\large$u_{j+1,2}$};
                \draw [color=xdxdff](24,8.3) node[anchor=center] {\large$u_{j+1,n-1}$};
                \draw [color=xdxdff](26.2,6) node[anchor=center] {\large$u_{j+1,n}$};
                \draw [color=xdxdff](27.6,2.8) node[anchor=center] {\large$u_{j+2,1}$};
                \draw [color=xdxdff](-0.9,0) node[anchor=center] {\large$l_{j,1}$};
                \draw [color=xdxdff](0.5,-2.7) node[anchor=center] {\large$l_{j,2}$};
                \draw [color=xdxdff](3,-5.1) node[anchor=center] {\large$l_{j,3}$};
                \draw [color=xdxdff](10.5,-5.3) node[anchor=center] {\large$l_{j,n-2}$};
                \draw [color=xdxdff](10.6,-2.3) node[anchor=center] {\large$l_{j,n-1}$};
                \draw [color=xdxdff](14.5,0) node[anchor=center] {\large$l_{j+1,1}$};
                \draw [color=xdxdff](15.8,-2.5) node[anchor=center] {\large$l_{j+1,2}$};
                \draw [color=xdxdff](15.8,-5) node[anchor=center] {\large$l_{j+1,3}$};
                \draw [color=xdxdff](24.2,-5.2) node[anchor=center] {\large$l_{j+1,n-2}$};
                \draw [color=xdxdff](26.7,-2.9) node[anchor=center] {\large$l_{j+1,n-1}$};
                \draw [color=xdxdff](27.5,-0.1) node[anchor=center] {\large$l_{j+2,1}$};
            \begin{scriptsize}
                \draw [fill=xdxdff] (0,3) circle (2.5pt);
                \draw [fill=xdxdff] (0,0) circle (2.5pt);
                \draw [fill=xdxdff] (1.3016512173526733,-2.702906603707256) circle (2.5pt);
                \draw [fill=xdxdff] (3.6471456647567626,-4.573376009283456) circle (2.5pt);
                \draw [fill=xdxdff] (6.571929401302232,-5.240938811152399) circle (2.5pt);
                \draw [fill=xdxdff] (9.496713137847703,-4.573376009283457) circle (2.5pt);
                \draw [fill=xdxdff] (11.842207585251792,-2.702906603707257) circle (2.5pt);
                \draw [fill=xdxdff] (13.143858802604466,0) circle (2.5pt);
                \draw [fill=xdxdff] (13.143858802604466,3) circle (2.5pt);
                \draw [fill=xdxdff] (11.842207585251794,5.702906603707255) circle (2.5pt);
                \draw [fill=xdxdff] (9.496713137847705,7.573376009283456) circle (2.5pt);
                \draw [fill=xdxdff] (6.571929401302235,8.2409388111524) circle (2.5pt);
                \draw [fill=xdxdff] (3.6471456647567644,7.573376009283457) circle (2.5pt);
                \draw [fill=xdxdff] (1.301651217352676,5.702906603707258) circle (2.5pt);
                \draw [fill=xdxdff] (14.445510019957139,-2.7029066037072567) circle (2.5pt);
                \draw [fill=xdxdff] (16.791004467361226,-4.5733760092834554) circle (2.5pt);
                \draw [fill=xdxdff] (19.715788203906698,-5.2409388111523985) circle (2.5pt);
                \draw [fill=xdxdff] (22.640571940452165,-4.573376009283457) circle (2.5pt);
                \draw [fill=xdxdff] (24.986066387856255,-2.7029066037072575) circle (2.5pt);
                \draw [fill=xdxdff] (26.28771760520893,0) circle (2.5pt);
                \draw [fill=xdxdff] (26.28771760520893,3) circle (2.5pt);
                \draw [fill=xdxdff] (24.986066387856255,5.7029066037072536) circle (2.5pt);
                \draw [fill=xdxdff] (22.64057194045217,7.573376009283454) circle (2.5pt);
                \draw [fill=xdxdff] (19.715788203906698,8.240938811152397) circle (2.5pt);
                \draw [fill=xdxdff] (16.79100446736123,7.573376009283454) circle (2.5pt);
                \draw [fill=xdxdff] (14.44551001995714,5.702906603707255) circle (2.5pt);
            \end{scriptsize}
            \end{tikzpicture}
        \caption{Labelling of edges and polygons in $\mathcal{O}^{\mathcal{S}}_{n,k}$}
        \label{fig: Simple Arrangement Labelling}
    \end{figure}

    \begin{figure}[h!]
        \centering
            \begin{tikzpicture}[line cap=round,line join=round,>=triangle 45,x=0.5cm,y=0.35cm]
            \clip(-2,-6) rectangle (29,10);
                
                \draw [line width=1pt] (0,3)-- (0,0);
                \draw [line width=1pt] (0,0)-- (1.3016512173526733,-2.702906603707256);
                \draw [line width=1pt] (1.3016512173526733,-2.702906603707256)-- (3.6471456647567626,-4.573376009283456);
                \draw [line width=1pt] (3.6471456647567626,-4.573376009283456)-- (6.571929401302232,-5.240938811152399);
                \draw [line width=1pt,dash pattern=on 4pt off 4pt] (6.571929401302232,-5.240938811152399)-- (9.496713137847703,-4.573376009283457);
                \draw [line width=1pt] (9.496713137847703,-4.573376009283457)-- (11.842207585251792,-2.702906603707257);
                \draw [line width=1pt] (11.842207585251792,-2.702906603707257)-- (13.143858802604466,0);
                \draw [line width=1pt] (13.143858802604466,0)-- (13.143858802604466,3);
                \draw [line width=1pt] (13.143858802604466,3)-- (11.842207585251794,5.702906603707255);
                \draw [line width=1pt] (11.842207585251794,5.702906603707255)-- (9.496713137847705,7.573376009283456);
                \draw [line width=1pt,dash pattern=on 4pt off 4pt] (9.496713137847705,7.573376009283456)-- (6.571929401302235,8.2409388111524);
                \draw [line width=1pt] (6.571929401302235,8.2409388111524)-- (3.6471456647567644,7.573376009283457);
                \draw [line width=1pt] (3.6471456647567644,7.573376009283457)-- (1.301651217352676,5.702906603707258);
                \draw [line width=1pt] (1.301651217352676,5.702906603707258)-- (0,3);
                \draw [line width=1pt] (13.143858802604466,3)-- (13.143858802604466,0);
                \draw [line width=1pt] (13.143858802604466,0)-- (14.445510019957139,-2.7029066037072567);
                \draw [line width=1pt] (14.445510019957139,-2.7029066037072567)-- (16.791004467361226,-4.5733760092834554);
                \draw [line width=1pt] (16.791004467361226,-4.5733760092834554)-- (19.715788203906698,-5.2409388111523985);
                \draw [line width=1pt,dash pattern=on 4pt off 4pt] (19.715788203906698,-5.2409388111523985)-- (22.640571940452165,-4.573376009283457);
                \draw [line width=1pt] (22.640571940452165,-4.573376009283457)-- (24.986066387856255,-2.7029066037072575);
                \draw [line width=1pt] (24.986066387856255,-2.7029066037072575)-- (26.28771760520893,0);
                \draw [line width=1pt] (26.28771760520893,0)-- (26.28771760520893,3);
                \draw [line width=1pt] (26.28771760520893,3)-- (24.986066387856255,5.7029066037072536);
                \draw [line width=1pt] (24.986066387856255,5.7029066037072536)-- (22.64057194045217,7.573376009283454);
                \draw [line width=1pt,dash pattern=on 4pt off 4pt] (22.64057194045217,7.573376009283454)-- (19.715788203906698,8.240938811152397);
                \draw [line width=1pt] (19.715788203906698,8.240938811152397)-- (16.79100446736123,7.573376009283454);
                \draw [line width=1pt] (16.79100446736123,7.573376009283454)-- (14.44551001995714,5.702906603707255);
                \draw [line width=1pt] (14.44551001995714,5.702906603707255)-- (13.143858802604466,3);
                \draw (6.5,1.5) node[anchor=center] {\Huge$O_{2j-1}$};
                \draw (-0.6,1.5) node[anchor=center] {\large$a_{j}$};
                \draw (27.3,1.5) node[anchor=center] {\large$a_{j+2}$};
                \draw (19.5,1.5) node[anchor=center] {\Huge$O_{2j}$};
                \draw (12.3,1.5) node[anchor=center] {\large$a_{j+1}$};
                \draw (11.8,4.3) node[anchor=center] {\large$b_{j,n}$};
                \draw (11.5,7.2) node[anchor=center] {\large$b_{j,n-1}$};
                \draw (4.9,8.7) node[anchor=center] {\large$b_{j,3}$};
                \draw (2,7.4) node[anchor=center] {\large$b_{j,2}$};
                \draw (0,4.9) node[anchor=center] {\large$b_{j,1}$};
                \draw (14.9,4.2) node[anchor=center] {\large$b_{j+1,1}$};
                \draw (14.6,7.2) node[anchor=center] {\large$b_{j+1,2}$};
                \draw (18,8.7) node[anchor=center] {\large$b_{j+1,3}$};
                \draw (27.1,4.6) node[anchor=center] {\large$b_{j+1,n-1}$};
                \draw (25.3,7.2) node[anchor=center] {\large$b_{j+1,n-2}$};
                \draw (-0,-1.6) node[anchor=center] {\large$c_{j,1}$};
                \draw (1.9,-4) node[anchor=center] {\large$c_{j,2}$};
                \draw (5,-5.5) node[anchor=center] {\large$c_{j,3}$};
                \draw (11.6,-4.1) node[anchor=center] {\large$c_{j,n-2}$};
                \draw (11.3,-1) node[anchor=center] {\large$c_{j,n-1}$};
                \draw (14.8,-1.2) node[anchor=center] {\large$c_{j+1,1}$};
                \draw (14.6,-4) node[anchor=center] {\large$c_{j+1,2}$};
                \draw (18.1,-5.4) node[anchor=center] {\large$c_{j+1,3}$};
                \draw (25.4,-4) node[anchor=center] {\large$c_{j+1,n-1}$};
                \draw (27,-1.6) node[anchor=center] {\large$c_{j+1,n}$};
                \draw [color=xdxdff](-1,3) node[anchor=center] {\large$u_{j,1}$};
                \draw [color=xdxdff](3,8.4) node[anchor=center] {\large$u_{j,3}$};
                \draw [color=xdxdff](0.4,6.2) node[anchor=center] {\large$u_{j,2}$};
                \draw [color=xdxdff](10.3,8.4) node[anchor=center] {\large$u_{j,n-1}$};
                \draw [color=xdxdff](10.9,5.5) node[anchor=center] {\large$u_{j,n}$};
                \draw [color=xdxdff](14.5,2.9) node[anchor=center] {\large$u_{j+1,1}$};
                \draw [color=xdxdff](15.6,8.2) node[anchor=center] {\large$u_{j+1,3}$};
                \draw [color=xdxdff](15.7,5.3) node[anchor=center] {\large$u_{j+1,2}$};
                \draw [color=xdxdff](24,8.3) node[anchor=center] {\large$u_{j+1,n-2}$};
                \draw [color=xdxdff](26.5,6) node[anchor=center] {\large$u_{j+1,n-1}$};
                \draw [color=xdxdff](27.6,2.8) node[anchor=center] {\large$u_{j+2,1}$};
                \draw [color=xdxdff](-0.9,0) node[anchor=center] {\large$l_{j,1}$};
                \draw [color=xdxdff](0.5,-2.7) node[anchor=center] {\large$l_{j,2}$};
                \draw [color=xdxdff](3,-5.1) node[anchor=center] {\large$l_{j,3}$};
                \draw [color=xdxdff](10.5,-5.3) node[anchor=center] {\large$l_{j,n-2}$};
                \draw [color=xdxdff](10.6,-2.3) node[anchor=center] {\large$l_{j,n-1}$};
                \draw [color=xdxdff](14.5,0) node[anchor=center] {\large$l_{j+1,1}$};
                \draw [color=xdxdff](15.8,-2.5) node[anchor=center] {\large$l_{j+1,2}$};
                \draw [color=xdxdff](15.8,-5) node[anchor=center] {\large$l_{j+1,3}$};
                \draw [color=xdxdff](24.2,-5.2) node[anchor=center] {\large$l_{j+1,n-1}$};
                \draw [color=xdxdff](26.5,-2.9) node[anchor=center] {\large$l_{j+1,n}$};
                \draw [color=xdxdff](27.5,-0.1) node[anchor=center] {\large$l_{j+2,1}$};
            \begin{scriptsize}
                \draw [fill=xdxdff] (0,3) circle (2.5pt);
                \draw [fill=xdxdff] (0,0) circle (2.5pt);
                \draw [fill=xdxdff] (1.3016512173526733,-2.702906603707256) circle (2.5pt);
                \draw [fill=xdxdff] (3.6471456647567626,-4.573376009283456) circle (2.5pt);
                \draw [fill=xdxdff] (6.571929401302232,-5.240938811152399) circle (2.5pt);
                \draw [fill=xdxdff] (9.496713137847703,-4.573376009283457) circle (2.5pt);
                \draw [fill=xdxdff] (11.842207585251792,-2.702906603707257) circle (2.5pt);
                \draw [fill=xdxdff] (13.143858802604466,0) circle (2.5pt);
                \draw [fill=xdxdff] (13.143858802604466,3) circle (2.5pt);
                \draw [fill=xdxdff] (11.842207585251794,5.702906603707255) circle (2.5pt);
                \draw [fill=xdxdff] (9.496713137847705,7.573376009283456) circle (2.5pt);
                \draw [fill=xdxdff] (6.571929401302235,8.2409388111524) circle (2.5pt);
                \draw [fill=xdxdff] (3.6471456647567644,7.573376009283457) circle (2.5pt);
                \draw [fill=xdxdff] (1.301651217352676,5.702906603707258) circle (2.5pt);
                \draw [fill=xdxdff] (14.445510019957139,-2.7029066037072567) circle (2.5pt);
                \draw [fill=xdxdff] (16.791004467361226,-4.5733760092834554) circle (2.5pt);
                \draw [fill=xdxdff] (19.715788203906698,-5.2409388111523985) circle (2.5pt);
                \draw [fill=xdxdff] (22.640571940452165,-4.573376009283457) circle (2.5pt);
                \draw [fill=xdxdff] (24.986066387856255,-2.7029066037072575) circle (2.5pt);
                \draw [fill=xdxdff] (26.28771760520893,0) circle (2.5pt);
                \draw [fill=xdxdff] (26.28771760520893,3) circle (2.5pt);
                \draw [fill=xdxdff] (24.986066387856255,5.7029066037072536) circle (2.5pt);
                \draw [fill=xdxdff] (22.64057194045217,7.573376009283454) circle (2.5pt);
                \draw [fill=xdxdff] (19.715788203906698,8.240938811152397) circle (2.5pt);
                \draw [fill=xdxdff] (16.79100446736123,7.573376009283454) circle (2.5pt);
                \draw [fill=xdxdff] (14.44551001995714,5.702906603707255) circle (2.5pt);
            \end{scriptsize}
            \end{tikzpicture}
        \caption{Labelling of edges and polygons in $\mathcal{O}^{\mathcal{A}}_{n,k}$}
        \label{fig: Alternate Arrangement Labelling}
    \end{figure}
    
    \subsection{Homotopy type of the perfect matching complex of \texorpdfstring{$\mathcal{O}^{\mathcal{S}}_{n,k}$ and $\mathcal{O}^{\mathcal{A}}_{n,k}$}{}}

    We can now discuss a lemma that will be useful in proving the required result for the alternate arrangement of polygons. Let $\pma$ denote the perfect matching complex of the general line tiling of $2k-\left(2n+1\right)-$gon when arranged alternatively.

    \begin{lemma}\label{lemma 5.1}
        For any facet (perfect matching) $\tau \in \pma$, the edges $a_{{2i}} \notin \tau$, for all $1 \leq i \leq k$.
    \end{lemma}

    \begin{remark}
        The above result also holds for the simple arrangement of odd-sided polygons in line tiling, and the proof is similar to this proof.
    \end{remark}

    \begin{proof}[Proof of \Cref{lemma 5.1}]
        Let us assume to the contrary that this is not true. Therefore, $a_{{2i}} \in \tau$, for some perfect matching $\tau$ of $\mathcal{O}^{\mathcal{A}}_{n,k}$ and for some index $i \in \left\{1, 2,\dots, k\right\}$.

        Note that $a_{{2i}}$ is the edge where the polygons $O_{2i-1}$ and $O_{2i}$ are attached, and there are an odd number of polygons to the left of the polygon $O_{2i}$, namely, $O_1, O_2, \dots, O_{2i-1}$. Now, if $a_{{2i}} \in \tau$, then we have an odd number of vertices on the left side of $a_{{2i}}$, which contradicts $\tau$ being a perfect matching since there will be a vertex left uncovered by $\tau$.
    \end{proof}

    The statement of the above result says that for every perfect matching $\tau$ of $\mathcal{O}^{\mathcal{A}}_{n,k}$, $\tau$ omit the edges $a_{{2i}}$, for all $1\leq i \leq k$. This is an important observation because now we can think of excluding the edges $a_{{2i}}$ from the graph as it does not contribute to any perfect matching of $\mathcal{O}^{\mathcal{A}}_{n,k}$ (see \Cref{fig: Omitting edges}). However, the structure thus obtained is similar to that of an even-sided polygon line tiling whose corresponding result we have already proved, \textit{i.e.}, \Cref{thm: homotopy of 2n-gon}.


    \begin{figure}[h]
        \vspace{-1em}
        \centering
            \hspace*{-1em}\begin{tikzpicture}[line cap=round,line join=round,>=triangle 45,x=0.3cm,y=0.25cm]
            \clip(-8.953571863424697,-14.182849343130211) rectangle (50.397192337223004,18.90171503503709);
                    \draw [line width=1pt] (13.856406460551018,15)-- (13.856406460551018,13);
                    \draw [line width=1pt] (13.856406460551018,15)-- (15.588457268119894,16);
                    \draw [line width=1pt] (15.588457268119894,16)-- (18.18653347947321,16);
                    \draw [line width=1pt] (18.18653347947321,16)-- (19.918584287042084,15);
                    \draw [line width=2pt,color=ffqqqq] (19.918584287042084,15)-- (19.918584287042084,13);
                    \draw [line width=1pt] (13.856406460551018,13)-- (16.88749537379655,11.75);
                    \draw [line width=1pt] (16.88749537379655,11.75)-- (19.918584287042084,13);
                    \draw [line width=1pt] (25.980762113533157,15)-- (25.980762113533157,13);
                    \draw [line width=1pt] (25.980762113533157,15)-- (27.71281292110203,16);
                    \draw [line width=1pt] (27.71281292110203,16)-- (30.310889132455348,16);
                    \draw [line width=1pt] (30.310889132455348,16)-- (32.042939940024226,15);
                    \draw [line width=2pt,color=ffqqqq] (32.042939940024226,15)-- (32.042939940024226,13);
                    \draw [line width=1pt] (29.01185102677869,11.75)-- (32.042939940024226,13);
                    \draw [line width=1pt] (29.01185102677869,11.75)-- (25.980762113533157,13);
                    \draw [line width=1pt] (6.023453104972012,16.005304061783896)-- (3.4641016151377553,16);
                    \draw [line width=1pt] (3.4641016151377553,16)-- (1.7320508075688772,15);
                    \draw [line width=1pt] (1.7320508075688772,15)-- (1.7320508075688772,13);
                    \draw [line width=1pt] (4.763139720814414,11.75)-- (1.7320508075688772,13);
                    \draw [line width=1pt] (38.10511776651529,15)-- (38.10511776651529,13);
                    \draw [line width=1pt] (10.825317547305483,16.25)-- (13.856406460551018,15);
                    \draw [line width=1pt] (9.526279441628825,12)-- (12.12435565298214,12);
                    \draw [line width=1pt] (12.12435565298214,12)-- (13.856406460551018,13);
                    \draw [line width=1pt] (19.918584287042084,15)-- (22.949673200287624,16.25);
                    \draw [line width=1pt] (22.949673200287624,16.25)-- (25.980762113533157,15);
                    \draw [line width=1pt] (19.918584287042084,13)-- (21.650635094610966,12);
                    \draw [line width=1pt] (21.650635094610966,12)-- (24.24871130596428,12);
                    \draw [line width=1pt] (24.24871130596428,12)-- (25.980762113533157,13);
                    \draw [line width=1pt] (32.042939940024226,15)-- (35.07402885326976,16.25);
                    \draw [line width=1pt] (35.07402885326976,16.25)-- (38.10511776651529,15);
                    \draw [line width=1pt] (32.042939940024226,13)-- (33.7749907475931,12);
                    \draw [line width=1pt] (33.7749907475931,12)-- (36.37306695894642,12);
                    \draw [line width=1pt] (36.37306695894642,12)-- (38.10511776651529,13);
                    \draw [line width=1pt] (6.023453104972012,16.005304061783896)-- (7.794228634059949,15);
                    \draw [line width=2pt,color=ffqqqq] (7.794228634059949,15)-- (7.794228634059949,13);
                    \draw [line width=1pt] (7.794228634059949,13)-- (4.763139720814414,11.75);
                    \draw [line width=1pt] (7.794228634059949,15)-- (10.825317547305483,16.25);
                    \draw [line width=1pt] (7.794228634059949,13)-- (9.526279441628825,12);
                    \draw [line width=1pt] (13.85640646055102,7)-- (13.85640646055102,5);
                    \draw [line width=1pt] (13.85640646055102,7)-- (15.588457268119898,8);
                    \draw [line width=1pt] (15.588457268119898,8)-- (18.18653347947321,8);
                    \draw [line width=1pt] (18.18653347947321,8)-- (19.918584287042084,7);
                    \draw [line width=1pt] (13.85640646055102,5)-- (16.88749537379655,3.75);
                    \draw [line width=1pt] (16.88749537379655,3.75)-- (19.918584287042084,5);
                    \draw [line width=1pt] (25.980762113533153,7)-- (25.980762113533153,5);
                    \draw [line width=1pt] (25.980762113533153,7)-- (27.71281292110203,8);
                    \draw [line width=1pt] (27.71281292110203,8)-- (30.310889132455348,8);
                    \draw [line width=1pt] (30.310889132455348,8)-- (32.04293994002422,7);
                    \draw [line width=1pt] (29.01185102677869,3.75)-- (32.04293994002422,5);
                    \draw [line width=1pt] (29.01185102677869,3.75)-- (25.980762113533153,5);
                    \draw [line width=1pt] (6.02345310497201,8.005304061783896)-- (3.464101615137755,8);
                    \draw [line width=1pt] (3.464101615137755,8)-- (1.7320508075688772,7);
                    \draw [line width=1pt] (1.7320508075688772,7)-- (1.7320508075688772,5);
                    \draw [line width=1pt] (4.763139720814413,3.75)-- (1.7320508075688772,5);
                    \draw [line width=1pt] (38.10511776651529,7)-- (38.10511776651529,5);
                    \draw [line width=1pt] (10.825317547305486,8.25)-- (13.85640646055102,7);
                    \draw [line width=1pt] (9.526279441628828,4)-- (12.124355652982144,4);
                    \draw [line width=1pt] (12.124355652982144,4)-- (13.85640646055102,5);
                    \draw [line width=1pt] (19.918584287042084,7)-- (22.94967320028762,8.25);
                    \draw [line width=1pt] (22.94967320028762,8.25)-- (25.980762113533153,7);
                    \draw [line width=1pt] (19.918584287042084,5)-- (21.650635094610962,4);
                    \draw [line width=1pt] (21.650635094610962,4)-- (24.248711305964278,4);
                    \draw [line width=1pt] (24.248711305964278,4)-- (25.980762113533153,5);
                    \draw [line width=1pt] (32.04293994002422,7)-- (35.07402885326976,8.25);
                    \draw [line width=1pt] (35.07402885326976,8.25)-- (38.10511776651529,7);
                    \draw [line width=1pt] (32.04293994002422,5)-- (33.7749907475931,4);
                    \draw [line width=1pt] (33.7749907475931,4)-- (36.37306695894642,4);
                    \draw [line width=1pt] (36.37306695894642,4)-- (38.10511776651529,5);
                    \draw [line width=1pt] (6.02345310497201,8.005304061783896)-- (7.794228634059948,7);
                    \draw [line width=1pt] (7.794228634059948,5)-- (4.763139720814413,3.75);
                    \draw [line width=1pt] (7.794228634059948,7)-- (10.825317547305486,8.25);
                    \draw [line width=1pt] (7.794228634059948,5)-- (9.526279441628828,4);
                    \draw [line width=1pt] (2.6669688996421703,-4.5068478400775405)-- (2.6669688996421703,-7.506847840077546);
                    \draw [line width=1pt] (2.6669688996421703,-7.506847840077546)-- (4.166968899642173,-10.104924051430867);
                    \draw [line width=1pt] (4.166968899642173,-10.104924051430867)-- (6.765045110995495,-11.60492405143087);
                    \draw [line width=1pt] (6.765045110995495,-11.60492405143087)-- (9.7650451109955,-11.60492405143087);
                    \draw [line width=1pt] (9.7650451109955,-11.60492405143087)-- (12.363121322348821,-10.104924051430869);
                    \draw [line width=1pt] (12.363121322348821,-10.104924051430869)-- (13.863121322348825,-7.506847840077547);
                    \draw [line width=1pt] (13.863121322348825,-7.506847840077547)-- (13.863121322348825,-4.506847840077541);
                    \draw [line width=1pt] (13.863121322348825,-4.506847840077541)-- (12.363121322348825,-1.90877162872422);
                    \draw [line width=1pt] (12.363121322348825,-1.90877162872422)-- (9.765045110995501,-0.4087716287242156);
                    \draw [line width=1pt] (9.765045110995501,-0.4087716287242156)-- (6.765045110995499,-0.4087716287242147);
                    \draw [line width=1pt] (6.765045110995499,-0.4087716287242147)-- (4.166968899642175,-1.9087716287242165);
                    \draw [line width=1pt] (4.166968899642175,-1.9087716287242165)-- (2.6669688996421703,-4.5068478400775405);
                    \draw [line width=1pt] (13.863121322348825,-4.506847840077541)-- (13.863121322348825,-7.506847840077547);
                    \draw [line width=1pt] (13.863121322348825,-7.506847840077547)-- (15.363121322348828,-10.104924051430867);
                    \draw [line width=1pt] (15.363121322348828,-10.104924051430867)-- (17.961197533702148,-11.60492405143087);
                    \draw [line width=1pt] (17.961197533702148,-11.60492405143087)-- (20.96119753370215,-11.60492405143087);
                    \draw [line width=1pt] (20.96119753370215,-11.60492405143087)-- (23.559273745055474,-10.104924051430869);
                    \draw [line width=1pt] (23.559273745055474,-10.104924051430869)-- (25.059273745055478,-7.506847840077548);
                    \draw [line width=1pt] (25.059273745055478,-7.506847840077548)-- (25.059273745055478,-4.506847840077542);
                    \draw [line width=1pt] (25.059273745055478,-4.506847840077542)-- (23.559273745055478,-1.9087716287242218);
                    \draw [line width=1pt] (23.559273745055478,-1.9087716287242218)-- (20.961197533702155,-0.40877162872421735);
                    \draw [line width=1pt] (20.961197533702155,-0.40877162872421735)-- (17.96119753370215,-0.40877162872421646);
                    \draw [line width=1pt] (17.96119753370215,-0.40877162872421646)-- (15.363121322348828,-1.9087716287242182);
                    \draw [line width=1pt] (15.363121322348828,-1.9087716287242182)-- (13.863121322348825,-4.506847840077541);
                    \draw [line width=1pt] (25.059273745055478,-4.506847840077542)-- (25.059273745055478,-7.506847840077548);
                    \draw [line width=1pt] (25.059273745055478,-7.506847840077548)-- (26.55927374505548,-10.104924051430867);
                    \draw [line width=1pt] (26.55927374505548,-10.104924051430867)-- (29.1573499564088,-11.60492405143087);
                    \draw [line width=1pt] (29.1573499564088,-11.60492405143087)-- (32.15734995640881,-11.60492405143087);
                    \draw [line width=1pt] (32.15734995640881,-11.60492405143087)-- (34.755426167762124,-10.10492405143087);
                    \draw [line width=1pt] (34.755426167762124,-10.10492405143087)-- (36.25542616776213,-7.5068478400775485);
                    \draw [line width=1pt] (36.25542616776213,-7.5068478400775485)-- (36.25542616776213,-4.506847840077543);
                    \draw [line width=1pt] (36.25542616776213,-4.506847840077543)-- (34.75542616776213,-1.9087716287242227);
                    \draw [line width=1pt] (34.75542616776213,-1.9087716287242227)-- (32.15734995640881,-0.40877162872421824);
                    \draw [line width=1pt] (32.15734995640881,-0.40877162872421824)-- (29.157349956408805,-0.40877162872421735);
                    \draw [line width=1pt] (29.157349956408805,-0.40877162872421735)-- (26.55927374505548,-1.9087716287242191);
                    \draw [line width=1pt] (26.55927374505548,-1.9087716287242191)-- (25.059273745055478,-4.506847840077542);
                    \draw [line width=1.5pt,dash pattern=on 3pt off 3pt,{Stealth[]}-] (19.918584287042087,0.5)-- (19.918584287042087,3);
                   
                    \draw [line width=1.5pt,dash pattern=on 3pt off 3pt,{Stealth[]}-] (19.918584287042087,8.5)-- (19.918584287042087,11);
                \begin{scriptsize}
                    \draw [fill=ududff] (13.856406460551018,15) circle (1.5pt);
                    \draw [fill=ududff] (13.856406460551018,13) circle (1.5pt);
                    \draw [fill=ududff] (15.588457268119894,16) circle (1.5pt);
                    \draw [fill=ududff] (18.18653347947321,16) circle (1.5pt);
                    \draw [fill=ududff] (19.918584287042084,15) circle (1.5pt);
                    \draw [fill=ududff] (19.918584287042084,13) circle (1.5pt);
                    \draw [fill=ududff] (16.88749537379655,11.75) circle (1.5pt);
                    \draw [fill=ududff] (25.980762113533157,15) circle (1.5pt);
                    \draw [fill=ududff] (25.980762113533157,13) circle (1.5pt);
                    \draw [fill=ududff] (27.71281292110203,16) circle (1.5pt);
                    \draw [fill=ududff] (30.310889132455348,16) circle (1.5pt);
                    \draw [fill=ududff] (32.042939940024226,15) circle (1.5pt);
                    \draw [fill=ududff] (32.042939940024226,13) circle (1.5pt);
                    \draw [fill=ududff] (29.01185102677869,11.75) circle (1.5pt);
                    \draw [fill=ududff] (6.023453104972012,16.005304061783896) circle (1.5pt);
                    \draw [fill=ududff] (3.4641016151377553,16) circle (1.5pt);
                    \draw [fill=ududff] (1.7320508075688772,15) circle (1.5pt);
                    \draw [fill=ududff] (1.7320508075688772,13) circle (1.5pt);
                    \draw [fill=ududff] (4.763139720814414,11.75) circle (1.5pt);
                    \draw [fill=ududff] (38.10511776651529,15) circle (1.5pt);
                    \draw [fill=ududff] (38.10511776651529,13) circle (1.5pt);
                    \draw [fill=ududff] (10.825317547305483,16.25) circle (1.5pt);
                    \draw [fill=ududff] (9.526279441628825,12) circle (1.5pt);
                    \draw [fill=ududff] (12.12435565298214,12) circle (1.5pt);
                    \draw [fill=ududff] (22.949673200287624,16.25) circle (1.5pt);
                    \draw [fill=ududff] (21.650635094610966,12) circle (1.5pt);
                    \draw [fill=ududff] (24.24871130596428,12) circle (1.5pt);
                    \draw [fill=ududff] (35.07402885326976,16.25) circle (1.5pt);
                    \draw [fill=ududff] (33.7749907475931,12) circle (1.5pt);
                    \draw [fill=ududff] (36.37306695894642,12) circle (1.5pt);
                    \draw [fill=ududff] (7.794228634059949,15) circle (1.5pt);
                    \draw [fill=ududff] (7.794228634059949,13) circle (1.5pt);
                    \draw [fill=ududff] (13.85640646055102,7) circle (1.5pt);
                    \draw [fill=ududff] (13.85640646055102,5) circle (1.5pt);
                    \draw [fill=ududff] (15.588457268119898,8) circle (1.5pt);
                    \draw [fill=ududff] (18.18653347947321,8) circle (1.5pt);
                    \draw [fill=ududff] (19.918584287042084,7) circle (1.5pt);
                    \draw [fill=ududff] (19.918584287042084,5) circle (1.5pt);
                    \draw [fill=ududff] (16.88749537379655,3.75) circle (1.5pt);
                    \draw [fill=ududff] (25.980762113533153,7) circle (1.5pt);
                    \draw [fill=ududff] (25.980762113533153,5) circle (1.5pt);
                    \draw [fill=ududff] (27.71281292110203,8) circle (1.5pt);
                    \draw [fill=ududff] (30.310889132455348,8) circle (1.5pt);
                    \draw [fill=ududff] (32.04293994002422,7) circle (1.5pt);
                    \draw [fill=ududff] (32.04293994002422,5) circle (1.5pt);
                    \draw [fill=ududff] (29.01185102677869,3.75) circle (1.5pt);
                    \draw [fill=ududff] (6.02345310497201,8.005304061783896) circle (1.5pt);
                    \draw [fill=ududff] (3.464101615137755,8) circle (1.5pt);
                    \draw [fill=ududff] (1.7320508075688772,7) circle (1.5pt);
                    \draw [fill=ududff] (1.7320508075688772,5) circle (1.5pt);
                    \draw [fill=ududff] (4.763139720814413,3.75) circle (1.5pt);
                    \draw [fill=ududff] (38.10511776651529,7) circle (1.5pt);
                    \draw [fill=ududff] (38.10511776651529,5) circle (1.5pt);
                    \draw [fill=ududff] (10.825317547305486,8.25) circle (1.5pt);
                    \draw [fill=ududff] (9.526279441628828,4) circle (1.5pt);
                    \draw [fill=ududff] (12.124355652982144,4) circle (1.5pt);
                    \draw [fill=ududff] (22.94967320028762,8.25) circle (1.5pt);
                    \draw [fill=ududff] (21.650635094610962,4) circle (1.5pt);
                    \draw [fill=ududff] (24.248711305964278,4) circle (1.5pt);
                    \draw [fill=ududff] (35.07402885326976,8.25) circle (1.5pt);
                    \draw [fill=ududff] (33.7749907475931,4) circle (1.5pt);
                    \draw [fill=ududff] (36.37306695894642,4) circle (1.5pt);
                    \draw [fill=ududff] (7.794228634059948,7) circle (1.5pt);
                    \draw [fill=ududff] (7.794228634059948,5) circle (1.5pt);
                    \draw [fill=ududff] (2.6669688996421703,-4.5068478400775405) circle (1.5pt);
                    \draw [fill=ududff] (2.6669688996421703,-7.506847840077546) circle (1.5pt);
                    \draw [fill=ududff] (4.166968899642173,-10.104924051430867) circle (1.5pt);
                    \draw [fill=ududff] (6.765045110995495,-11.60492405143087) circle (1.5pt);
                    \draw [fill=ududff] (9.7650451109955,-11.60492405143087) circle (1.5pt);
                    \draw [fill=ududff] (12.363121322348821,-10.104924051430869) circle (1.5pt);
                    \draw [fill=ududff] (13.863121322348825,-7.506847840077547) circle (1.5pt);
                    \draw [fill=ududff] (13.863121322348825,-4.506847840077541) circle (1.5pt);
                    \draw [fill=ududff] (12.363121322348825,-1.90877162872422) circle (1.5pt);
                    \draw [fill=ududff] (9.765045110995501,-0.4087716287242156) circle (1.5pt);
                    \draw [fill=ududff] (6.765045110995499,-0.4087716287242147) circle (1.5pt);
                    \draw [fill=ududff] (4.166968899642175,-1.9087716287242165) circle (1.5pt);
                    \draw [fill=ududff] (15.363121322348828,-10.104924051430867) circle (1.5pt);
                    \draw [fill=ududff] (17.961197533702148,-11.60492405143087) circle (1.5pt);
                    \draw [fill=ududff] (20.96119753370215,-11.60492405143087) circle (1.5pt);
                    \draw [fill=ududff] (23.559273745055474,-10.104924051430869) circle (1.5pt);
                    \draw [fill=ududff] (25.059273745055478,-7.506847840077548) circle (1.5pt);
                    \draw [fill=ududff] (25.059273745055478,-4.506847840077542) circle (1.5pt);
                    \draw [fill=ududff] (23.559273745055478,-1.9087716287242218) circle (1.5pt);
                    \draw [fill=ududff] (20.961197533702155,-0.40877162872421735) circle (1.5pt);
                    \draw [fill=ududff] (17.96119753370215,-0.40877162872421646) circle (1.5pt);
                    \draw [fill=ududff] (15.363121322348828,-1.9087716287242182) circle (1.5pt);
                    \draw [fill=ududff] (26.55927374505548,-10.104924051430867) circle (1.5pt);
                    \draw [fill=ududff] (29.1573499564088,-11.60492405143087) circle (1.5pt);
                    \draw [fill=ududff] (32.15734995640881,-11.60492405143087) circle (1.5pt);
                    \draw [fill=ududff] (34.755426167762124,-10.10492405143087) circle (1.5pt);
                    \draw [fill=ududff] (36.25542616776213,-7.5068478400775485) circle (1.5pt);
                    \draw [fill=ududff] (36.25542616776213,-4.506847840077543) circle (1.5pt);
                    \draw [fill=ududff] (34.75542616776213,-1.9087716287242227) circle (1.5pt);
                    \draw [fill=ududff] (32.15734995640881,-0.40877162872421824) circle (1.5pt);
                    \draw [fill=ududff] (29.157349956408805,-0.40877162872421735) circle (1.5pt);
                    \draw [fill=ududff] (26.55927374505548,-1.9087716287242191) circle (1.5pt);
                \end{scriptsize}
            \end{tikzpicture}
        \vspace{-2.5em}
        \caption{Omitting the edges $a_{{2i}}$ from the alternate arrangement of heptagons}
        \label{fig: Omitting edges}
    \end{figure}

    In $\mathcal{O}^{\mathcal{A}}_{n,k}$, after excluding the edges $a_{{2i}}$, for all $1\leq i \leq k$, the structure we get has an equal number of edges above and below the attached edges. Whereas, in $\mathcal{O}^{\mathcal{S}}_{n,k}$, after excluding the edges $a_{{2i}}$, the resulting structure has \textit{two} more edges above the attached edge than the edges below it. Hence, this case cannot be solved as a corollary to \Cref{thm: homotopy of 2n-gon}.
    
    We first complete the case when the polygons are attached alternatively.

    \begin{theorem}\label{Theorem 6}
         Let us consider the alternate arrangement of $2k-\left(2n+1\right)-$gon, that is, $\mathcal{O}^{\mathcal{A}}_{n,k}$, for $n\geq 2$ and $k \geq 2$. Then the perfect matching complex of $\mathcal{O}^{\mathcal{A}}_{n,k}$, that is, $\pma$ is contractible.
    \end{theorem}

    \begin{proof}
        From \Cref{lemma 5.1}, $\pma = \mathcal{M}_p\left(\mathcal{E}_{2n,k}\right)$. Hence, using \Cref{theorem 2}, $\pma$ is contractible.
    \end{proof}

    We now discuss the homotopy type of the perfect matching complex of the general polygon line tiling of triangles. We have already discussed that only an alternate arrangement of triangles is possible in a line tiling; thus, we can use the above lemma to get the desired result. Let $\Delta_k$ denote the general line tiling of $2k-$triangles, where $k\geq 2$, that is, $2k$ number of triangles are in the line tiling.

    \begin{theorem}\label{Theorem 7}
        Let us consider the general polygon line tiling of $2k-$triangles, that is,$\Delta_k$, $k\geq 2$. Let $\pmt$ denote the perfect matching complex of $\Delta_k$. If
        
        \begin{enumerate}[label = \roman*.]
            \item $k$ is even, then $\pmt$ is contractible;
            \item $k$ is odd, then $\pmt$ is homotopic to $\left\lfloor{\frac{k}{2}}\right\rfloor$-sphere.
        \end{enumerate}
    \end{theorem}

    \begin{proof}
        From \Cref{lemma 5.1}, $\pmt = \mathcal{M}_p\left(\mathcal{G}_{2\times \left(k+1\right)}\right)$. Hence, we get the desired result using \Cref{thm: 2n grid graph homotopy}.
    \end{proof}

    We now conclude by discussing the homotopy type of the perfect matching complex of $\mathcal{O}^{\mathcal{S}}_{n,k}$. 
    Let $\pms$ denote the perfect matching complex of $\mathcal{O}^{\mathcal{S}}_{n,k}$. 
    
    \begin{theorem}\label{Theorem 8}
        For $n\geq 2$ and $k \geq 2$, the perfect matching complex of $\mathcal{O}^{\mathcal{S}}_{n,k}$, $\pms$ is contractible.
    \end{theorem}

    \begin{proof}
        We perform a sequence of element pairings on $\mathcal{M}_p(\mathcal{O}^{\mathcal{S}}_{n,k})$ using the edges $a_{{1}}$, $b_{{1,1}}$ and $c_{{4,n-1}}$, respectively, of $\mathcal{O}^{\mathcal{S}}_{n,k}$. Ultimately, we are left with no critical cells, making our simplicial complex contractible. The proof follows similar steps as that of the proof of \Cref{thm: homotopy of 2n-gon} (precisely the case when $n$ was taken odd), with slight modification in notations.
    \end{proof}

\section{Concluding remarks and future directions}

We found the homotopy type of the perfect matching complex of $\twongg$ and polygon line tiling (for both even and odd-sided polygons). Our primary tool in determining this was discrete Morse theory and characterizing all the bad matchings of $\twongg$. 

We discussed how polygon line tiling can be considered a generalization of $\twongg$. Another observation drawn from $\twongg$ is that it is a member of the general $\left(m \times n\right)$-grid graph family, where $m,n\geq2$. Thus, we raise an obvious question about what we can say about the topology of the perfect matching complex of $\mathcal{G}_{m\times n}$. Motivated by this idea, we performed some calculations on SageMath (\cite{sagemath}), and based on the information we obtained, we conjecture the following.

\begin{conj}
    For all $m,n\ge 1$, the perfect matching complex of $\mathcal{G}_{m\times n}$ is homotopy equivalent to a wedge of spheres. 
\end{conj}

The analysis of the topology of the perfect matching complex could be quite interesting because one might be able to understand the interplay between the matching complexes and perfect matching complexes from it. For instance, due to \cite{Matsushita2018MatchingCO}, we have all the information about the homotopy type of the matching complex of $\twongg$. Thus, the next apparent aim could be to find a relationship between the matching complex and the perfect matching complex of $\twongg$.

\section*{Acknowledgements}
     We thank the anonymous referees for their insightful comments and suggestions, which have greatly enhanced the clarity and presentation of this article. Himanshu Chandrakar gratefully acknowledges the assistance provided by the Council of Scientific and Industrial Research (CSIR), India, through grant 09/1237(15675)/2022-EMR-I. Anurag Singh is partially supported by the Start-up Research Grant SRG/2022/000314 from SERB, DST, India.

\subsection*{Data availability} No data was gathered or used in this paper, so a “data availability statement” is not applicable.

\subsection*{Conflict of interest} The author states that there is no conflict of interest.

\bibliographystyle{abbrv}
\bibliography{new}
\addresseshere

\newpage
\section*{Appendix}

In the accepted version of our article, we omitted the original proof of \Cref{thm: 2n grid graph homotopy}, which uses element pairing, and retained the shorter proof suggested by the anonymous referee, which is based on the fold lemma of the independence complex.

The article remains identical to the version accepted by the journal. However, we are including our original proof in the appendix for reference.

\begin{proof}[Proof of \Cref{thm: 2n grid graph homotopy}]
We will perform the following element pairings to prove this result:

    \begin{itemize}
        \item \textbf{When $n$ is odd}, we will start element pairing using $a_{1}$, followed by $a_{3}$,$a_{5}$,$\dots$,$a_n$, respectively. In this case, we will have no elements left in the critical cell of the last element pairing, and thus, $\mathcal{M}_p \left(\mathcal{G}_{2\times n}\right)$ is contractible using \Cref{Corollary 1}.
        \item \textbf{When $n$ is even}, we will start element pairing using $a_{1}$, followed by $b_{1}$,$a_{3}$,$b_{3}$, $a_{5}$,$b_{5}$,$\dots$,$a_{n-1}$, $b_{n-1}$, respectively. In this case we will have the set $\left\{b_{1},b_{3},\dots,b_{n-1}\right\}$ as the only critical cell after last element pairing. Note that, the above set has exactly $k+1$ number of elements, where $n=2k+2$, and thus, $\mathcal{M}_p \left(\mathcal{G}_{2\times n}\right) \simeq \mathbb{S}^k$, using using \Cref{Corollary 1}.
    \end{itemize}

    \begin{case}
        \textbf{When $n$ is odd.} \label{2n odd case}
    \end{case}

    Let $n=2k+1$. For $0\leq m \leq k$, we will consecutively perform element pairing using $a_{{2m+1}}$ starting at $m=0$ and ending at $m=k$. We will use the induction method to identify critical cells after each pairing. Our claim for the set of critical cells is as follows.

    \begin{claim}\label{claim:grid n odd case}
        For $0\leq m \leq k-1$, we will get the following set of critical cells after element pairing with $a_{{2m+1}}$, 
        \[
        C_{a_{{2m+1}}} = \left\{ \sigma \in \pmcgg\ \middle|\begin{array}{c}
        a_{{2i+1}} \notin \sigma,\sigma \cup \left\{a_{{2i+1}}\right\} \notin \pmcgg,\\
        \text{for } i\in\left\{0,1,2,\dots,m\right\} \end{array}\right\}.\]
        Moreover, $C_{a_{{2k+1}}} = \emptyset$. If $\sigma \in C_{a_{{2m+1}}}$, for $1\leq m \leq k-1$ then $\sigma \cap \left\{b_{{2m+1}},c_{{2m+1}}\right\} \neq \emptyset$. As consequence of this and \Cref{Lemma 1}, we have $\sigma \cap \left\{b_{{2m+2}}, c_{{2m+2}}\right\} = \emptyset$.
    \end{claim}


    In this proof, we will encounter the argument that $\sigma \cup \left\{a_{{2m+1}}\right\}$ is not a matching or a bad matching, for some $\sigma \in \pmcgg$ repeatedly. We do not check the latter part because if $\sigma \cup \left\{a_{{2m+1}}\right\}$ is a bad matching, then it would mean it contains a subset of the form \hyperref[X1]{X1} or \hyperref[X2]{X2}. However, this would imply that $\sigma$ contains a subset of the form \hyperref[X1]{X1} or \hyperref[X2]{X2} contradicting $\sigma \in \pmcgg$. We start by proving this claim for $m=0$, our base case for the induction.
    
    \vspace{0.5cm}
    \noindent\textbf{Element pairing using $a_1$ (\textit{i.e.}, $a_{{2m+1}}$ when $m=0$):} Observe that, an element $\sigma$ is unpaired after pairing with $a_{1}$ if and only if $\sigma \cup \left\{a_{1}\right\} \notin \pmcgg$, where $a_{1} \notin \sigma$. Thus,
    \[C_{a_{1}} = \left\{ \sigma \in \pmcgg\ |\ a_{1} \notin \sigma,\sigma\cup \left\{a_{1}\right\} \notin \pmcgg \right\}\]
    
    Here, $\sigma \in \pmcgg$, $a_{1} \notin \sigma$ and $\sigma \cup \left\{a_{1}\right\} \notin \pmcgg$ implies that $\sigma$ is a \textit{matching} in $\mathcal{G}_{2\times n}$ without having $a_{1}$ in it, but as soon as we take $a_{1}$ with $\sigma$, it is not a matching anymore. Thus, $\sigma \cap \left\{b_{1},c_{1}\right\} \neq \emptyset$, implying $a_{2} \notin \sigma$. Furthermore, using \Cref{Lemma 1} and \Cref{Lemma 2}, we conclude that $\sigma \cap \left\{b_{2},c_{2}\right\} = \emptyset$.

     For the inductive step, assume $0 \leq t \leq k-1$ and for $0 \leq m \leq t$, $C_{a_{2m+1}}$ is as given in \Cref{claim:grid n odd case}. Then we show $C_{a_{2t+3}}$ is as given in \Cref{claim:grid n odd case}.

    \vspace{0.5cm}
    \noindent\textbf{Element pairing using $a_{{2t+3}}$ (\textit{i.e.}, $a_{{2m+1}}$ when $m=t+1$):}
    We define element pairing using $a_{{2t+3}}$ on the elements of $C_{a_{{2t+1}}}$. If an element $\sigma \in C_{a_{{2t+1}}}$is left unpaired after element pairing with $a_{{2t+3}}$, then one of the following two conditions must hold,
    


    \begin{enumerate}[label=\Alph*.]
        \item $a_{{2t+3}}\notin\sigma$ and $\sigma\cup\left\{a_{{2t+3}}\right\} \notin C_{a_{{2t+1}}}$, which means at least one of the following must hold,
        
        \begin{enumerate}[label = \roman*.]
            \item $\sigma\cup\left\{a_{{2t+3}}\right\} \notin \pmcgg$.
            
            This means $\sigma \cap \left\{b_{{2t+3}},c_{{2t+3}}\right\} \neq \emptyset$, \textit{i.e.}, either $b_{{2t+3}}$ or $c_{{2t+3}}$ is in $\sigma$, which can occur. Hence, this case is possible.
            
            \item $\sigma\cup\left\{a_{{2t+3}}\right\} \in \pmcgg$ and $\sigma\cup\left\{a_{{2t+3}}\right\} \notin C_{a_{1}}$.

            If $\sigma\cup\left\{a_{{2t+3}}\right\} \notin C_{a_{1}}$ then $\sigma \cup \left\{a_{1},a_{{2t+3}}\right\} \in \pmcgg$. This will contradict $\pmcgg$ being a simplicial complex as $\sigma \cup \left\{a_{1}\right\} \notin \pmcgg$ (since $\sigma \in C_{a_{1}}$). Hence, this case is not possible.
            
            \item $\sigma\cup\left\{a_{{2t+3}}\right\} \in C_{a_{{2j-1}}} \text{ and } \sigma\cup\left\{a_{{2t+3}}\right\} \notin C_{a_{{2j+1}}}$, for some $1\leq j\leq t$.

            If $\sigma\cup\left\{a_{{2t+3}}\right\} \notin C_{a_{{2j+1}}}$ then $\sigma \cup \left\{a_{{2l+1}},a_{{2j+1}}\right\} \in \pmcgg$, for at least one $l$, where $1\leq l\leq j-1$. This case will contradict $\pmcgg$ being a simplicial complex, as $\sigma \cup \left\{a_{{2l+1}}\right\} \notin \pmcgg$ (since $\sigma \in C_{a_{{2l+1}}}$). Hence, this case is not possible.
        \end{enumerate}

        \item $a_{{2t+3}}\in\sigma$ and $\sigma\setminus\left\{a_{{2t+3}}\right\} \notin C_{a_{{2t+1}}}$, which means at least one of the following must hold,

            \begin{enumerate}[label = \roman*.]
                \item $\sigma\setminus\left\{a_{{2t+3}}\right\} \notin \pmcgg$.
                
                This case will contradict $\pmcgg$ being a simplicial complex since $\sigma \in \pmcgg$. Hence, this case is not possible.
                
                \item $\sigma\setminus\left\{a_{{2t+3}}\right\} \in \pmcgg \text{ and } \sigma\setminus\left\{a_{{2t+3}}\right\} \notin C_{a_{1}}$.

                If $\sigma\setminus\left\{a_{{2t+3}}\right\} \notin C_{a_{1}}$, then $\sigma\cup\left\{a_{1}\right\}\setminus\left\{a_{{2t+3}}\right\} \in \pmcgg$. This is not true as $\sigma \cap \left\{b_{1},c_{1}\right\} \neq \emptyset$ (since $\sigma \in C_{a_{1}})$. Hence, this case is not possible.

                \item $\sigma\setminus\left\{a_{{2t+3}}\right\} \in C_{a_{{2j-1}}}$ and $\sigma\setminus\left\{a_{{2t+3}}\right\} \notin C_{a_{{2j+1}}}$, for some $1\leq j\leq t$.

                If $\sigma\setminus\left\{a_{{2t+3}}\right\} \notin C_{a_{{2j+1}}}$, then $\sigma \cup \left\{a_{{2l+1}}\right\}\setminus \left\{a_{{2j+1}}\right\} \in \pmcgg$, for at least one $l$, where $1\leq l\leq j-1$. A similar argument we used in the previous part will discard this case, where instead of $a_{1},b_{1},c_{1}$, we use $a_{{2l+1}},b_{{2l+1}},c_{{2l+1}}$, respectively. Hence, this case is not possible.
            \end{enumerate}
    
    \end{enumerate}

    Thus, 
    $$ C_{a_{{2t+3}}} = \left\{ \sigma \in \pmcgg\ \middle|\begin{array}{c}
    a_{{2i+1}} \notin \sigma,\sigma \cup \left\{a_{{2i+1}}\right\} \notin \pmcgg,\\
    \text{for } i\in\left\{0,1,2,\dots,t+1\right\} \end{array}\right\}.$$

    Using a similar argument we used in the base case, we can conclude that $\sigma \cap \left\{b_{{2t+3}},c_{{2t+3}}\right\} \neq \emptyset$, implying $\sigma \cap \left\{b_{{2t+4}},c_{{2t+4}}\right\} = \emptyset$.
    
    Now consider $C_{a_{n}}$. Note that $\sigma \cup \left\{a_{n}\right\} \notin \pmcgg$ implies that $\sigma \cap \left\{b_{n-1},c_{n-1}\right\} \neq \emptyset$ which is a contradiction since $\sigma \in C_{a_{n-2}}$, \textit{i.e.}, there is no $\sigma\in\pmcgg$ satisfying the above conditions. Therefore, using \Cref{Corollary 1}, we get
        \[C_{a_n} = \emptyset.\]
    Thus, $\pmcgg$ is contractible.

    \begin{case}
        \textbf{When $n$ is even.} \label{2n grid even case}
    \end{case}

    Let $n=2k+2$. Again, we use induction to prove this.

    \begin{claim}\label{claim: 2n grid even induction}
    For $0\leq m \leq k$, we consecutively perform element pairing using $a_{{2m+1}}$ followed by $b_{{2m+1}}$ and obtain the following set of critical cells:
    \begin{enumerate}
        \item After element pairing with $a_{{2m+1}}$, we get the following set of critical cells,

        $$ C_{a_{{2m+1}}} = \left\{ \sigma \in \pmcgg\ \middle|\begin{array}{c}
        a_{{i}} \notin \sigma,b_{{j}} \in \sigma,\sigma \cup \left\{a_{{i}}\right\} \notin \pmcgg,\\
        \sigma \cup \left\{a_{{j}}\right\} \setminus \left\{b_{{j}}\right\}\in \pmcgg,\\
        \text{for all } i\in\left\{1,3,\dots,2m+1\right\},\\ 
        \text{for all } j\in\left\{1,3,\dots,2m-1\right\}\end{array}\right\}.$$

        If $\sigma \in C_{a_{{2m+1}}}$ then it will satisfy the similar properties defined for the case when $n$ was odd along with the properties of critical cells in $C_{b_{{2m-1}}}$, discussed in the following point.

        \item After element pairing with $b_{{2m+1}}$, we'll get the following set of critical cells,
        $$ C_{b_{{2m+1}}} = C_{a_{{2m+1}}} \cap \left\{ \sigma \in \pmcgg\ \middle|\ \begin{array}{c}
        b_{{2m+1}} \in \sigma,\\
        \sigma \cup \left\{a_{{2m+1}}\right\} \setminus \left\{b_{{2m+1}}\right\}\in \pmcgg 
        \end{array}\right\}; 
        $$

        If $\sigma \in C_{b_{{2m+1}}}$ then $b_{{2m+1}} \in \sigma$; and due to $\sigma \cup \left\{a_{{2m+1}}\right\} \setminus \left\{b_{{2m+1}}\right\}\in \pmcgg$, it follows that $c_{{2m+1}} \notin \sigma$. It will also satisfy the properties of $C_{a_{{2m+1}}}$. 
    \end{enumerate}
    \end{claim}

     Before defining element pairing, note that we will not analyse the case if $\sigma \cup \left\{a_{{2m+1}}\right\}$ is a bad matching, and the reason for it is the same as given for the case when $n$ was odd. However, we will analyse the case if $\sigma \cup \left\{b_{{2m+1}}\right\}$ is a bad matching. For the base case, we prove the claim for $m=0$.
    \vspace{0.5cm}

    \noindent\textbf{Element pairing using $a_{{1}}$ and $b_{{1}}$ (\textit{i.e.}, $a_{{2m+1}}$ and $b_{{2m+1}}$ when $m=0$):}
    
    The element pairing defined for $a_{1}$ is similar to the case when $n$ is odd. Therefore, we directly define element pairing using $b_{1}$ on the elements of $C_{a_{1}}$. If an element $\sigma \in C_{a_{1}}$ is left unpaired after pairing with $b_{1}$, then one of the following two conditions must hold,
    
    \begin{enumerate}[label = \Alph*.]
        \item $b_{1}\notin\sigma$ and $\sigma\cup\left\{b_{1}\right\} \notin C_{a_{1}}$, which means either of the following:
        
        \begin{enumerate}[label = \roman*.]
            \item $\sigma\cup\left\{b_{1}\right\} \notin \pmcgg$.
            
            If $\sigma\cup\left\{b_{1}\right\}$ is not a matching then $\sigma \cap \left\{a_{1},a_{2},b_{2} \right\} \neq \emptyset$. From $\sigma \cup \left\{a_{1}\right\} \notin \pmcgg$, we have already concluded that $\sigma \cap \left\{b_{1},c_{1} \right\} \neq \emptyset$. Since, $b_{1} \notin \sigma$, this would imply that $c_{1} \in \sigma$. This means that $\sigma \cap \left\{a_{1},a_{2},c_{2}\right\} = \emptyset$. Furthermore, using \Cref{Lemma 1}, $b_{2}\notin \sigma$. Hence, this case is not possible. If it is a bad matching then $\sigma \cap \left\{a_{1},a_{2},b_{2} \right\} = \emptyset$. Using $\sigma \in \pmcgg$ and \Cref{Lemma 1}, $\sigma \cup \left\{b_{1}\right\}$ must contain a subset of the form \hyperref[X1]{X1}, \textit{i.e.}, $c_{2} \in \sigma$. This would imply $c_{1} \notin \sigma$ which contradicts $\sigma \cap \left\{b_{1},c_{1} \right\} \neq \emptyset$. Hence, this case is not possible.
            
            \item $\sigma\cup\left\{b_{1}\right\} \in \pmcgg \text{ and } \sigma\cup\left\{b_{1}\right\} \notin C_{a_{1}}$.

            If $\sigma\cup\left\{b_{1}\right\} \notin C_{a_{1}}$ then $\sigma \cup \left\{a_{1},b_{1}\right\} \in \pmcgg$. This will contradict $\pmcgg$ being a simplicial complex as $\sigma \cup \left\{a_{1}\right\} \notin \pmcgg$ (since $\sigma \in C_{a_{1}}$). Hence, this case is not possible. 
        \end{enumerate}

        \item $b_{1}\in\sigma$ and $\sigma\setminus\left\{b_{1}\right\} \notin C_{a_{1}}$.

        If $\sigma\setminus\left\{b_{1}\right\} \notin C_{a_{1}}$, then either $\sigma \setminus\left\{b_{1}\right\} \notin \pmcgg$ or $\sigma \setminus\left\{b_{1}\right\} \in \pmcgg$ and $\sigma \setminus\left\{b_{1}\right\} \notin C_{a_{1}}$. The former case contradicts $\pmcgg$ being a simplicial complex since $\sigma \in \pmcgg$. The later case implies $\sigma \cup \left\{a_{1}\right\} \setminus\left\{b_{1}\right\} \in \pmcgg$. We know that, $\sigma \cap \left\{b_{1},c_{1}\right\} \neq \emptyset$. Since $b_{1}\in \sigma$ and after removing $b_{1}$ from $\sigma\cup\left\{a_{1}\right\}$, it is still in $\pmcgg$, we conclude that $c_{1} \notin \sigma$, which can occur. Hence, this case is possible.
    \end{enumerate} 

    Thus,
    \begin{align*}
        C_{b_{1}} = \left\{ \sigma \in \pmcgg\ \middle|\begin{array}{c}
        a_{1} \notin \sigma,b_{1} \in \sigma,\sigma\cup \left\{a_{1}\right\} \notin \pmcgg,\\
        \sigma\cup\left\{a_{1}\right\}\setminus\left\{b_{1}\right\} \in \pmcgg 
        \end{array}\right\}
    \end{align*}
    
    Thus, if $\sigma \in C_{b_{1}}$, then $a_{1} \notin \sigma$; $b_{1} \in \sigma$ implying $a_{2},b_{2} \notin \sigma$ and due to $\sigma\cup\left\{a_{1}\right\}\setminus\left\{b_{1}\right\} \in \pmcgg$, we conclude that $c_{1} \notin \sigma$.

    For the inductive step, assume $0\leq t \leq k-1$, and for $0 \leq m \leq t$, $C_{a_{{2m+1}}}$ and $C_{b_{{2m+1}}}$ is as given in \Cref{claim: 2n grid even induction}. Then we show that $C_{a_{{2t+3}}}$ and $C_{b_{{2t+3}}}$ is as given in \Cref{claim: 2n grid even induction}.
    \vspace{0.5cm}
    
    \noindent\textbf{Element pairing using $a_{{2t+3}}$ and $b_{{2t+3}}$ (\textit{i.e.}, $a_{{2m+1}}$ and $b_{{2m+1}}$ when $m=t+1$):}

    We now define element pairing using $a_{{2t+3}}$ on the elements of $C_{b_{{2t+1}}}$. If an element $\sigma \in C_{b_{{2t+1}}}$ is left unpaired after pairing with $a_{{2t+3}}$, then at least one of the following two conditions must hold:
    

    \begin{enumerate}[label = \Alph*.]
        \item $a_{{2t+3}}\notin\sigma$ and $\sigma\cup\left\{a_{{2t+3}}\right\} \notin C_{b_{{2t+1}}}$, which means at least one of the following:
            \begin{enumerate}[label = \roman*.]
                \item $\sigma\cup\left\{a_{{2t+3}}\right\} \notin \pmcgg$.
                
                This means, $\sigma \cap \left\{ b_{{2t+3}},c_{{2t+3}}\right\} \neq \emptyset$, \textit{i.e.}, either $b_{{2t+3}}$ or $c_{{2t+3}}$ is in $\sigma$, which can occur. Hence, this case is possible.
                
                \item $\sigma\cup\left\{a_{{2t+3}}\right\} \in \pmcgg \text{ and } \sigma\cup\left\{a_{{2t+3}}\right\} \notin C_{a_{1}}$.

                If $\sigma\cup\left\{a_{{2t+3}}\right\} \notin C_{a_{1}}$ then $\sigma \cup \left\{a_{1},a_{{2t+3}}\right\} \in \pmcgg$. This will contradict $\pmcgg$ being a simplicial complex as $\sigma \cup \left\{a_{1}\right\} \notin \pmcgg$ (since $\sigma \in C_{a_{1}}$). Hence, this case is not possible.

                \item $\sigma\cup\left\{a_{{2t+3}}\right\} \in C_{a_{{2j_{1}-1}}} \text{ and } \sigma\cup\left\{a_{{2t+3}}\right\} \notin C_{b_{{2j_{1}-1}}}$, for some $1\leq j_{1} \leq t+1$.

                If $\sigma\cup\left\{a_{{2t+3}}\right\} \notin C_{b_{{2j_{1}-1}}}$ then either $\sigma \cup \left\{a_{{2j_{1}-1}},a_{{2t+3}}\right\} \in \pmcgg$ or $\sigma \cup \left\{a_{{2j_{1}-1}},a_{{2t+3}}\right\}\setminus \left\{b_{{2j_{1}-1}}\right\}\notin \pmcgg$. The case, $\sigma \cup \left\{a_{{2j_{1}-1}},a_{{2t+3}}\right\} \in \pmcgg$ is not possible because it will contradict $\pmcgg$ being a simplicial complex as $\sigma \cup \left\{a_{{2j_{1}-1}}\right\} \notin \pmcgg$ (since $\sigma \in C_{a_{{2j_{1}-1}}}$). For $\sigma \cup \left\{a_{{2j_{1}-1}},a_{{2t+3}}\right\}\setminus \left\{b_{{2j_{1}-1}}\right\}\notin \pmcgg$, would imply $\sigma \cap \left\{c_{{2t+3}},b_{{2t+3}}\right\} \neq \emptyset$. However, from our assumption $\sigma\cup\left\{a_{{2t+3}}\right\} \in C_{a_{{2j_{1}-1}}} \subset \pmcgg$ implying $\sigma \cap \left\{c_{{2t+3}},b_{{2t+3}}\right\} = \emptyset$ leading to a contradiction. Hence, this case is not possible.
                
                \item $\sigma\cup\left\{a_{{2t+3}}\right\} \in C_{b_{{2j_{2}-1}}} \text{ and } \sigma\cup\left\{a_{{2t+3}}\right\} \notin C_{a_{{2j_{2}+1}}}$, for some $1\leq j_{2} \leq t$. 
                
                This would imply that $\sigma \cup \left\{a_{{2j_{2}+1}},a_{{2t+3}}\right\} \in \pmcgg$ contradicting $\pmcgg$ being a simplicial complex as $\sigma \cup \left\{a_{{2j_{2}+1}}\right\} \notin \pmcgg$ (since $\sigma \in C_{a_{1}}$). Hence, this case is not possible.
            \end{enumerate}

        \item $a_{{2t+3}}\in\sigma$ and $\sigma\setminus\left\{a_{{2t+3}}\right\} \notin C_{b_{{2t+1}}}$, which means at least one of the following:

          \begin{enumerate}[label = \roman*.]
                \item $\sigma\setminus\left\{a_{{2t+3}}\right\} \notin \pmcgg$.
                
                This case contradicts $\pmcgg$ being a simplicial complex as $\sigma \in \pmcgg$. Hence, this case is not possible.
                
                \item $\sigma\setminus\left\{a_{{2t+3}}\right\} \in \pmcgg \text{ and } \sigma\setminus\left\{a_{{2t+3}}\right\} \notin C_{a_{1}}$.

                This would imply that $\sigma \cup \left\{a_{1}\right\} \setminus \left\{a_{{2t+3}}\right\} \in \pmcgg$. This is not true as $b_{1} \in \sigma$ (since $\sigma \in C_{b_{1}}$). Hence, this case is not possible.

                \item $\sigma\setminus\left\{a_{{2t+3}}\right\} \in C_{a_{{2j_{1}-1}}} \text{ and } \sigma\setminus\left\{a_{{2t+3}}\right\} \notin C_{b_{{2j_{1}-1}}}$, for some $1\leq j_{1} \leq t+1$.

                If $\sigma\setminus\left\{a_{{2t+3}}\right\} \notin C_{b_{{2j_{1}-1}}}$ then $\sigma\cup\left\{a_{{2j_{1}-1}}\right\}\setminus\left\{b_{{2j_{1}-1}},a_{{2t+3}}\right\}\notin\pmcgg$. However, this contradicts $\pmcgg$ being a simplicial complex since $\sigma\cup\left\{a_{{2j_{1}-1}}\right\}\setminus\left\{b_{{2j_{1}-1}}\right\} \in\pmcgg$ (since $\sigma \in C_{b_{{2j_{1}-1}}}$). Hence, this case is not possible.

                \item $\sigma\setminus\left\{a_{{2t+3}}\right\} \in C_{b_{{2j_{2}-1}}} \text{ and } \sigma\setminus\left\{a_{{2t+3}}\right\} \notin C_{a_{{2j_{2}+1}}}$, for some $1\leq j_{2} \leq t$.

                If $\sigma\setminus\left\{a_{{2t+3}}\right\} \notin C_{a_{{2j_{2}+1}}}$ then $\sigma\cup\left\{a_{{2j_{2}+1}}\right\}\setminus\left\{a_{{2t+3}}\right\}\in\pmcgg$. This is not true as $b_{{2j_{2}+1}} \in \sigma$ (since $\sigma \in C_{b_{{2j_{2}+1}}}$). Hence, this case is not possible.
            \end{enumerate}
    \end{enumerate}

    Thus,
    \begin{align*}
        C_{a_{{2t+3}}} = \left\{ \sigma \in \pmcgg\ \middle|\begin{array}{c}
        a_{{i}} \notin \sigma,b_{{j}} \in \sigma,\sigma \cup \left\{a_{{i}}\right\} \notin \pmcgg,\\
        \sigma \cup \left\{a_{{j}}\right\} \setminus \left\{b_{{j}}\right\}\in \pmcgg,\\
        \text{for all } i\in\left\{1,3,\dots,2t+3\right\},\\ 
        \text{for all } j\in\left\{1,3,\dots,2t-1\right\}\end{array}\right\}.
    \end{align*}

     Observe that, $\sigma$ satisfies all the properties of $C_{b_{{2t+1}}}$ since, $\sigma \in C_{a_{{2t+3}}} \subset C_{b_{{2t+1}}}$. Also, $a_{{2t+3}} \notin \sigma$ and $\sigma \cup \left\{a_{{2t+3}}\right\} \notin \pmcgg$ implies that $\sigma \cap \left\{ b_{{2t+3}},c_{{2t+3}}\right\} \neq \emptyset$, implying $a_{{2t+4}} \notin \sigma$. Furthermore, using \Cref{Lemma 1} and \Cref{Lemma 2}, we conclude that $\sigma \cap \left\{b_{{2t+4}},c_{{2t+4}}\right\} = \emptyset$.

     At last, we define element pairing using $b_{{2t+3}}$ on the elements of $C_{a_{{2t+3}}}$. If an element $\sigma \in C_{a_{{2t+3}}}$ is left unpaired after pairing with $b_{{2t+3}}$ then one of the following two conditions must hold:
    

    \begin{enumerate}[label = \Alph*.]
        \item $b_{{2t+3}} \notin \sigma$ and $\sigma\cup\left\{b_{{2t+3}}\right\} \notin C_{a_{{2t+3}}}$, which means at least one of the following:
            
        \begin{enumerate}[label = \roman*.]
            \item $\sigma\cup\left\{b_{{2t+3}}\right\} \notin \pmcgg$.
                
            \item $\sigma\cup\left\{b_{{2t+3}}\right\} \in \pmcgg \text{ and } \sigma\cup\left\{b_{{2t+3}}\right\} \notin C_{a_{1}}$.

            These two cases are similar to those we discussed while performing element pairing with $b_{1}$; replace the index $m=0$ with $m=t+1$.
            \vspace{1em}
            \item $\sigma\cup\left\{b_{{2t+3}}\right\} \in C_{a_{{2j_{1}-1}}} \text{ and } \sigma\cup\left\{b_{{2t+3}}\right\} \notin C_{b_{{2j_{1}-1}}}$, for some $1\leq j_{1} \leq t+1$.

            If $\sigma\cup\left\{b_{{2t+3}}\right\} \notin C_{b_{{2j_{1}-1}}}$, then $\sigma\cup\left\{a_{{2j_{1}-1}},b_{{2t+3}}\right\} \setminus \left\{b_{{2j_{1}-1}}\right\} \notin \pmcgg$. If it is not a matching then, $\sigma \cap \left\{a_{{2t+3}},a_{{2t+4}},b_{{2t+2}},b_{{2t+4}}\right\} \neq \emptyset$, which is a contradiction since $\sigma\cup\left\{b_{{2t+3}}\right\} \in C_{a_{{2j_{1}-1}}} \subset \pmcgg$ and thus
            none of these edges (\textit{i.e.}, $a_{{2t+3}},a_{{2t+4}},b_{{2t+2}},b_{{2t+4}}$) are in $\sigma$. If it is a bad matching, then using \Cref{Lemma 1}, $\sigma \cap \left\{ c_{{2t+2}},c_{{2t+4}}\right\} \neq \emptyset$. We know that $b_{{2t+3}} \notin \sigma$ and $\sigma \cap \left\{b_{{2t+3}}, c_{{2t+3}}\right\} \neq \emptyset$ (since $\sigma \in C_{a_{{2t+3}}}$), implying $c_{{2t+3}} \in \sigma$. This contradicts $\sigma \cap \left\{c_{{2t+2}},c_{{2t+4}}\right\} \neq \emptyset$. Hence, this case is not possible.

            \item $\sigma\cup\left\{b_{{2t+3}}\right\} \in C_{b_{{2j_{2}-1}}} \text{ and } \sigma\cup\left\{b_{{2t+3}}\right\} \notin C_{a_{{2j_{2}+1}}}$, for some $1\leq j_{2} \leq t+1$.

            If $\sigma\cup\left\{b_{{2t+3}}\right\} \notin C_{a_{{2j_{2}+1}}}$, then $\sigma \cup \left\{a_{{2j_{2}+1}}, b_{{2t+3}}\right\} \in \pmcgg$. This contradicts $\pmcgg$ being a simplicial complex since $\sigma \cup \left\{a_{{2j_{2}+1}}\right\} \notin \pmcgg$. Hence, this case is not possible.
        \end{enumerate}

        \item $b_{{2t+3}}\in\sigma$ and $\sigma\setminus\left\{b_{{2t+3}}\right\} \notin C_{a_{{2t+3}}}$, which means at least one of the following:

        \begin{enumerate}[label = \roman*.]
            \item $\sigma\setminus\left\{b_{{2t+3}}\right\} \notin \pmcgg$.
            
            This case contradicts $\pmcgg$ being a simplicial complex as $\sigma \in \pmcgg$. Hence, this case is not possible.
                
            \item $\sigma\setminus\left\{b_{{2t+3}}\right\} \in \pmcgg \text{ and } \sigma\setminus\left\{b_{{2t+3}}\right\} \notin C_{a_{1}}$.

            This would imply that $\sigma \cup \left\{a_{1}\right\} \setminus \left\{b_{{2t+3}}\right\} \in \pmcgg$. This is not true as $b_{1} \in \sigma$ (since $\sigma \in C_{b_{1}}$). Hence, this case is not possible.

            \item $\sigma\setminus\left\{b_{{2t+3}}\right\} \in C_{a_{{2j_{1}-1}}} \text{ and } \sigma\setminus\left\{b_{{2t+3}}\right\} \notin C_{b_{{2j_{1}-1}}}$, for some $1\leq j_{1} \leq t+1$.

            If $\sigma\setminus\left\{b_{{2t+3}}\right\} \notin C_{b_{{2j_{1}-1}}}$ then $\sigma\cup\left\{a_{{2j_{1}-1}}\right\}\setminus\left\{b_{{2j_{1}-1}},b_{{2t+3}}\right\}\notin\pmcgg$. However, this contradicts $\pmcgg$ being a simplicial complex since $\sigma\cup\left\{a_{{2j_{1}-1}}\right\}\setminus\left\{b_{{2j_{1}-1}}\right\} \in\pmcgg$ (since $\sigma \in C_{b_{{2j_{1}-1}}}$). Hence, this case is not possible.

            \item $\sigma\setminus\left\{b_{{2t+3}}\right\} \in C_{b_{{2j_{2}-1}}} \text{ and } \sigma\setminus\left\{b_{{2t+3}}\right\} \notin C_{a_{{2j_{2}+1}}}$, for some $1\leq j_{2} \leq t$.
            
            This would imply that $\sigma \cup \left\{a_{{2j_{2}+1}}\right\} \setminus \left\{b_{{2t+3}}\right\} \in \pmcgg$. This is not true as $b_{{2j_{2}+1}} \in \sigma$ (since $\sigma \in C_{b_{{2j_{2}+1}}}$). Hence, this case is not possible.

            \item $\sigma\setminus\left\{b_{{2t+3}}\right\} \in C_{b_{{2t+1}}} \text{ and } \sigma\setminus\left\{b_{{2t+3}}\right\} \notin C_{a_{{2t+3}}}$.

            If $\sigma\setminus\left\{b_{{2t+3}}\right\} \notin C_{a_{{2t+3}}}$ then $\sigma \cup\left\{a_{{2t+3}}\right\} \setminus \left\{b_{{2t+3}}\right\} \in \pmcgg$. We know that, $\sigma \cap \left\{b_{{2t+3}},c_{{2t+3}}\right\} \neq \emptyset$. Since, $b_{{2t+3}} \in \sigma$ and after removing it from $\sigma \cup\left\{a_{{2t+3}}\right\}$, it is still in $\pmcgg$, we conclude that $c_{{2t+3}} \notin \sigma$, which can occur. Hence, this case is possible.
        \end{enumerate}
    \end{enumerate}

     Thus,
    \begin{align*}
        C_{b_{{2t+3}}} = C_{a_{{2t+3}}} \cap \left\{ \sigma \in \pmcgg\ \middle|\ \begin{array}{c}
        b_{{2t+3}} \in \sigma,\\
        \sigma \cup \left\{a_{{2t+3}}\right\} \setminus \left\{b_{{2t+3}}\right\}\in \pmcgg 
        \end{array}\right\}.
    \end{align*}

    Thus, if $\sigma \in C_{b_{{2t+3}}}$, it satisfies all the conditions of $C_{a_{{2t+3}}}$. Also, due to $\sigma\cup\left\{a_{{2t+3}}\right\}\setminus\left\{b_{{2t+3}}\right\} \in \pmcgg$, we conclude that $c_{{2t+3}} \notin \sigma$. This completes our induction step.

     If $\sigma\in C_{b_{n-1}}$, then $a_{i} \notin \sigma$ for $i\in\left\{1,2,3,\dots,n\right\}$; $b_{1}$, $b_{3}$, $b_{5}$, $\dots$, $b_{{n-1}} \in \sigma$ which implies that $b_{l},c_{l}\notin \sigma$ for $l\in\left\{2,4,6,\dots,n\right\}$. Furthermore, since $\sigma\cup\left\{a_{i}\right\}\setminus\left\{b_{i}\right\} \in \pmcgg$ for $i\in\left\{1,3,5,\dots,n-1\right\}$, we conclude that $c_{i} \notin \pmcgg$. Thus,
        \[C_{b_{n-1}} = \left\{\left\{b_{1},b_{3},\dots,b_{n-1}\right\}\right\}.\]
    Note that, the cell contained in $C_{b_{n-1}}$ has cardinality $k+1$, where $n=2k+2$. Therefore, using \Cref{Corollary 1}, we get, $$\pmcgg\simeq \mathbb{S}^k.$$
This completes the proof.
\end{proof}

\end{document}